\documentclass[reqno]{amsart}

\usepackage{tikz}
\usepackage{graphicx}
\usepackage{amssymb}
\usepackage{hyperref}
\usepackage{cleveref}
\usepackage{subfigure}
\usepackage[a4paper, total={15cm, 23cm},centering]{geometry}

\usepackage{pstricks,tikz}

\usepackage{xcolor}
\hypersetup{
    colorlinks,
    linkcolor={blue!50!black},
    citecolor={darkgray!},
    urlcolor={Gray!}
}

\usepackage{mathrsfs}
\usepackage{float}
\usepackage[T1]{fontenc} 

\newtheorem{thm}{Theorem}[section]
\newtheorem{lem}[thm]{Lemma}

\newtheorem{prop}[thm]{Proposition}
\newtheorem{cor}[thm]{Corollary}

\newtheorem{conj}[thm]{Conjecture}
 \newtheorem{thm1}{Theorem}
 \newtheorem{cor1}[thm1]{Corollary}
 \newtheorem{conj1}{Conjecture}

\theoremstyle{definition}
\newtheorem{defn}[thm]{Definition}

\newtheorem{exa}[thm]{Example}

\newtheorem{alg}[thm]{Algorithm}

\theoremstyle{remark}
\newtheorem{rem}[thm]{Remark}

\def\Gmin{\widetilde{S}}
\def\Phisph{\Phi_{\mathrm{sph}}}
\def\Preg{\mathscr{P}_{\mathrm{reg}}(W)}
\def\intT{\mathrm{Int}(T)}

\def\scrC{\mathscr{C}}
\def\scrH{\mathscr{H}}
\def\scrD{\mathscr{D}}

\def\scrS{\mathscr{S}}
\def\scrG{\mathscr{G}}
\def\scrR{\mathscr{R}}

\def\scrP{\mathscr{P}}

\def\scrH{\mathscr{H}}

\def\scrT{\mathscr{T}}
\def\scrX{\mathscr{X}}

\def\scrJ{\mathscr{J}}
\def\scrQ{\mathscr{Q}}

\def\cA{\mathcal{A}}

\def\cE{\mathcal{E}}

\def\cL{\mathcal{L}}

\def\cS{\mathcal{S}}

\def\sA{\mathsf{A}}
\def\sB{\mathsf{B}}
\def\sG{\mathsf{G}}
\def\sX{\mathsf{X}}
\def\sY{\mathsf{Y}}
\def\sZ{\mathsf{Z}}
\def\sC{\mathsf{C}}
\def\sD{\mathsf{D}}
\def\sF{\mathsf{F}}

\def\peq{\preccurlyeq}

\definecolor{ForestGreen}{cmyk}{0.91,0,0.88,0.42}

\numberwithin{equation}{section}

\makeatletter
\@namedef{subjclassname@2020}{%
  \textup{2020} Mathematics Subject Classification}
\makeatother

\begin{document}

\title{Cone Types, Automata, and Regular Partitions in Coxeter Groups}

\author{James Parkinson}
\address{School of Mathematics and Statistics, University of Sydney, NSW, Australia}
\email{jamesp@maths.usyd.edu.au}

\author{Yeeka Yau}
\address{School of Mathematics and Statistics, University of Sydney, NSW, Australia}
\email{y.yau@maths.usyd.edu.au}


\date{\today}



\begin{abstract}
In this article we introduce the notion of a \textit{regular partition} of a Coxeter group. We develop the theory of these partitions, and show that the class of regular partitions is essentially equivalent to the class of automata (not necessarily finite state) recognising the language of reduced words in the Coxeter group. As an application of this theory we prove that each cone type in a Coxeter group has a unique minimal length representative. This result can be seen as an analogue of Shi's classical result that each component of the Shi arrangement of an affine Coxeter group has a unique minimal length element. We further develop the theory of cone types in Coxeter groups by identifying the minimal set of roots required to express a cone type as an intersection of half-spaces. This set of \textit{boundary roots} is closely related to the elementary inversion sets of Brink and Howlett, and also to the notion of the base of an inversion set introduced by Dyer. 
\end{abstract}

\maketitle


\section*{Introduction}

In \cite{BH:93}, Brink and Howlett showed that every finitely generated Coxeter system $(W,S)$ is automatic by providing an explicit construction of a finite state automaton $\cA_0$ recognising the language $\cL(W,S)$ of reduced words of $(W,S)$. A key insight in~\cite{BH:93} was the introduction of the set $\cE$ of \textit{elementary roots} of the associated root system, and the proof that $\cE$ is finite for all finitely generated Coxeter systems. This remarkable work paved the way for the study of other structures in Coxeter groups which also induce automata, such as the notion of \textit{Garside shadows} and \textit{low elements} introduced by Dehornoy, Dyer, and Hohlweg (see \cite{DDH:15,DH:16}). 

By the Myhill-Nerode Theorem there exists a unique (up to isomorphism) \textit{minimal} (with respect to the number of states) automaton $\cA(W,S)$ recognising $\cL(W,S)$. This automaton has been of considerable interest recently. For example, in \cite[Conjecture 2]{HNW:16} Hohlweg, Nadeau and Williams conjectured necessary and sufficient conditions for minimality of the automaton $\cA_0$ constructed by Brink and Howlett, and this conjecture was verified by the current authors in~\cite[Theorem~1]{PY:19}. Furthermore, in \cite{HNW:16}, an automaton recognising $\cL(W,S)$ is constructed using the smallest Garside shadow, and it is conjectured in \cite[Conjecture 1]{HNW:16} that this automaton is always minimal. 

In this paper we provide a detailed investigation of the minimal automaton $\cA(W,S)$. The set of accept states of this automaton is the set $\mathbb{T}$ of all \textit{cone types} in $W$, where the cone type of $x\in W$ is 
$$
T(x)=\{y\in W\mid \ell(xy)=\ell(x)+\ell(y)\}.
$$
Here $\ell:W\to\mathbb{N}$ is the usual length function, and we note that $|\mathbb{T}|<\infty$ (this is a consequence of the fact that $|\cE|<\infty$). Thus studying $\cA(W,S)$ amounts to studying cone types in Coxeter groups. 

A main contribution of this paper is the introduction of the notion of a \textit{regular partition} of~$W$. This concept is inspired by, and simultaneously generalises, both the Shi arrangement and the theory of Garside shadows, and we show that the class of regular partitions of $W$ is essentially equivalent to the class of automata recognising $\cL(W,S)$ (see Theorem~\ref{thm:main3}).

We will outline our main results below. In order to do so, we fix the following terminology (see Section~\ref{sec:1} for precise definitions). Let $(W,S)$ be a finitely generated Coxeter system, and let $\Phi$ be an associated root system. For $w\in W$ write $\Phi(w)$ for the inversion set of $w$, and $\cE(w)=\Phi(w)\cap\cE$ for the elementary inversion set of~$w$. The \textit{right weak order} on $W$ is defined by $x\peq y$ if and only if $\ell(x^{-1}y)=\ell(y)-\ell(x)$ (equivalently, if and only if $x$ is a prefix of $y$). The left descent set of $w\in W$ is $D_L(w)=\{s\in S\mid \ell(sw)=\ell(w)-1\}$. 

Recall, from \cite{DDH:15}, that a \textit{Garside shadow} in $W$ is a subset $B\subseteq W$ such that (i) $S\subseteq B$, (ii) $B$ is closed under taking suffixes of elements, and (iii) $B$ is closed under taking joins (in the right weak order) of bounded subsets. Since the intersection of Garside shadows is again a Garside shadow, there exists a unique smallest Garside shadow, denoted~$\Gmin$, and $|\Gmin|<\infty$.

One of the main contributions of this paper is the proof of the following fundamental property of cone types in Coxeter groups (see Corollary~\ref{cor:gateexist}).

\begin{thm1}\label{thm:main1}
For each cone type $T$ there is a unique element $m_T\in W$ of minimal length such that $T(m_T)=T$. Moreover, if $w\in W$ with $T(w)=T$ then $m_T$ is a suffix of $w$.
\end{thm1}


The path to proving Theorem~\ref{thm:main1} is surprisingly circuitous, and along the way we introduce several new concepts, as outlined below. An initial observation is that while the set $\{w\in W\mid T(w)=T\}$ of all cone type representatives of $T$ is disconnected (in the Coxeter complex), the inverse of this set turns out to be connected, and indeed convex (see Proposition~\ref{prop:convexitybasics}). Thus we consider the sets 
$$
X_T=\{w\in W\mid T(w^{-1})=T\},\quad\text{for $T\in\mathbb{T}$}.
$$
We will ultimately prove Theorem~\ref{thm:main1} by showing that $X_T$ contains a unique minimal length element $g_T$, and that whenever $w\in W$ is such that $T(w^{-1})=T$ then $g_T\peq w$. We call $\Gamma=\{g_T\mid T\in\mathbb{T}\}$ the set of \textit{gates} of $W$. Of course $m_T=g_T^{-1}$, however it turns out that the set $\Gamma$ appears to be more fundamental than the set of minimal length cone type representatives. 

We call the partition $\scrT=\{X_T\mid T\in\mathbb{T}\}$ of $W$ the \textit{cone type partition}. When $W$ is affine, the partition $\scrT$ shares many properties with the partition $\scrS$ of $W$ induced by the classical Shi arrangement introduced by Shi in~\cite{Shi:87a,Shi:87b} (and extensively studied ever since). We illustrate this below in the case $\tilde{\sB}_2$ (the parts of the partitions are the connected components; see Example~\ref{ex:completingB2} for the details of how to compute $\scrT$).

\begin{figure}[H]
\centering
\subfigure[The cone type partition~$\scrT$]{
\begin{tikzpicture}[scale=0.65]
\path [fill=gray!90] (0,0) -- (1,1) -- (0,1) -- (0,0);
\path [fill=blue!30] (-1,1)--(0,2)--(2,2)--(2,0)--(1,-1)--(0,-1)--(0,0)--(1,1)--(-1,1);
\path [fill=blue!30] (0,2)--(-1,3)--(1,3)--(0,2);
\path [fill=blue!30] (1,3)--(1,4)--(2,4)--(1,3);
\path [fill=blue!30] (-1,1)--(0,1)--(0,-1)--(-1,-1)--(-1,1);
\path [fill=blue!30] (-1,1)--(-2,0)--(-2,2)--(-1,1);
\path [fill=blue!30] (-2,0)--(-3,0)--(-3,-1)--(-2,0);
\path [fill=blue!30] (1,-1)--(1,-2)--(2,-2)--(1,-1);
\path [fill=blue!30] (2,0)--(3,0)--(3,-1)--(2,0);
\draw (-4,-2) -- (5,-2); 
\draw (-4,-1) -- (5,-1);
\draw [line width=2pt](-4,0) -- (5,0);
\draw [line width=2pt](-4,1) -- (5,1);
\draw (-4,2) -- (5,2);
\draw (-4,3) -- (5,3); 
\draw (-4,4) -- (5,4);
\draw (-4,5) -- (5,5);
\draw (-4,-3) -- (5,-3);
\draw (-4,-4) -- (5,-4);
\draw (-2,-4) -- (-2,5);
\draw (-1,-4) -- (-1,5);
\draw [line width=2pt](0,-4) -- (0,5);
\draw [line width=2pt](1,-4) -- (1,5);
\draw (2,-4) -- (2,5);
\draw (-4,-4) -- (-4,5);
\draw (-3,-4) -- (-3,5);
\draw (3,-4) -- (3,5);
\draw (4,-4) -- (4,5);
\draw (5,-4) -- (5,5);
\draw (-4,4)--(-3,5);
\draw (-4,2)--(-1,5);
\draw (-4,0) -- (1,5);
\draw (-4,-2) -- (3,5);
\draw [line width=2pt](-4,-4) -- (5,5);
\draw (-2,-4) -- (5,3);
\draw (0,-4) -- (5,1);
\draw (2,-4)--(5,-1);
\draw (4,-4)--(5,-3);
\draw (3,5)--(5,3);
\draw (1,5)--(5,1);
\draw (-1,5)--(5,-1);
\draw[line width=2pt] (-3,5)--(5,-3);
\draw[line width=2pt] (-4,4)--(4,-4);
\draw (-4,2)--(2,-4);
\draw (-4,0)--(0,-4);
\draw (-4,-2)--(-2,-4);
\draw[line width=2pt] (-4,-2)--(-1,1);
\draw[line width=2pt] (0,2)--(3,5);
\end{tikzpicture}
}\quad\qquad
\subfigure[The Shi partition $\scrS$]{\begin{tikzpicture}[scale=0.65]
\path [fill=gray!90] (0,0) -- (1,1) -- (0,1) -- (0,0);
\path [fill=blue!30] (-1,1)--(0,2)--(2,2)--(2,0)--(1,-1)--(0,-1)--(0,0)--(1,1)--(-1,1);
\path [fill=blue!30] (0,2)--(-1,3)--(1,3)--(0,2);
\path [fill=blue!30] (1,3)--(1,4)--(2,4)--(1,3);
\path [fill=blue!30] (-1,1)--(0,1)--(0,-1)--(-1,-1)--(-1,1);
\path [fill=blue!30] (-1,1)--(-2,0)--(-2,2)--(-1,1);
\path [fill=blue!30] (-2,0)--(-3,0)--(-3,-1)--(-2,0);
\path [fill=blue!30] (1,-1)--(1,-2)--(2,-2)--(1,-1);
\path [fill=blue!30] (2,0)--(3,0)--(3,-1)--(2,0);
\path [fill=red!30] (-1,1)--(-1,2)--(0,2)--(-1,1);
\draw (-4,-2) -- (5,-2); 
\draw (-4,-1) -- (5,-1);
\draw [line width=2pt](-4,0) -- (5,0);
\draw [line width=2pt](-4,1) -- (5,1);
\draw (-4,2) -- (5,2);
\draw (-4,3) -- (5,3); 
\draw (-4,4) -- (5,4);
\draw (-4,5) -- (5,5);
\draw (-4,-3) -- (5,-3);
\draw (-4,-4) -- (5,-4);
\draw (-2,-4) -- (-2,5);
\draw (-1,-4) -- (-1,5);
\draw [line width=2pt](0,-4) -- (0,5);
\draw [line width=2pt](1,-4) -- (1,5);
\draw (2,-4) -- (2,5);
\draw (-4,-4) -- (-4,5);
\draw (-3,-4) -- (-3,5);
\draw (3,-4) -- (3,5);
\draw (4,-4) -- (4,5);
\draw (5,-4) -- (5,5);
\draw (-4,4)--(-3,5);
\draw (-4,2)--(-1,5);
\draw (-4,0) -- (1,5);
\draw (-4,-2) -- (3,5);
\draw [line width=2pt](-4,-4) -- (5,5);
\draw (-2,-4) -- (5,3);
\draw (0,-4) -- (5,1);
\draw (2,-4)--(5,-1);
\draw (4,-4)--(5,-3);
\draw (3,5)--(5,3);
\draw (1,5)--(5,1);
\draw (-1,5)--(5,-1);
\draw[line width=2pt] (-3,5)--(5,-3);
\draw[line width=2pt] (-4,4)--(4,-4);
\draw (-4,2)--(2,-4);
\draw (-4,0)--(0,-4);
\draw (-4,-2)--(-2,-4);
\draw[line width=2pt] (-4,-2)--(3,5);
\end{tikzpicture}
}
\caption{The partitions $\scrS$ and $\scrT$ for $\tilde{\sB}_2$}\label{fig:B2partitions}
\end{figure}

In Figure~\ref{fig:B2partitions}(a) the gates $g_T$ are shaded blue. By direct observation, each part of the partition~$\scrS$ also contains a unique minimal length element, and these are shaded blue and red (the red element in Figure~\ref{fig:B2partitions}(b) is shaded red to highlight the difference with Figure~\ref{fig:B2partitions}(a)). Note that $\scrS$ is a refinement of $\scrT$ (written $\scrT\leq \scrS$), and that $\scrT$ is not a hyperplane arrangement. 

One can define the \textit{Shi partition} $\scrS$ for an arbitrary Coxeter group by declaring $x,y\in W$ to lie in the same part of $\scrS$ if and only if $\cE(x)=\cE(y)$ (see \cite[Definition~3.18]{HNW:16}, and note that in the affine case this agrees with the classical definition, as the hyperplanes of the Shi arrangement are precisely the hyperplanes corresponding to the elementary roots of $W$). While the celebrated result of Shi~\cite{Shi:87b} tells us that in the affine case each component of $\scrS$ contains a unique minimal length element, it is unknown if this is true for general Coxeter type, and this analogy underscores the difficulty in proving Theorem~\ref{thm:main1}.

The above discussion suggests that the language of ``partitions of $W$'' is the appropriate framework in which to study cone types and related structures. Indeed our approach to Theorem~\ref{thm:main1} is via a detailed study of a special class of partitions that we call the \textit{regular partitions} of $W$. A partition $\scrP$ of $W$ is \textit{regular} if the following conditions are satisfied for each part $P\in\scrP$:
\begin{enumerate}
\item if $x,y\in P$ then $D_L(x)=D_L(y)$ (write $D_L(P)$ for this common value), and
\item if $s\notin D_L(P)$ then $sP\subseteq P'$ for some part $P'\in\scrP$. 
\end{enumerate}
A partition satisfying~(1) is called \textit{locally constant}. Let $\scrP(W)$ denote the set of all partitions of $W$, and let $\Preg$ denote the set of all regular partitions of~$W$.

It is not hard to see that $\scrT$ is a regular partition, and other interesting examples of regular partitions include partitions induced by Garside shadows, and the partitions induced by general Shi arrangements (see Theorem~\ref{thm:regularpartitions}).

The following theorem (see Theorems~\ref{thm:regularautomaton} and~\ref{thm:converseautomaton}) shows that regular partitions are equivalent to ``reduced'' automata recognising $\cL(W,S)$ (here ``reduced'' is a natural and mild hypothesis, see Section~\ref{subsec:regularpartitions}).

\begin{thm1}\label{thm:main3} For each regular partition $\scrR$ of $W$ there exists an explicitly defined reduced automaton recognising $\cL(W,S)$, with accept states being the parts of $\scrR$. Moreover, every reduced automaton recognising $\cL(W,S)$ arises in such a way from some regular partition~$\scrR$.
\end{thm1}

Theorem~\ref{thm:main3} highlights the fundamental role regular partitions play in the automatic structure of $W$. In particular, our construction encapsulates all of the known constructions of automata recognising $\cL(W,S)$, and produces infinitely many new examples. 

To make further progress, and continuing towards a proof of Theorem~\ref{thm:main1}, we undertake a study of the structure of the partially ordered set $\Preg$ of all regular partitions. Here the partial order is $\scrP\leq \scrP'$ if $\scrP'$ is a refinement of $\scrP$. We prove (see Theorem~\ref{thm:regularlattice2}):

\begin{thm1}\label{thm:main4}
The partially ordered set $(\Preg,\leq)$ is a complete lattice, with bottom element $\scrT$ and top element $\mathbf{1}=\{\{w\}\mid w\in W\}$ (the partition into singletons). 
\end{thm1}

Thus for any partition $\scrP\in\scrP(W)$ one may define the \textit{regular completion} $\widehat{\scrP}\in\Preg$ by
$$
\widehat{\scrP}=\bigwedge \{\scrR\in\Preg\mid \scrP\leq \scrR\}.
$$
While it is not immediately obvious from the definition, we show in Theorem~\ref{thm:newterminateatregularisation} that  $\scrP\leq\widehat{\scrP}$, and hence $\widehat{\scrP}$ is the minimal regular partition refining $\scrP$. We give an algorithm (called the \textit{simple refinements algorithm}) to compute the regular completion, and prove sufficient conditions for this algorithm to terminate in finite time (see Algorithm~\ref{alg:regularisation} and Theorem~\ref{thm:finitetermination}). 

Thus we have an essentially free construction of regular partitions, and hence by Theorem~\ref{thm:main3} an essentially free construction of automata recognising the language $\cL(W,S)$. Furthermore, our sufficient conditions for the simple refinements algorithm to terminate in finite time leads to sufficient conditions for the resulting automata to be finite state.

An important corollary of Theorem~\ref{thm:main4} is the following characterisation of the cone type partition~$\scrT$. Let $\scrD$ be the partition of $W$ according to left descent sets (that is, $x$ and $y$ in the same part of $\scrD$ if and only if $D_L(x)=D_L(y)$). Then (see Corollary~\ref{cor:regularlattice1}):

\begin{cor1}\label{cor:TD}
We have $\scrT=\widehat{\scrD}$. 
\end{cor1}

Corollary~\ref{cor:TD}, along with the simple refinements algorithm, allows for $\scrT$ to be computed algorithmically (see Example~\ref{ex:completingB2}). Moreover, Corollary~\ref{cor:TD} is a key step in establishing Theorem~\ref{thm:main1}. 

We next introduce the notion of a \textit{gated} partition. A partition $\scrP$ of $W$ is called \textit{gated} if for each part $P\in\scrP$ there exists an element $g\in P$ with $g\peq x$ for all $x\in P$. These elements $g$ are called the ``gates'' of the partition, and we write $\Gamma(\scrP)$ for the set of gates of $\scrP$ (it is clear that each part $P$ of a gated partition has a unique gate). 

We show that if $\scrP$ is a gated and convex partition, then the simple refinements algorithm preserves the gated property. Here \textit{convex} means that each part of the partition is convex in the usual sense. Thus we have the following theorem (see Corollary~\ref{cor:regularcompletiongated}), which, combined with Corollary~\ref{cor:TD}, finally leads to a proof of Theorem~\ref{thm:main1}. 

\begin{thm1}\label{thm1:preservegatedness}
Let $\scrP$ be a locally constant, convex and gated partition of $W$. If the simple refinements algorithm terminates in finite time, then the regular completion $\widehat{\scrP}$ is gated and convex.
\end{thm1}

The finite set $\Gamma=\Gamma(\scrT)$ (the set of gates of~$W$) has many remarkable properties. For example, $\Gamma$ is closed under taking suffix, contains all spherical elements of $W$, is contained in every Garside shadow, and is contained in the set $\Gamma(\scrP)$ of gates of every gated regular partition~$\scrP$ (see Proposition~\ref{prop:gatesbasicfacts}). Moreover we make the following conjecture (which in turn would resolve \cite[Conjecture~1]{HNW:16}, see Theorem~\ref{thm:GminGamma}).

\begin{conj1}\label{conj:garside}
The set $\Gamma$ is closed under join, and hence is a Garside shadow. 
\end{conj1}

Let $\Phisph^+$ denote the set of \textit{spherically supported} positive roots. We have $\Phisph^+\subseteq \cE$, however this containment can be strict (see~\cite[Theorem~1]{PY:19} for the classification of Coxeter systems for which $\cE=\Phisph^+$). Let $L\subseteq W$ denote the set of \textit{low elements} of $W$ introduced by Dehornoy, Dyer, and Hohlweg in~\cite{DDH:15} (see Definition~\ref{defn:nlow}). We prove the following, providing evidence for Conjecture~\ref{conj:garside} (see Theorem~\ref{thm:conjectures}).

\begin{thm1}\label{thm:conjforspherical}
If $\cE=\Phisph^+$ then $\Gamma=\Gmin=L$, and so $\Gamma$ is a Garside shadow.
\end{thm1}

%
%

Other main contributions of this paper include the following. In Section~\ref{sec:conetypes} we characterise the minimal set $\partial T\subseteq \Phi^+$ of positive roots required to determine $T$ (we call $\partial T$ the \textit{boundary roots} of~$T$), and provide a precise formula for cone types in terms of these roots. If $T$ is a cone type, and $x\in W$ is such that $T=T(x^{-1})$, we define
$$
\partial T=\{\beta\in\Phi^+\mid \text{there exists $w\in W$ with $\Phi(x)\cap\Phi(w)=\{\beta\}$}\}.
$$
This set of roots is independent of the particular representative $x\in W$ with $T=T(x^{-1})$ chosen, and we prove the following theorem (see Theorem~\ref{thm:geometry2}). 

\begin{thm1}\label{thm:main2}
Let $T$ be a cone type. Then
$$
T=\bigcap_{\beta\in\partial T}H_{\beta}^+\quad\text{where}\quad H_{\beta}^+=\{w\in W\mid \ell(s_{\beta}w)>\ell(w)\}.
$$
Moreover, removing any root from the above intersection results in strict containment.
\end{thm1}

We note that if $T(x^{-1})=T$ then $\partial T\subseteq\cE(x)$, however strict containment can occur. Moreover, we have $\partial T\subseteq \Phi^1(x)$ (where $\Phi^1(x)$ is the ``base'' of $\Phi(x)$, defined by Dyer~\cite{Dye:19}), however again strict containment can occur. We make the following conjecture (for which we can only prove the reverse implication):

\begin{conj1}\label{conj:boundary}
Let $x\in W$. Then $x\in \Gamma$ if and only if $\partial T(x^{-1})=\Phi^1(x)$. 
\end{conj1}

We also identify the ``partition theoretic'' equivalent to Garside shadows, in the following sense. If $\scrP$ is a regular gated partition, then the set $\Gamma(\scrP)$ of all gates of $\scrP$ contains $S$ and is automatically closed under suffix. However $\Gamma(\scrP)$ is not necessarily closed under join. In Section~\ref{sec:conical} we define \textit{conical partitions} (see Definition~\ref{defn:conical}). These partitions are necessarily gated, and the set of gates of a conical partition is necessarily closed under join. Thus regular conical partitions are equivalent to Garside shadows (see Corollary~\ref{cor:garsideequivalent}).

In Section~\ref{sec:ultralow}, motivated by Conjecture~\ref{conj:boundary}, we define \textit{ultra low} elements in a Coxeter group to be the elements $x\in W$ with $\partial T(x^{-1})=\Phi^1(x)$, and investigate their properties. We have $U\subseteq \Gamma\subseteq \Gmin$. Conjecture~\ref{conj:boundary}, if true, implies that $U=\Gamma$, and Conjecture~\ref{conj:garside}, if true, implies that $\Gamma=\Gmin$.

Finally, in Section~\ref{sec:superelementary} we consider the question of which elementary roots occur as a boundary root of some cone type. We show that in spherical and affine Coxeter groups all elementary roots occur, and we exhibit a family of rank~$4$ Coxeter groups where the inclusion is strict.

We thank C. Hohlweg and M. Dyer for helpful comments on an earlier version of this work, and R. Howlett for useful discussions concerning elementary roots and super elementary roots. This work was supported by funding from the Australian Research Council under the Discovery Project~DP200100712.

\newpage

\section{Preliminaries}\label{sec:1}

In this section we give an overview of background and preliminary results on Coxeter groups, root systems, elementary roots, the Coxeter complex, low elements, Garside shadows, cone types, and automata recognising the language of reduced words in a Coxeter group. Our main references are \cite{AB:08,BB:05,Deo:82} (for Coxeter groups, the Coxeter complex, and root systems), \cite{BH:93} (for elementary roots), \cite{Dye:21,DH:16} (for low elements), \cite{DDH:15,HNW:16} (for Garside shadows), and \cite{Eps:92,HNW:16} (for relevant automata theory).

\subsection{Coxeter groups}\label{sec:1:Coxetergroups} Let $(W,S)$ be a Coxeter system. We will assume throughout that $|S| < \infty$. For $s,t\in S$ let $m_{s,t}$ denote the order of $st$. The \textit{length} of $w\in W$ is 
$$
\ell(w)=\min\{n\geq 0\mid w=s_1\cdots s_n\text{ with }s_1,\ldots,s_n\in S\},
$$
where $\ell(e)=0$, with $e$ the identity element. An expression $w=s_1 \cdots s_n$ with $n=\ell(w)$ is called a \textit{reduced expression} for~$w$. An element $v \in W$ is a \textit{prefix} of $w$ if $\ell(w) = \ell(v) + \ell(v^{-1}w)$. Similarly, an element $u \in W$ is a \textit{suffix} of $w$ if $\ell(w) = \ell(u) + \ell(w u^{-1})$. Note that $v$ is a prefix (respectively suffix) of $w$ if and only if there is a reduced expression for $w$ starting (respectively ending) with a reduced expression for $v$.

Let $J\subseteq S$. The \textit{$J$-parabolic subgroup} of $W$ is the subgroup $W_J=\langle J\rangle$, and we say that $J$ is \textit{spherical} if $|W_J| < \infty$. If $J$ is spherical then there exists a unique longest element of $W_J$, denoted~$w_J$, and we have $\ell(w_Jw)=\ell(ww_J)=\ell(w_J)-\ell(w)$ for all $w\in W_J$. If $|W|<\infty$ (that is, $S$ is spherical) then we often write $w_0=w_S$ for the longest element of~$W$.

The \textit{left descent set} of $w\in W$ is
$$
D_L(w)=\{s\in S\mid \ell(sw)=\ell(w)-1\},
$$
and similarly the right descent set is $D_R(w)=\{s\in S\mid \ell(ws)=\ell(w)-1\}$. By \cite[Proposition~2.17]{AB:08} both $D_L(w)$ and $D_R(w)$ are spherical subsets of $S$ for all $w\in W$.

Let $J\subseteq S$. It is well known (see, for example \cite[Proposition~2.20]{AB:08}) that each coset $W_Jw$ contains a unique representative of minimal length. Let $W^J$ be the transversal of these minimal length coset representatives. Then each $w\in W$ has a unique decomposition
\begin{align}
\label{eq:WJdecomposition}w=uv\quad\text{with $u\in W_J$, $v\in W^J$},
\end{align}
and moreover whenever $u\in W_J$ and $v\in W^J$ we have $\ell(uv)=\ell(u)+\ell(v)$.

The \textit{right weak order} is the partial order defined on $W$ with $v \peq w$ if $v$ is a prefix of~$w$. The partially ordered set $(W, \peq)$ is a complete meet semilattice (see \cite[Chapter 3.2]{BB:05}), and thus for any subset $X \subseteq W$ there is a greatest lower bound (or \textit{meet}), denoted $\bigwedge X$. A \textit{bound} for a subset $X\subseteq W$ is an element $w\in W$ such that $x\peq w$ for all $x\in X$. It follows from the existence of meets that every bounded subset $X \subseteq W$ admits a least upper bound (or \textit{join}), given by
$$
    \bigvee X = \bigwedge \{ w \in W \mid x \peq w \text{ for all }x \in X \}.
$$
If $X=\{x,y\}$ is bounded we write $\bigwedge \{x,y\}=x\wedge y$ and $\bigvee \{x,y\}=x\vee y$.

\subsection{Root systems}\label{sec:1:rootsystems} Let $(W,S)$ be a Coxeter system. Let $V$ be an $\mathbb{R}$-vector space with basis $\Pi=\{\alpha_s\mid s\in S\}$, and define a symmetric bilinear form on $V$ by linearly extending $\langle\alpha_s,\alpha_t\rangle=-\cos(\pi/m_{s,t})$. The Coxeter group $W$ acts on $V$ by the rule $s\lambda=\lambda-2\langle \lambda,\alpha_s\rangle \alpha_s$ for $s\in S$ and $\lambda\in V$, and the root system of $W$ is 
$$
\Phi=\{w\alpha_s\mid w\in W,\,s\in S\}.
$$ 
The elements of $\Phi$ are called \textit{roots}, and the \textit{simple roots} are the roots $\alpha_s$ with $s\in S$. 

\begin{rem} Note that $\langle\alpha_s,\alpha_t\rangle=-1$ if $m_{st}=\infty$. One may work more generally with an arbitrary \textit{based root system} $\Phi$ associated to $W$, as in \cite[\S2.3]{DH:16}, however for the purpose of this paper the concrete choice of realisation described above suffices.
\end{rem}

Each root $\alpha\in\Phi$ can be written as $\alpha=\sum_{s\in S}c_s\alpha_s$ with either $c_s\geq 0$ for all $s\in S$, or $c_s\leq 0$ for all $s\in S$. In the first case $\alpha$ is called \textit{positive} (written $\alpha>0$), and in the second case $\alpha$ is called \textit{negative} (written $\alpha<0$). Let $\Phi^+$ denote the set of all positive roots.

The set of \textit{reflections} of $W$ is $\{wsw^{-1}\mid w\in W,\,s\in S\}$. If $w\alpha_s=\beta$ we define $s_{\beta}=wsw^{-1}$. Note that this reflection acts on $V$ by $s_{\beta}\lambda=\lambda-2\langle \lambda,\beta\rangle\beta$. 

The \textit{inversion set} of $w\in W$ is 
$$
\Phi(w)=\{\alpha\in \Phi^+\mid w^{-1}\alpha<0\}.
$$

We recall some well-known facts in the following proposition.

\begin{prop}\label{prop:rootsystembasics}
Let $u,v,w\in W$ and $s\in S$. 
\begin{enumerate}
\item We have $\ell(ws)>\ell(w)$ if and only if $w\alpha_s>0$. 
\item We have $\ell(sw)>\ell(w)$ if and only if $w^{-1}\alpha_s>0$. 
\item If $w=s_1\cdots s_n$ is reduced, then $\Phi(w)=\{\beta_1,\ldots,\beta_n\}$ where 
$$
\beta_j=s_1\cdots s_{j-1}\alpha_{s_j}\quad\text{for $1\leq j\leq n$}.
$$
\item We have $\Phi(w)=\{\beta\in\Phi^+\mid\ell(s_{\beta}w)<\ell(w)\}$.
\item $\Phi(v)\subseteq \Phi(w)$ if and only if $v\peq w$. 
\item If $w=uv$ with $\ell(w)=\ell(u)+\ell(v)$ then $\Phi(w)=\Phi(u)\sqcup u\Phi(v)$. 
\end{enumerate}
\end{prop}

\begin{proof}
See~\cite[Proposition~2.2 and Proposition~3.1]{Deo:82} and \cite[Proposition~3.1.3]{BB:05}.
\end{proof}

 The \textit{support} of a root $\alpha\in\Phi$ is $\mathrm{supp}(\alpha)=\{s\in S\mid c_s\neq 0\}$, where $\alpha=\sum_{s\in S}c_s\alpha_s$.  For $J\subseteq S$ let 
$$
\Phi_J=\{\alpha\in\Phi\mid\mathrm{supp}(\alpha)\subseteq J\},
$$
and for $w\in W$ write $\Phi_J(w)=\Phi(w)\cap\Phi_J$. 

\begin{lem}\cite[Corollary~2.13]{HNW:16}\label{lem:rootsdecomposition}
Let $J\subseteq S$. If $w=uv$ with $u\in W_J$ and $v\in W^J$ then $\Phi(u)=\Phi_J(w)$. In particular, we have
$
W^J=\{v\in W\mid \Phi_J(v)=\emptyset\}.
$
\end{lem}


Each root $\beta\in\Phi^+$ partitions $W$ into two sets
$$
H_{\beta}^+=\{w\in W\mid \ell(s_{\beta}w)>\ell(w)\}\quad\text{and}\quad H_{\beta}^-=\{w\in W\mid \ell(s_{\beta}w)<\ell(w)\}.
$$
Note that $e\in H_{\beta}^+$. We call $H_{\beta}^+$ and $H_{\beta}^-$ the \textit{half-spaces determined by $\beta$}. Note that if $\beta\in\Phi^+$ then $\beta\in\Phi(w)$ if and only if $w\in H_{\beta}^-$.

\subsection{Elementary roots}\label{sec:1:elementaryroots} A root $\beta\in\Phi^+$ is said to \textit{dominate} a root $\alpha\in \Phi^+$ if $w^{-1}\beta<0$ implies that $w^{-1}\alpha<0$ (for all $w\in W$). A root $\beta\in \Phi^+$ is said to be \textit{elementary} if $\beta$ dominates no other positive root $\alpha\neq \beta$. Geometrically, $\beta$ dominates $\alpha$ if and only if $H_{\beta}^-\subseteq H_{\alpha}^-$, or equivalently, if and only if $H_{\beta}^+\supseteq H_{\alpha}^+$. We note that elementary roots are also called \textit{small}, \textit{humble} or \textit{minimal} in the literature. 

Let $\cE\subseteq \Phi^+$ denote the set of all elementary roots. By \cite[Theorem~2.8]{BH:93} the set $\cE$ is finite for all (finitely generated) Coxeter systems~$(W,S)$.

The \textit{elementary inversion set} of $w\in W$ is 
$$
\cE(w)=\{\beta\in\cE\mid w^{-1}\beta<0\}=\Phi(w)\cap \cE.
$$
Let $\mathbb{E}=\{\cE(w)\mid w\in W\}$ be the set of all elementary inversion sets. Since $\cE$ is finite, $\mathbb{E}$ is finite too.

Let $n\in\mathbb{N}$. A root $\beta\in \Phi^+$ is called \textit{$n$-elementary} if it dominates at most $n$ roots $\alpha\in\Phi^+\backslash\{\beta\}$. Thus $0$-elementary roots are the same as elementary roots. Let $\cE_n$ denote the set of all $n$-elementary roots. 

The $n$-elementary inversion set of $w\in W$ is $\cE_n(w)=\Phi(w)\cap\cE_n$. Let $\mathbb{E}_n$ denote the set of all $n$-elementary inversion sets. By \cite[Corollary~3.9]{Fu:12} the set $\cE_n$ is finite for each $n\in\mathbb{N}$, and hence:

\begin{cor}\label{cor:Enfinite} 
The set $\mathbb{E}_n$ is finite for each $n\in\mathbb{N}$. 
\end{cor}

The following lemma is key to the automatic structure of~$W$ (see \cite{BH:93} and \cite[Lemma~3.21]{DH:16}).

\begin{lem}\label{lem:elementaryrootsbasics}
Let $w\in W$, $s\in S$, and $n\in\mathbb{N}$. If $\ell(sw)>\ell(w)$ then 
$$\cE_n(sw)=(\{\alpha_s\}\sqcup s\cE_n(w))\cap\cE_n.
$$
\end{lem}

The set of \textit{spherical roots} is 
$$
\Phisph=\{\alpha\in\Phi\mid\mathrm{supp}(\alpha)\subseteq J\text{ for some spherical subset $J\subseteq S$}\}.
$$
Let 
$$
\mathbb{S}=\{\Phisph(w)\mid w\in W\},\quad\text{where $\Phisph(w)=\Phi(w)\cap\Phisph$}.
$$
Clearly $\mathbb{S}$ is finite. 

We have $\Phisph^+\subseteq \cE$, however this containment can be strict. The classification of Coxeter systems for which $\cE=\Phisph^+$ is as follows (see~\cite[Theorem~1]{PY:19}). Let $\mathcal{X}$ denote the set of connected Coxeter graphs which are either of affine or compact hyperbolic type and contain neither circuits nor infinite bonds. Then $\cE=\Phisph^+$ if and only if the Coxeter graph of $(W,S)$ does not have a subgraph contained in $\mathcal{X}$. In particular, if follows that $\cE=\Phisph^+$ whenever $(W,S)$ is spherical, of type $\tilde{A}_n$, right-angled, or has complete Coxeter graph (that is, $m_{s,t}\geq 3$ for all $s,t\in S$ with $s\neq t$).

\subsection{The Coxeter complex}\label{sec:1:Coxetercomplex}

The \textit{Coxeter complex} of a Coxeter system is a certain abstract simplicial complex $\Sigma(W,S)$ on which~$W$ naturally acts. While no result of this paper formally depends on the Coxeter complex, it is nonetheless a useful concept for providing a geometric intuition for Coxeter groups.

We refer to \cite[Chapter~3]{AB:08} for the formal construction of $\Sigma(W,S)$. Here we provide a less formal sketch. For each $w\in W$ let $C_w$ be a combinatorial simplex with $|S|$ vertices, and assign each vertex $x$ of $C_w$ a \textit{type} $\tau(x)\in S$ such that $C_w$ contains precisely one vertex of each type~$s\in S$. For each $w\in W$ and $s\in S$ we glue $C_w$ and $C_{ws}$ together along their cotype $\{s\}$ faces, identifying the vertex of type $s'$ in $C_{w}$ with the vertex of type $s'$ in $C_{ws}$ for all $s'\in S\backslash\{s\}$. The resulting simplicial complex $\Sigma(W,S)$ is called the \textit{Coxeter complex} of $(W,S)$. The Coxeter complex has maximal simplices $C_w$, $w\in W$, and these are called the \textit{chambers} (or sometimes \textit{alcoves}) of the complex. The Coxeter group $W$ acts on $\Sigma(W,S)$ by type preserving simplicial complex automorphisms. On the level of chambers this action is given by $wC_v=C_{wv}$ for all $w,v\in W$, and the action on the set of chambers is simply transitive. Let $C_0=C_e$ be the \textit{fundamental chamber}, and so $C_w=wC_0$. By construction, the chambers $wC_0$ and $wsC_0$ are \textit{$s$-adjacent} (meaning they share a cotype~$\{s\}$ face). 

Let $\beta\in\Phi^+$. The set 
$$
H_{\beta}=\{\sigma\in\Sigma(W,S)\mid s_{\beta}(\sigma)=\sigma\}
$$ 
of all simplices fixed by $s_{\beta}$ is called a \textit{wall} of the Coxeter complex. Since $s_{\beta}$ fixes no chambers, there are no chambers contained in the wall $H_{\beta}$. This illustrates the utility of the Coxeter complex, as one can now speak of the wall $H_{\beta}$ separating the half-spaces $H_{\beta}^+$ and $H_{\beta}^-$.

We will sometimes identify $W$ with the set of chambers of $\Sigma(W,S)$ by identifying $w\leftrightarrow wC_0$. Thus one may simultaneously think of $W$ as a group, and more geometrically as the associated Coxeter complex.

\subsection{Low elements}\label{sec:1:lowelements}

The \textit{base} of an inversion set is defined in terms of extreme rays of the cone of $\Phi(w)$ (see \cite{Dye:19} and \cite{DH:16}), however for our purposes the following equivalent characterisation is sufficient (see \cite[Proposition~4.6]{DH:16}). 

\begin{defn}\label{defn:phi1}
Let $w \in W$. The \textit{base} of the inversion set $\Phi(w)$ is
$$
    \Phi^1(w) = \{ \beta \in \Phi^+ \mid \ell(s_{\beta}w) = \ell(w) - 1 \}.
$$
\end{defn}

By Proposition~\ref{prop:rootsystembasics}(4) we have $\Phi^1(w)\subseteq\Phi(w)$. For $A \subseteq \Phi^+$ let $\mathrm{cone}(A)$ be the set of all non-negative linear combinations of roots in $A$ and write $\mathrm{cone}_{\Phi}(A) = \mathrm{cone}(A) \cap \Phi^+$. The set $\Phi^1(w)$ determines the inversion set $\Phi(w)$ in the following sense. 

\begin{thm}\cite[Lemma 1.7]{Dye:19}\label{thm:eqcond} Let $w\in W$. Then
\begin{align*}
\Phi(w) = \mathrm{cone}_{\Phi}(\Phi^1(w)),
\end{align*}
and moreover if $A \subseteq \Phi^+$ is such that $\Phi(w) = \mathrm{cone}_{\Phi}(A)$ then $\Phi^1(w) \subseteq A$.
\end{thm}

\newpage

In \cite{DH:16} Dyer and Hohlweg introduced the notion of an $n$-low element of a Coxeter group~$W$. 

\begin{defn}\label{defn:nlow}
Let $n\in\mathbb{N}$. An element $w\in W$ is \textit{$n$-low} if $\Phi(w)=\mathrm{cone}_{\Phi}(A)$ for some $A\subseteq\cE_n$. A $0$-low element is called \textit{low}. Let $L_n$ denote the set of all $n$-low elements, and let $L=L_0$ denote the set of low elements. Note that by Theorem~\ref{thm:eqcond} we have that $w$ is $n$-low if an only if $\Phi^1(w)\subseteq \cE_n$.
\end{defn}

Let $\Theta_n:L_n\to\mathbb{E}_n$ be the map $\Theta_n(x)=\cE_n(x)$ (introduced by Dyer and Hohlweg in \cite{DH:16}). This map is injective (see \cite[Proposition~3.26]{DH:16}), and hence $|L_n|\leq|\mathbb{E}_n|$ for all $n\in\mathbb{N}$. In \cite[Conjecture~2]{DH:16} Dyer and Hohlweg conjecture that $\Theta_n$ is a bijection for all $n\in\mathbb{N}$. 
%
%
%

The following result is useful when working with joins. 

\begin{prop}\cite[Proposition 2.8]{DH:16} \label{prop:inversionsetjoin}
If $X \subseteq W$ is bounded, then
$$
    \Phi(\bigvee X) = \mathrm{cone}_{\Phi}(\bigcup_{x \in X} \Phi(x)).
$$
\end{prop}

Each reduced expression $w=s_1\cdots s_n$ gives rise to an ordering of the inversion set of $w$, as in Proposition~\ref{prop:rootsystembasics}(3). In particular, the ``final root''  of this ordered sequence is $\beta=s_1\cdots s_{n-1}\alpha_{s_n}=-w\alpha_{s_n}>0$ (see Proposition~\ref{prop:rootsystembasics}(3)). The set of such roots $\beta$, as the reduced expression for $w$ varies, plays an important role later in this work. 

\begin{defn}\label{defn:finalroots}
Let $w\in W$. The set of \textit{final roots} of $w$ is 
$$
\Phi^0(w)=\{-w\alpha_s\mid s\in D_R(w)\}.
$$
\end{defn}

Note that $\beta\in\Phi^0(w)$ if and only if $s_{\beta}w=ws$ for some $s\in D_R(w)$, if and only if $\beta=-w\alpha_s>0$ for some $s\in S$. Also note that $\Phi^0(w)\subseteq \Phi^1(w)$.

%
%
%
%
%
%
%
%

\subsection{Garside shadows}

The notion of a \textit{Garside shadow} in a Coxeter system $(W,S)$ was introduced and investigated by Dehornoy, Dyer and Hohlweg \cite{DDH:15} and Dyer and Hohlweg in~\cite{DH:16}. 

\begin{defn} \label{def:garside_shadow}
A \textit{Garside shadow} is a subset $B \subseteq W$ such that $S \subseteq B$ and
\begin{enumerate}
    \item for $X \subseteq B$ if $w = \bigvee X$ exists, then $w \in B$;
    \item if $w \in B$ and $v$ is a suffix of $w$ then $v \in B$.
\end{enumerate}
We refer to (1) as \textit{closure under join}, and (2) as \textit{closure under taking suffixes}. 
\end{defn}

It is clear that the intersection of two Garside shadows is again a Garside shadow (see \cite[Proposition~2.2]{DH:16}) and hence there is a unique smallest Garside shadow, denoted $\Gmin$. Using the finiteness of the set of elementary roots, Dyer and Hohlweg show in \cite[Theorem~1.1]{DH:16} that $\Gmin$ is finite for all finitely generated Coxeter systems~$(W,S)$.

If $B$ is a Garside shadow then each element $w\in W$ can be ``projected'' onto $B$ as follows.

\begin{defn}\cite[Definition~2.4]{HNW:16} Let $B \subseteq W$ be a Garside shadow. The \textit{projection} of $W$ onto $B$ is the function $\pi_B:W\to B$ given by
$$
    \pi_B(w) = \bigvee \{ b \in B \mid b \peq w \}.
$$
Note that $\pi_B(w)\in B$ because $B$ is closed under join.
\end{defn}

The following important theorem was first conjectured in \cite[Conjecture~1]{DH:16}, where it was proved in the case $n=0$ (see \cite[Theorem~1.1]{DH:16}), and for all $n\in\mathbb{N}$ in the case that $W$ is affine (see \cite[Theorem~4.17]{DH:16}). Recently Dyer~\cite{Dye:21} has proved the theorem for all $n\in\mathbb{N}$ for arbitrary~$W$. 

\begin{thm}\label{thm:nlowgarside}\cite[Corollary~1.7]{Dye:21}
Let $n\in\mathbb{N}$. The set $L_n$ of $n$-low elements is a finite Garside shadow. 
\end{thm}

\subsection{Cone types}\label{sec:1:conetypes}

The \textit{cone type} of $w\in W$ is
$$
    T(w) = \{ v \in W \mid \ell(wv) = \ell(w) + \ell(v) \}.
$$
Thus $T(w)$ consists of all elements $v$ that ``extend'' $w$. Let $\mathbb{T}=\{T(w)\mid w\in W\}$ be the set of all cone types of $W$. Cone types play a central role in this work. The following proposition collects some basic results. 

\newpage

\begin{prop} \label{prop:conetypebasics}
Let $x,y \in W$. The following are equivalent:
\begin{enumerate}
     \item $\ell(x^{-1}y) = \ell(x) + \ell(y)$
    \item $y \in T(x^{-1})$
    \item $x\in T(y^{-1})$
    \item $\Phi(x) \cap \Phi(y) = \emptyset$
    \item $\Phi(x^{-1}y) = \Phi(x^{-1}) \sqcup x^{-1}\Phi(y)$.
\end{enumerate}
\end{prop}

\begin{proof}
The equivalence of (1), (2) and (3) is immediate from the definitions. For the equivalence of (1) with (4) see, for example \cite[Lemma~1.2]{BH:93}. Finally, if $\ell(x^{-1}y)=\ell(x)+\ell(y)$ then $\Phi(x^{-1}y)=\Phi(x^{-1})\sqcup x^{-1}\Phi(y)$ by Proposition~\ref{prop:rootsystembasics}(6), and conversely if $\Phi(x^{-1}y)=\Phi(x^{-1})\sqcup x^{-1}\Phi(y)$ then $\ell(x^{-1}y)=\ell(x)+\ell(y)$ because $\ell(w)=|\Phi(w)|$ for all $w\in W$, completing the proof.
%
%
\end{proof}

We also note the following obvious fact.

\begin{lem}\label{lem:containoneway}
If $x\peq y$ then $T(y^{-1})\subseteq T(x^{-1})$. 
\end{lem}

\begin{proof}
If $w\in T(y^{-1})$ then $\Phi(y)\cap\Phi(w)=\emptyset$ (by Proposition~\ref{prop:conetypebasics}), and hence $\Phi(x)\cap\Phi(w)=\emptyset$ (as $x\peq y$ implies that $\Phi(x)\subseteq \Phi(y)$) and hence $w\in T(x^{-1})$ (again by Proposition~\ref{prop:conetypebasics}). 
\end{proof}

%
%
%
%

The following result gives a formula for cone types in terms of inversion sets and half-spaces. 

\begin{thm} \label{thm:geometry1}
For $x\in W$ we have
$$
    T(x^{-1}) = \bigcap_{\beta \in \Phi(x)} H_{\beta}^{+}.
$$
\end{thm}
\begin{proof}
We have $y\in T(x^{-1})$ if and only if $\Phi(x)\cap\Phi(y)=\emptyset$ (by Proposition~\ref{prop:conetypebasics}), if and only if $y^{-1}\beta>0$ for all $\beta\in\Phi(x)$, if and only if $\ell(s_{\beta}y)>\ell(y)$ for all $\beta\in \Phi(x)$ (by Proposition~\ref{prop:rootsystembasics}), if and only if $y\in\bigcap_{\beta \in \Phi(x)} H_{\beta}^{+}$. 
\end{proof}

\begin{exa}
To apply the formula in Theorem~\ref{thm:geometry1} to determine $T(w^{-1})$, one considers all walls of the Coxeter complex that separate $e$ and $w$ (the positive roots corresponding to these walls are the elements of $\Phi(w)$), and take the intersection of the half-spaces containing the identity for each of these walls. Let us illustrate in an example. 

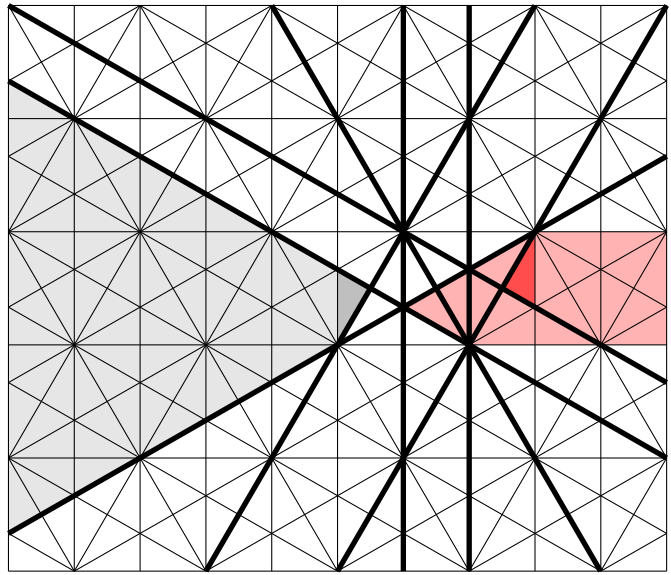
\begin{figure}[H]
\centering
\begin{tikzpicture}[scale=1]
\path [fill=gray!20] (-4.33,-2.5)--(0,0)--(0.433,0.75)--(-4.33,3.5);
\path [fill=gray!50] (0,0) -- (0.433,0.75) -- (0,1) -- (0,0);
\path [fill=red!30] (0.866,0.5)--(2.598,1.5)--(4.33,1.5)--(4.33,0)--(1.732,0);
\path [fill=red!70] (2.165,0.75)--(2.598,1.5)--(2.598,0.5);
\draw(-4.33,4.5)--(4.33,4.5);
\draw(-4.33,3)--(4.33,3);
\draw(-4.33,1.5)--(4.33,1.5);
\draw(-4.33,0)--(4.33,0);
\draw(-4.33,-1.5)--(4.33,-1.5);
\draw(-4.33,-3)--(4.33,-3);
\draw(-4.33,-3)--(-4.33,4.5);
\draw(-3.464,-3)--(-3.464,4.5);
\draw(-2.598,-3)--(-2.598,4.5);
\draw(-1.732,-3)--(-1.732,4.5);
\draw(-.866,-3)--(-.866,4.5);
\draw(0,-3)--(0,4.5);
\draw(.866,-3)--(.866,4.5);
\draw(1.732,-3)--(1.732,4.5);
\draw(2.598,-3)--(2.598,4.5);
\draw(3.464,-3)--(3.464,4.5);
\draw(4.33,-3)--(4.33,4.5);
\draw(-4.33,3.5)--({-3*0.866},4.5);
\draw(-4.33,2.5)--({-1*0.866},4.5);
\draw(-4.33,1.5)--({1*0.866},4.5);
\draw(-4.33,.5)--({3*0.866},4.5);
\draw(-4.33,-.5)--(4.33,4.5);
\draw(-4.33,-1.5)--(4.33,3.5);
\draw(-4.33,-2.5)--(4.33,2.5);
\draw(-3.464,-3)--(4.33,1.5);
\draw(-1.732,-3)--(4.33,.5);
\draw(0,-3)--(4.33,-.5);
\draw(1.732,-3)--(4.33,-1.5);
\draw(3.464,-3)--(4.33,-2.5);
\draw(4.33,3.5)--({3*0.866},4.5);
\draw(4.33,2.5)--({1*0.866},4.5);
\draw(4.33,1.5)--({-1*0.866},4.5);
\draw(4.33,.5)--({-3*0.866},4.5);
\draw(4.33,-.5)--(-4.33,4.5);
\draw(4.33,-1.5)--(-4.33,3.5);
\draw(4.33,-2.5)--(-4.33,2.5);
\draw(3.464,-3)--(-4.33,1.5);
\draw(1.732,-3)--(-4.33,.5);
\draw(0,-3)--(-4.33,-.5);
\draw(-1.732,-3)--(-4.33,-1.5);
\draw(-3.464,-3)--(-4.33,-2.5);
\draw(-4.33,-1.5)--(-3.464,-3);
\draw(-4.33,1.5)--(-1.732,-3);
\draw(-4.33,4.5)--(0,-3);
\draw({-3*0.866},4.5)--(1.732,-3);
\draw({-1*0.866},4.5)--(3.464,-3);
\draw({1*0.866},4.5)--(4.33,-1.5);
\draw({3*0.866},4.5)--(4.33,1.5);
\draw(4.33,-1.5)--(3.464,-3);
\draw(4.33,1.5)--(1.732,-3);
\draw(4.33,4.5)--(0,-3);
\draw({3*0.866},4.5)--(-1.732,-3);
\draw({1*0.866},4.5)--(-3.464,-3);
\draw({-1*0.866},4.5)--(-4.33,-1.5);
\draw({-3*0.866},4.5)--(-4.33,1.5);
\draw[line width=2pt](0.866,-3)--(0.866,4.5);%
\draw[line width=2pt](-1.732,-3)--({3*0.866},4.5);%
\draw[line width=2pt](3.46,-3)--(-0.866,4.5);%
\draw[line width=2pt](-4.33,-2.5)--(4.33,2.5);
\draw[line width=2pt](-4.33,3.5)--(4.33,-1.5);
\draw[line width=2pt](1.732,-3)--(1.732,4.5);
\draw[line width=2pt](0,-3)--(4.33,4.5);
\draw[line width=2pt](-4.33,4.5)--(4.33,-0.5);
\end{tikzpicture}
\caption{A cone type}\label{fig:G2examplecone}
\end{figure}

\noindent Let $w$ be the element shaded dark red in Figure~\ref{fig:G2examplecone}. The identity is shaded grey, and the walls separating $e$ from $w$ are show in bold. The intersection of the corresponding positive half-spaces is shown in light grey -- this is the cone type $T=T(w^{-1})$. Note that some of the walls are ``redundant'' in the sense that the corresponding roots can be removed from the intersection in Theorem~\ref{thm:geometry1}. We will address this issue in Theorem~\ref{thm:geometry2}. The red shaded region is $X_T=\{x\in W\mid T(x^{-1})=T\}$ (see Proposition~\ref{prop:partsdescription}). Note that this is a convex region, with a unique minimal length element $g$, and moreover for all $x\in X_T$ we have $g\peq x$. We will prove these observations in general in Corollary~\ref{cor:gateexist}.
\end{exa}

The \textit{cone} of $w\in W$ is 
$$
C(w)=\{v\in W\mid \ell(v)=\ell(w)+\ell(w^{-1}v)\}=\{v\in W\mid w\peq v\}.
$$
Note that $T(w)=w^{-1}C(w)$. The following lemma shows that joins and intersections of cones are closely related. 

\begin{lem}\label{lem:conejoin}
A subset $X\subseteq W$ is bounded if and only if $\bigcap_{x\in X}C(x)\neq\emptyset$, and if $X\subseteq W$ is bounded then $C(\bigvee X)=\bigcap_{x\in X}C(x)$. 
\end{lem}

\begin{proof}
Both statements are clear from the fact that $y\in \bigcap_{x\in X}C(x)$ if and only if $y$ is an upper bound for $X$. 
\end{proof}

We note, in passing the following result, which superficially appears similar to Lemma~\ref{lem:conejoin}, however requires a rather different proof. While we will not require this result in this paper, we record it for future reference.

\begin{prop} \label{prop:joinconetype}
If $X\subseteq W$ is bounded with $y=\bigvee X$ then $T(y^{-1}) =\bigcap_{x\in X} T(x^{-1})$.
\end{prop}

\begin{proof}
The inclusion $T(y^{-1}) \subseteq\bigcap_{x\in X} T(x^{-1})$ follows from Lemma~\ref{lem:containoneway} because $x\peq y$ for all $x\in X$. Now suppose that $w \in \bigcap_{x\in W}T(x^{-1})$. Then by Proposition~\ref{prop:conetypebasics} we have $\Phi(x) \cap \Phi(w) = \emptyset$ for all $x\in X$. We claim that $\Phi(y) \cap \Phi(w) = \emptyset$. For if there exists $\beta \in \Phi(y) \cap \Phi(w)$ then by Proposition~\ref{prop:inversionsetjoin} we have $\Phi(y)=\mathrm{cone}_{\Phi}(\bigcup_{x\in X}\Phi(x))$, and so
$$
\beta = \sum c_i \beta_i
$$ 
where $\beta_i \in \Phi(x) \cup \Phi(y)$ with $x\in X$ and $c_i \ge 0$. Since $w^{-1}\beta < 0$ we have $w^{-1}\beta_i < 0$ for some $\beta_i \in\bigcup_{x\in X} \Phi(x)$ and hence $\Phi(x) \cap \Phi(w)$ is non-empty for some $x\in X$, a contradiction.
\end{proof}

%
%
%

\subsection{Automata recognising the language of reduced words}

An automaton~$\mathcal{A}$ can be viewed as a computing device for defining a language over a finite alphabet~$A$. Any string over~$A$ can be input into the automaton, which is then either accepted or rejected by $\mathcal{A}$. The set of strings accepted by $\mathcal{A}$ is the \textit{language recognised by} $\mathcal{A}$ and any language $\cL$ for which there exists a finite state automaton recognising $\cL$, is called a \textit{regular} language.

In this paper we are interested in automata recognising the language of reduced words in a Coxeter system~$(W,S)$. This allows for some minor simplifications to the general definition of an automaton, as explained below. We will work in the setting of $G$ being any group generated by a finite set~$S$. Let $\ell_S:G\to\mathbb{N}$ be the associated length function (defined as in the Coxeter group case). A word $(s_1,\cdots,s_n)\in S^n$ is \textit{reduced} if $\ell_S(s_1\cdots s_n)=n$. Let 
$\cL(G,S)$ be the set of all reduced words (the \textit{language of reduced words} in $(G,S)$). 

\begin{defn} An \textit{automaton} with \textit{alphabet $S$} is a quadruple $\cA=(Y,\mu,o,\dagger)$ where $Y$ is a set (called the \textit{state set}), $o\in Y$ is the \textit{start state}, $\dagger\notin Y$ is the \textit{dead state}, and $\mu:(Y\cup\{\dagger\})\times S\to Y\cup\{\dagger\}$ is a function (called the \textit{transition function}) such that $\mu(\dagger,s)=\dagger$ for all $s\in S$. If $|Y|<\infty$ then $\cA$ is a \textit{finite state} automaton. The \textit{language accepted by $\cA$} is the set of all words $(s_1,\ldots,s_n)$ such that $y_n\in Y$, where $y_0=o$ and $y_j=\mu(y_{j-1},s_j)$ for $1\leq j\leq n$. 
\end{defn}

We will often omit $\dagger$ from the notation, and simply give the automaton as a triple $\cA=(Y,\mu,o)$. It is helpful to think of an automaton $\cA=(Y,\mu,o)$ as a directed graph with labelled edges. The vertex set of this graph is $Y$, and if $x,y\in Y$ with $\mu(x,s)=y$ we draw an arrow from $x$ to $y$ with label $s$. Note that the dead state $\dagger$ is not drawn, and we have $\mu(x,s)=\dagger$ if and only if there is no $s$-arrow exiting the state~$x$. If $\mathcal{A}=(Y,\mu,o)$ recognises the language $\cL(G,S)$ then there exists a path in the associated directed graph starting at $o$ with edge labels $(s_1,\ldots,s_n)$ if and only if $(s_1,\ldots,s_n)$ is reduced.

%
%
%
%
%

The concepts of quotients and totally surjective morphisms are useful when comparing two automata.

\begin{defn}
Let $\mathcal{A}=(Y,\mu,o,\dagger)$ and $\mathcal{A}'=(Y',\mu',o',\dagger')$ be automata recognising $\cL(G,S)$. We say that $\cA'$ is a \textit{quotient} of $\cA$ if there exists a function $f:Y\cup\{\dagger\}\to Y'\cup\{\dagger'\}$ such that:
\begin{enumerate}
\item $f(o)=o'$, $f(\dagger)=\dagger'$, and $f(Y)=Y'$;
\item if $x,y\in Y$ with $\mu(x,s)=y$ then $\mu'(f(x),s)=f(y)$;
\item if $x',y'\in Y'$ with $\mu'(x',s)=y'$ then there exists $x,y\in Y$ with $f(x)=x'$, $f(y)=y'$, and $\mu(x,s)=y$. 
\end{enumerate}
We call such a function $f$ a \textit{totally surjective morphism} $f:\cA\to\cA'$. If, in addition, $f:Y\to Y'$ is injective then we call $f$ an \textit{isomorphism}, and we say that $\cA$ and $\cA'$ are \textit{isomorphic}, and write $\cA\cong \cA'$. 
\end{defn}
More intuitively, condition (2) says that if $x\to_s y$ is a transition in $\cA$ then $f(x)\to_s f(y)$ is a transition in $\cA'$, and condition (3) says that every transition $x'\to_s y'$ in $\cA'$ is the image under $f$ of some transition $x\to_s y$ in $\cA$.

The cone type of $g\in G$ is $T(g)=\{h\in G\mid \ell_S(gh)=\ell_S(g)+\ell_S(h)\}$, and we write $\mathbb{T}(G,S)$ for the set of all cone types.

\begin{lem}\label{lem:stateconetype}
Let $\cA=(Y,\mu,o)$ be an automaton recognising $\cL(G,S)$. If $(s_1,\ldots,s_n)$ and $(s_1',\ldots,s_m')$ are reduced words such that the corresponding paths in the automaton end at the same state, then $T(s_1\dots s_n)=T(s_1'\cdots s_m')$.
\end{lem}

\begin{proof}
Let $(t_1,\ldots,t_k)$ be a word, with $t_1,\ldots,t_k\in S$. Since the paths in $\cA$ with edge labels $(s_1,\ldots,s_n)$ and $(s_1',\ldots,s_m')$ both end at the same state, and since $\cA$ recognises the language $\cL(G,S)$, we have that the word $(s_1,\ldots,s_n,t_1,\ldots,t_k)$ is accepted if and only if the word $(s_1',\ldots,s_m',t_1,\ldots,t_k)$ is accepted. Hence the result.
\end{proof}

The following theorem is essentially the Myhill-Nerode Theorem (see \cite[Theorem~1.2.9]{Eps:92}). We sketch a proof in our context. 

\begin{thm}\label{thm:MyhillNerode}
Let $G$ be a group generated by a finite set~$S$. Let $\cA(G,S)=(\mathbb{T}(G,S),\mu,T(e))$, where $\mu$ is given by (for $T\in \mathbb{T}(G,S)$ and $s\in S$)
$$
\mu(T,s)=\begin{cases}
T(gs)&\text{if $s\in T$ and $g\in G$ is such that $T=T(g)$}\\
\dagger&\text{if $s\notin T$}.
\end{cases}
$$
Then 
\begin{enumerate}
\item $\cA(G,S)$ is an automaton recognising $\cL(G,S)$;
\item $\cA(G,S)$ is a quotient of every automaton recognising $\cL(G,S)$;
\item $\cL(G,S)$ is regular if and only if $|\mathbb{T}(G,S)|<\infty$;
\item if $\cL(G,S)$ is regular then $\cA(G,S)$ is the unique minimal (with respect to the number of states) automaton up to isomorphism recognising $\cL(G,S)$.
\end{enumerate}
\end{thm}

\begin{proof}
It is elementary to check that if $s\in T$ and $g,g'\in G$ with $T=T(g)=T(g')$ then $T(gs)=T(g's)$, and hence $\mu$ is well defined. The proof of (1) is a simple induction on the length of the word. 

To prove (2), by Lemma~\ref{lem:stateconetype} if $(s_1,\ldots,s_n)$ and $(s_1',\ldots,s_m')$ are reduced words such that the corresponding paths in the automaton end at the same state~$y\in Y'$, then $T(s_1\dots s_n)=T(s_1'\cdots s_m')$. Thus we can define a function $f:Y'\cup\{\dagger'\}\to\mathbb{T}(G,S)\cup\{\dagger\}$ by setting $f(\dagger')=\dagger$ and $f(y)=T(s_1\cdots s_n)$, and it is straightforward to check that $f$ is a totally surjective morphism. Then (3) follows from (1) and (2) and the definition of regular languages. 

(4) If $\cL(G,S)$ is regular then $\cA(G,S)$ has finitely many states (by (3)), and for any finite state automata $\cA'=(Y',\mu',o')$ recognising $\cL(G,S)$ we have $|\mathbb{T}(G,S)|\leq |Y'|$ by (2). If equality holds then the totally surjective morphism from $\cA'$ is injective, and hence an isomorphism. 
\end{proof}

\begin{defn}
We refer to the automaton $\cA(G,S)$ constructed in Theorem~\ref{thm:MyhillNerode} as the \textit{cone type automaton}. 
\end{defn}

\subsection{Examples of automata recognising $\cL(W,S)$}

We now recall examples from the literature of finite state automata recognising the language $\cL(W,S)$ of reduced words in a Coxeter group. 

The first construction of a finite state automaton recognising $\cL(W,S)$ was given by Brink and Howlett in~\cite[Section 3]{BH:93} using elementary inversion sets (in fact, the automaton in~\cite{BH:93}  recognises the language of lexicographically minimal reduced words in $(W,S)$). This concept was extended by Hohlweg, Nadeau, and Williams in~\cite{HNW:16}, leading to the following construction.

\begin{thm}\label{thm:canonicalautomaton} (see \cite{BH:93}, \cite{Eri:94a} and \cite[Section~3.4]{HNW:16})
For $n\in\mathbb{N}$ let $\cA_{n}=(\mathbb{E}_n,\mu,\emptyset)$, where, for $A\in\mathbb{E}_n$,
$$
\mu(A,s)=\begin{cases}
(\{\alpha_s\}\cup sA)\cap \cE_n&\text{if $\alpha_s\notin A$}\\
\dagger&\text{if $\alpha_s\in A$}
\end{cases}
$$
Then $\cA_n$ is a finite state automaton recognising~$\cL(W,S)$. 
\end{thm}

The automaton $\cA_n$ is called the \textit{$n$-canonical automaton}. We sometimes call $\cA_0$ the \textit{Brink-Howlett automaton}. By Lemma~\ref{lem:elementaryrootsbasics} the transition function of $\cA_n$ is given by $\mu(\cE_n(w),s)=\cE_n(sw)$ whenever $\ell(sw)>\ell(w)$.

\begin{cor}\label{cor:finiteconetypes1}
Each finitely generated Coxeter system has finitely many cone types.
\end{cor}

\begin{proof}
The fact that $\cA_0$ is finite state (as $|\cE|<\infty$) implies, by Theorem~\ref{thm:MyhillNerode}, that the cone type automaton $\cA(W,S)$ is also finite state.
\end{proof}

To each Garside shadow~$B$ there is an associated automaton $\cA_B=(B,\mu,e)$ (finite state if $|B|<\infty$) defined as follows.

\begin{thm}\label{thm:garsideautomaton} \cite[Theorem~1.2]{HNW:16} Let $B$ be a Garside shadow, and let $\cA_B=(B,\mu,e)$ where
$$
\mu(w,s)=\begin{cases}
\pi_B(sw)&\text{if $s\notin D_L(w)$}\\
\dagger&\text{if $s\in D_L(w)$.}
\end{cases}
$$
Then $\cA_B$ is an automaton recognising $\cL(W,S)$. 
\end{thm}

It is conjectured by Hohlweg, Nadeau and Williams~\cite[Conjecture~1]{HNW:16} that the automaton $\cA_{\Gmin}$ (where $\Gmin$ is the smallest Garside shadow) is the minimal automaton recognising $\cL(W,S)$ (and hence isomorphic to the cone type automaton $\cA(W,S)$). 

By \cite[Corollary~1.7]{Dye:21} the set $L_n$ of $n$-low elements forms a finite Garside shadow, and hence $\cA_{L_n}=(L_n,\mu,e)$ is a finite state automaton recognising the language of reduced words in~$W$. It is conjectured by Dyer and Hohlweg~\cite[Conjecture~2]{DH:16} that the map $\Theta_n:L_n\to \mathbb{E}_n$ with $\Theta_n(w)=\cE_n(w)$ is a bijection. If $(W,S)$ is such that $\Theta_n$ is a bijection, then it follows that $\cA_n\cong \cA_{L_n}$ (see also the discussion in \cite[\S3.6]{HNW:16}).

We note that the results of \cite{DH:16}, \cite{HNW:16} and \cite{PY:19} imply the following, confirming~\cite[Conjecture~1]{HNW:16} in the case that $\cE=\Phisph^+$, and~\cite[Conjecture~2]{DH:16} in the case that $\cE=\Phisph^+$ and $n=0$.

\begin{thm}\label{thm:conjsspherical}
If $\cE=\Phisph^+$ then the automaton $\cA_{\Gmin}$ is minimal, and the map $\Theta_0:L\to\mathbb{E}$ with $\Theta_0(x)=\cE(x)$ is bijective. 
\end{thm}

\begin{proof}
Since $L$ is a Garside shadow we have $\Gmin\subseteq L$, and by \cite[Proposition~3.26]{DH:16} we have $|L|\leq|\mathbb{E}|$, and so $|\Gmin|\leq|L|\leq|\mathbb{E}|$. On the other hand $|\mathbb{E}|$ is the number of states of the Brink-Howlett automaton $\cA_0$, and by \cite[Theorem~1]{PY:19} this automaton is minimal if and only if $\cE=\Phisph^+$. Thus $|\mathbb{E}|\leq|\Gmin|$ (as $\cA_{\Gmin}$ recognises $\cL(W,S)$ by \cite[Theorem~1.2]{HNW:16}), and hence $\Gmin=L$ and $|L|=|\mathbb{E}|$. 
\end{proof}


\section{Cone types in Coxeter groups}\label{sec:conetypes}

In this section we study cone types in Coxeter groups. In Section~\ref{sec:2:1} we give an explicit formula for the transitions between cone types in $\cA(W,S)$. In Section~\ref{sec:boundaryroots} we define the \textit{boundary roots} of a cone type, and show that these roots form the minimal set of roots required to express a cone type as an intersection of half-spaces. In Section~\ref{sec:2:3} we consider the connection between containment of cone types and the property $x\peq y$, collecting some results that will be useful later in the paper.

\subsection{Transitions in the cone type automata}\label{sec:2:1} The construction of the transition function in the cone type automata $\cA(W,S)$ in Theorem~\ref{thm:MyhillNerode} requires one to choose cone type representatives (however, ultimately, is independent of these choices). The following lemma can be used to remove these choices (see Corollary~\ref{cor:transitions} below), and gives an iterative method of computing cone types. 

\begin{lem} \label{lem:conetypeevolution}
Let $T\in\mathbb{T}$ and $s\in S$, and suppose that $s\in T$. Then the set
$$
T'=s\{w\in T\mid \ell(sw)=\ell(w)-1\}=s(T\backslash H_{\alpha_s}^+)
$$
is a cone type. Moreover, if $T=T(v)$ then $\ell(vs)=\ell(v)+1$ and $T'=T(vs)$. 
\end{lem}
\begin{proof}
Let $v\in W$ be such that $T=T(v)$. Since $s\in T$ we have $\ell(vs)=\ell(v)+1$. We claim that 
$
T(vs)=s\{w\in T\mid \ell(sw)=\ell(w)-1\},
$
from which the result follows. 

If $u\in T(vs)$ then $\ell(vsu)=\ell(vs)+\ell(u)=\ell(v)+\ell(u)+1$, and it follows that $\ell(vsu)=\ell(v)+\ell(su)$ and $\ell(su)=\ell(u)+1$. Thus the element $w=su$ satisfies $w\in T(v)=T$ and $\ell(sw)=\ell(w)-1$. Conversely, suppose that $w\in T=T(v)$ and $\ell(sw)=\ell(w)-1$. Then 
$$
\ell((vs)(sw))=\ell(vw)=\ell(v)+\ell(w)=\ell(v)+\ell(sw)+1=\ell(vs)+\ell(sw),
$$
and so $sw\in T(vs)$. 
\end{proof}

Lemma~\ref{lem:conetypeevolution} gives a geometric description of the transition function in the automaton $\cA(W,S)$.

\begin{cor}\label{cor:transitions}
The transition function of $\cA(W,S)=(\mathbb{T},\mu,T(e))$ is given by 
$$
\mu(T,s)=\begin{cases}
s(T\backslash H_{\alpha_s}^+)&\text{if $s\in T$}\\
\dagger&\text{if $s\notin T$}
\end{cases}
$$
\end{cor}



\begin{exa}
Figure~\ref{fig:evolution} illustrates the transitions $T(e)\to_s T(s)\to_t T(st)\to_u T(stu)$ in the cone type automaton of the rank~$3$ Coxeter group with $m_{s,t}=4$, $m_{t,u}=3$ and $m_{s,u}=3$.
\begin{figure}[H]
\centering
\subfigure[The cone type $T(e)$]{
\includegraphics[totalheight=5cm]{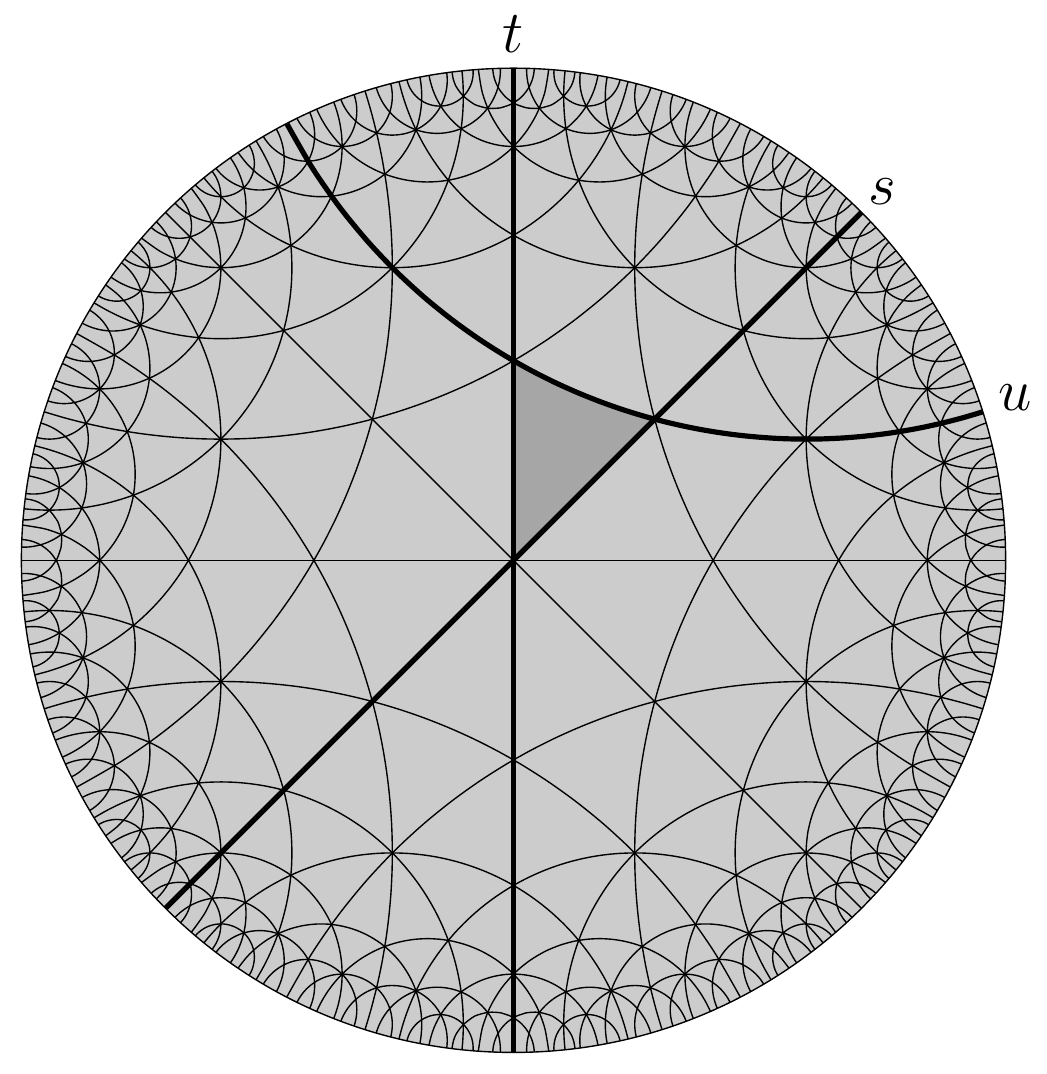}}\hspace{0.5cm}
\subfigure[The cone type $T(s)$]{
\includegraphics[totalheight=5cm]{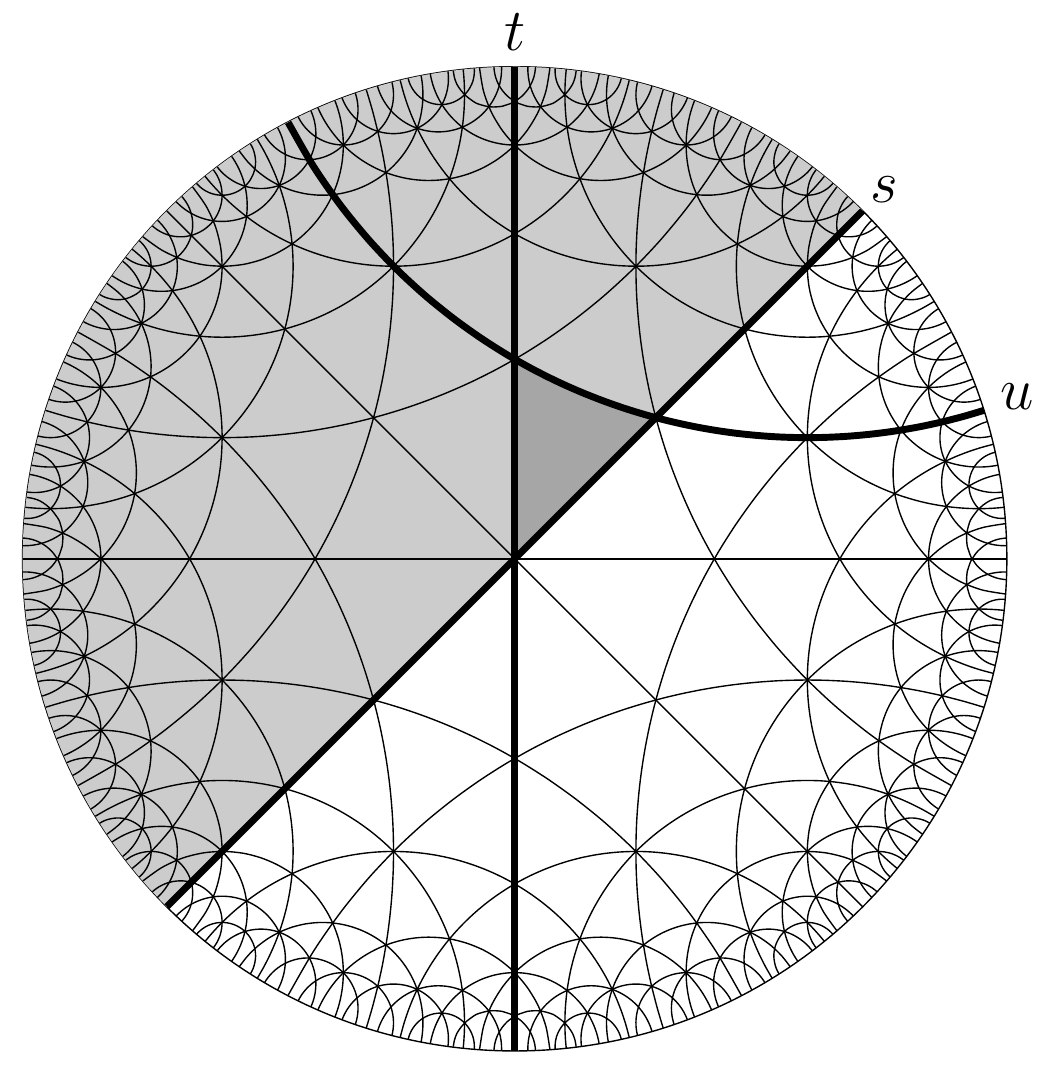}}
\subfigure[The cone type $T(st)$]{
\includegraphics[totalheight=5cm]{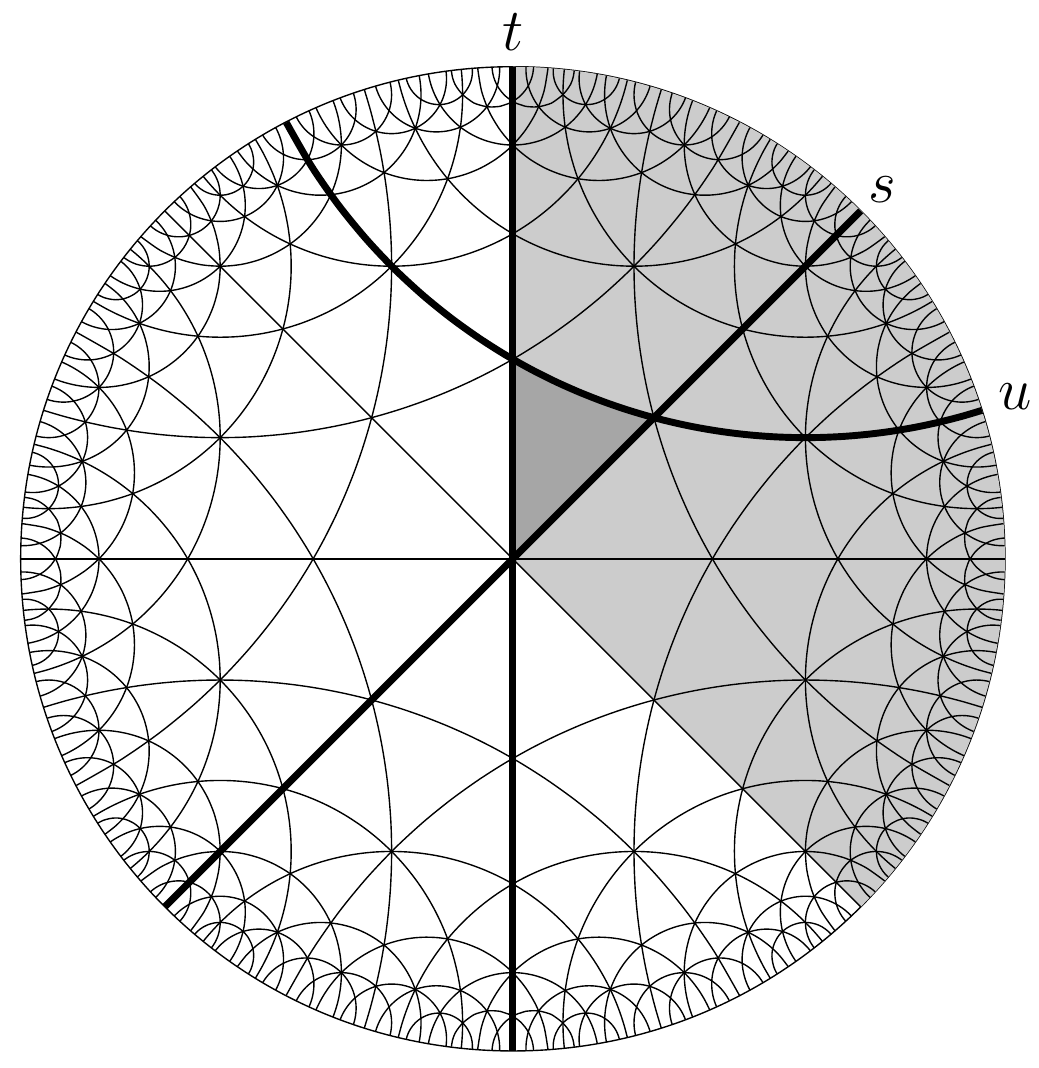}}\hspace{0.5cm}
\subfigure[The cone type $T(stu)$]{
\includegraphics[totalheight=5cm]{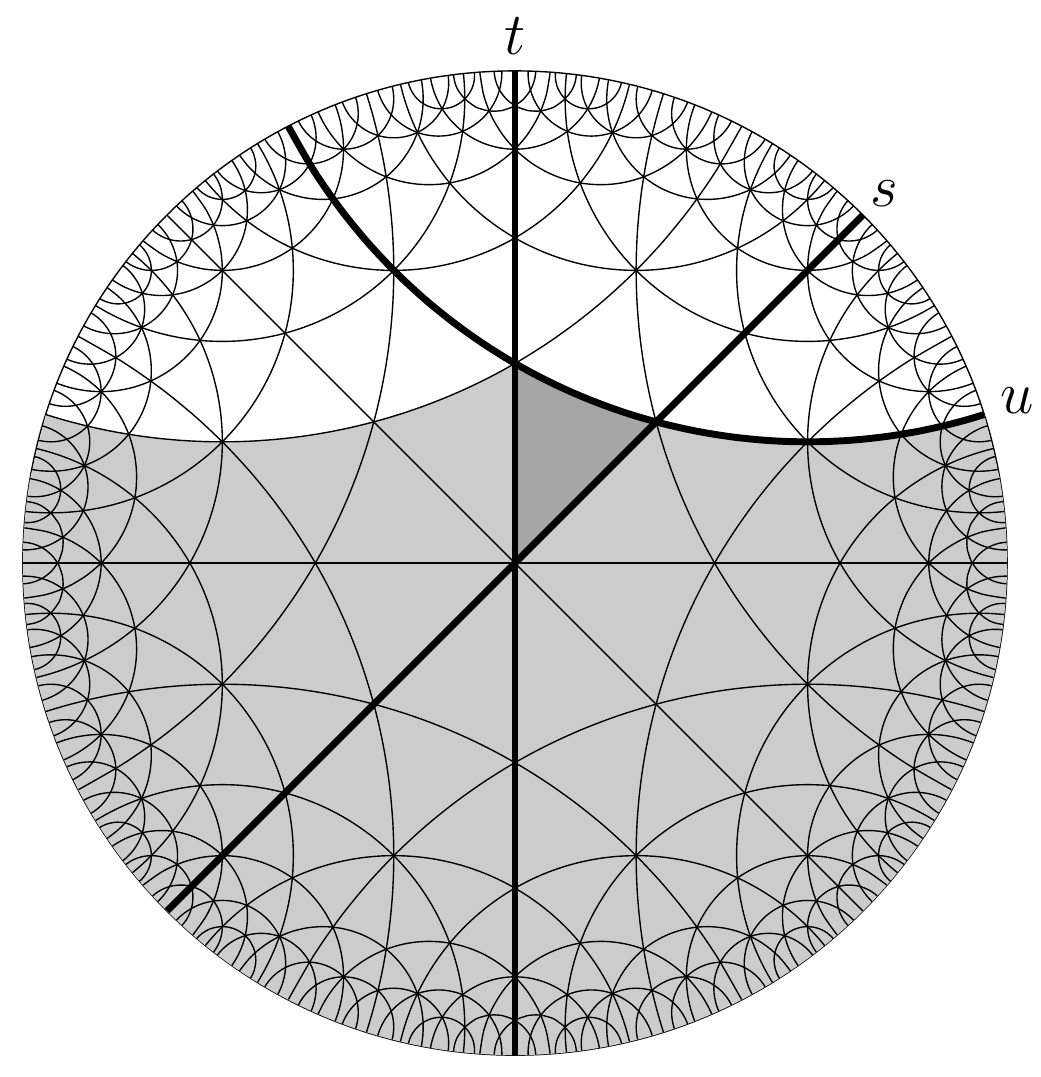}}
\caption{Transitions in $\cA(W,S)$}\label{fig:evolution}
\end{figure}
\end{exa}

\subsection{Boundary roots}\label{sec:boundaryroots}

In this section we give a geometric description of cone types, proving a more precise, and indeed optimal,  version of  Theorem~\ref{thm:geometry1}. The main result of this section is Theorem~\ref{thm:geometry2}, giving a formula for a cone type in terms of a minimal set of ``boundary roots'' of the cone type. One interesting consequence this formula is another proof of Corollary~\ref{cor:finiteconetypes1} (the finiteness of $\mathbb{T}$) without directly appealing to automata theory, however the finiteness of $\cE$ is still required in the proof.

\begin{lem}\label{lem:EandPhi1}
Let $x,y\in W$ and $\beta\in\Phi^+$. Suppose that $\Phi(x)\cap\Phi(y)=\{\beta\}$. Then:
\begin{enumerate}
\item $\beta\in\cE$, and
\item $\beta\in\Phi^1(x)\cap\Phi^1(y)$. 
\end{enumerate}
\end{lem}

\begin{proof}
(1) If $\beta \notin \cE$ then $\beta$ dominates some root $\alpha \in \Phi^{+}$ with $\alpha\neq \beta$. Since $x^{-1}\beta < 0$ and $y^{-1}\beta<0$ we have $x^{-1}\alpha < 0$ and $y^{-1}\alpha<0$ (by the definition of dominance) and hence $\alpha\in\Phi(x)\cap\Phi(y)$, a contradiction.

(2) Let $y = s_1 \cdots s_n$ be a reduced expression. Since $\beta\in\Phi(y)$ we have $\beta = s_1 \cdots s_{j-1}(\alpha_{s_j})$ for some $1\leq j\leq n$ (see Proposition~\ref{prop:rootsystembasics}). Let $y' = s_1 \cdots s_j$. Then $s_{\beta}y' = y' s_j$ with $\ell(s_{\beta}y') = \ell(y') -1$. We have
$$
    \ell(x^{-1}s_{\beta}y') =\ell((s_{\beta}x)^{-1}y')\le \ell(s_{\beta}x) + \ell(y').
$$
Since $\Phi(s_{\beta}y')=\Phi(y's_j)=\Phi(y')\backslash \{\beta\}$ we have $\Phi(x)\cap\Phi(s_{\beta}y')=\emptyset$, and so by Proposition~\ref{prop:conetypebasics}
$$
    \ell(x^{-1}s_{\beta}y') = \ell(x) + \ell(s_{\beta}y') = \ell(x) + \ell(y') - 1.
$$
Thus $\ell(s_{\beta}x)\geq \ell(x^{-1}s_{\beta}y')-\ell(y')=\ell(x)-1$. 

On the other hand, $\ell(s_{\beta}x)\leq \ell(x)-1$ because $\beta\in\Phi(x)$ (see Proposition~\ref{prop:rootsystembasics}). Thus $\ell(s_{\beta}x)=\ell(x)-1$ and so $\beta\in\Phi^1(x)$. 
Interchanging the roles of $x$ and $y$ shows that $\beta \in \Phi^1(y)$ too.
\end{proof}

\begin{defn} \label{def:boundary_roots}
Let $T$ be a cone type. The \textit{boundary roots} of $T$ are the roots $\beta \in \Phi^+$ such that there exists $w \in W$ and $s \in S$ with $w \notin T$ and $s_{\beta}w = ws \in T$. Let $\partial T$ be the set of all boundary roots of~$T$. 
\end{defn}

The conditions $w\notin T$ and $ws\in T$ in Definition~\ref{def:boundary_roots} force $\ell(ws)=\ell(w)-1$ because if $T=T(x)$, we have $\ell(xw)<\ell(x)+\ell(w)$ and $\ell(xws)=\ell(x)+\ell(ws)$, and so $\ell(ws)=\ell(xws)-\ell(x)\leq \ell(xw)+1-\ell(x)<\ell(w)+1$.

In terms of the simplicial structure of the Coxeter complex, the roots $\beta\in\partial T$ are the roots $\beta\in\Phi^+$ such that the wall $H_{\beta}$ bounds~$T$. To understand this interpretation, note that the chambers $wC_0$ and $wsC_0$ are adjacent in the Coxeter complex, and the panel (codimension~$1$ simplex) $\pi=wC_0\cap wsC_0$ lies on the wall $H_{\beta}$, and separates the chamber $wsC_0$ (which is contained in $T$) from the chamber $wC_0$ (which is not contained in $T$). For example, in Figure~\ref{fig:G2examplecone} the walls $H_{\beta}$ with $\beta\in\partial T$ are the three walls bounding~$T$. 

We will formalise the above interpretation in Theorem~\ref{thm:geometry2}. We first develop some important properties of the boundary roots. 

%
%

\begin{thm} \label{thm:boundaryroots}
Let $T$ be a cone type. If $T=T(x^{-1})$ then $\beta\in\partial T$ if and only if there exists $w\in W$ with
$$
\Phi(x)\cap\Phi(w)=\{\beta\}.
$$
Moreover if $\beta\in\partial T$ then there exists $w\in W$, independent of $x$, such that $\Phi(x)\cap\Phi(w)=\{\beta\}$ whenever $T=T(x^{-1})$. 
\end{thm}

\begin{proof}
Let $\beta\in\partial T$. Thus there exists $w \in W$ and $s \in S$ with $w \notin T$ and $s_{\beta} w = ws \in T$, and necessarily $\ell(ws)=\ell(w)-1$. Let $x \in W$ with $T(x^{-1}) = T$. Since $w\notin T$ and $ws\in T$ we have $\Phi(x) \cap \Phi(w) \neq \emptyset$ and $\Phi(x) \cap \Phi(ws) = \emptyset$ (by Proposition~\ref{prop:conetypebasics}). Since $\ell(ws)=\ell(w)-1$ and $s_{\beta}w=ws$ we have $\Phi(ws)=\Phi(w)\backslash\{-w\alpha_s\}$ and $\beta=-w\alpha_s$, and thus $\Phi(x)\cap\Phi(w)=\{\beta\}$ (with $w$ independent of the particular $x\in W$ with $T=T(x^{-1})$). 

Suppose that $T=T(x^{-1})$ and that there is $w \in W$ with $\Phi(x) \cap \Phi(w) = \{ \beta \}$. Let $w = s_1 \cdots s_n$ be a reduced expression, and let $1 \le j \le n$ be such that $\beta = s_1 \cdots s_{j-1}\alpha_{s_j}$ (by Proposition~\ref{prop:rootsystembasics}). Let $v = s_1 \cdots s_{j}$. Then $vs_j = s_{\beta} v$ and we have $\Phi(x) \cap \Phi(v) = \{ \beta \}$ and $\Phi(x) \cap \Phi(vs_j) = \emptyset$. Thus $v\notin T$ and $vs\in T$ (by Proposition~\ref{prop:conetypebasics}), and so $\beta \in \partial T$.
\end{proof}

\newpage

We immediately have the following corollary:

\begin{cor}\label{cor:boundaryroots}
Let $T$ be a cone type, with $T=T(x^{-1})$. Then $|\partial T|<\infty$ and $\partial T\subseteq \Phi^1(x)\cap\cE(x)$. In particular, all boundary roots of $T$ are elementary. 
\end{cor}

\begin{proof}
We have $|\partial T|<\infty$ by Theorem~\ref{thm:boundaryroots}, as $\partial T\subseteq \Phi(x)$ whenever $T=T(x^{-1})$. If $\beta\in\partial T$ and $T=T(x^{-1})$ then by Theorem~\ref{thm:boundaryroots} there exists $w\in W$ with $\{\beta\}=\Phi(x)\cap \Phi(w)$. The result follows from Lemma~\ref{lem:EandPhi1}.
\end{proof}

The main theorem of this section is as follows.

\begin{thm}
\label{thm:geometry2}
If $T$ is a cone type then
$$
    T = \bigcap_{\beta\in\partial T} H_{\beta}^+.
$$
Moreover, no root can be removed from this intersection (in the sense that if a root is omitted then the equality becomes strict containment). 
\end{thm}

\begin{proof}
Let $x$ be any element with $T(x^{-1}) = T$. By Corollary~\ref{cor:boundaryroots} we have $\partial T \subseteq \Phi(x)$ and thus by Theorem~\ref{thm:geometry1} we have
\begin{align} \label{eq:contain}
    T=T(x^{-1})  = \bigcap_{\beta \in \Phi(x)} H_{\beta}^{+} \subseteq \bigcap_{\beta \in \partial T} H_{\beta}^{+}
\end{align}
Now suppose that there exists
$$
    v \in \big( \bigcap_{\beta \in \partial T} H_{\beta}^{+} \big) \backslash T.
$$
Let $v = s_1 \cdots s_n$ be a reduced expression. Since $v \notin T$ we have $\ell(x^{-1}v) < \ell(x^{-1}) + \ell(v)$ and so there is an index $1 \le j < n$ such that $\ell(x^{-1}s_1 \ldots s_j) = \ell(x^{-1}) + j$ and $\ell(x^{-1}s_1 \ldots s_{j+1}) < \ell(x^{-1}) + j+1$. Let $w = s_1 \cdots s_{j+1}$ and let $s = s_{j+1}$. Then $w \notin T$ and $ws \in T$ and writing $\beta = -w\alpha_s > 0$ we have $s_{\beta} w = ws$. Thus $\beta \in \partial T$. Since $\beta\in\Phi(w)$ (as $w^{-1}\beta=-\alpha_s<0$) and $\Phi(w)\subseteq \Phi(v)$ (as $w\peq v$) we have $v\in H_{\beta}^-$, a contradiction. Thus equality holds in~(\ref{eq:contain}).

Now suppose $\beta_0 \in \partial T$ is omitted from the intersection, and let $ w \in W$, $s \in S$ be such that $w \notin T$ and $s_{\beta_0}w = ws \in T$. Since $\Phi(ws) = \Phi(w) \backslash \{ \beta_0 \}$ and $ws \in T$ we have 
$$
    w \in \bigcap_{\beta \in \partial T \setminus \{ \beta_0 \}} H_{\beta}^{+},
$$
and so the right hand side strictly contains $T$ (as $w\notin T$).
\end{proof}

While Theorem~\ref{thm:geometry2} gives the most precise formula for the cone type (that is, with no redundancies in the intersection), the following corollary collects various other useful formulae for the cone type.

\begin{cor} \label{cor:geometry3}
Let $T$ be a cone type. If $T=T(x^{-1})$ then
$$
    T = \bigcap_{\Phi(x)} H_{\beta}^+ = \bigcap_{\cE(x)}H_{\beta}^+ = \bigcap_{\Phi^1(x)}H_{\beta}^+ = \bigcap_{\Phi^1(x) \cap \cE(x)} H_{\beta}^+ = \bigcap_{\partial T} H_{\beta}^+.
$$
\end{cor}
\begin{proof}
By Corollary~\ref{cor:boundaryroots} we have $\partial T \subseteq \Phi^1(x) \subseteq \Phi(x)$ and hence by Theorems~\ref{thm:geometry1} and~\ref{thm:geometry2} we have
$$
    T = \bigcap_{\Phi(x)} H_{\beta}^+ \subseteq \bigcap_{\Phi^1(x)} H_{\beta}^+ \subseteq \bigcap_{\partial T} H_{\beta}^+ = T,
$$
and so equality holds throughout. By Corollary~\ref{cor:boundaryroots} we also have $\partial T \subseteq \cE(x)$, and so
$$
    T = \bigcap_{\Phi(x)} H_{\beta}^+ \subseteq \bigcap_{\cE(x)} H_{\beta}^+ \subseteq \bigcap_{\partial T} H_{\beta}^+ = T,
$$
and so equality holds throughout.
\end{proof}

\newpage

The formulae in Corollary~\ref{cor:geometry3} give another proof of Corollary~\ref{cor:finiteconetypes1}, independent of automata theory, as follows.

\begin{cor}\label{cor:finiteconetypes2}
Each finitely generated Coxeter system has finitely many cone types.
\end{cor}

\begin{proof}
The result follows from the formula 
$$
T(x^{-1})=\bigcap_{\beta\in\cE(x)}H_{\beta}^+
$$
and the fact that there are only finitely many elementary roots.
\end{proof}

In Proposition~\ref{prop:conetypebasics} we listed some equivalences to the statement $y\in T(x^{-1})$. We now record some further equivalences. 

\begin{cor} \label{cor:conetypeequivalences}
Let $x, y \in W$. The following are equivalent.
\begin{enumerate}
\item $y\in T(x^{-1})$;
\item $\cE(x)\cap\cE(y)=\emptyset$;
\item $\Phi^1(x)\cap\Phi^1(y)=\emptyset$;
\item $\partial T(x^{-1})\cap \Phi(y)=\emptyset$.
\end{enumerate}
\end{cor}

\begin{proof}
Using the formulae in Corollary~\ref{cor:geometry3} we have $y\in T(x^{-1})$ if and only if $y\in H_{\beta}^+$ for all $\beta\in\cE(x)$, if and only if $\ell(s_{\beta}y)>\ell(y)$ for all $\beta\in\cE(x)$, if and only if $\beta\notin \Phi(y)$ for all $\beta\in\cE(x)$. Thus (1) and (2) are equivalent. Similarly (1) and (3) are equivalent, and (1) and (4) are equivalent.
\end{proof}

\begin{rem}
The following example shows that each formula in Corollary~\ref{cor:geometry3}, except for the boundary root formula, may have redundancies. Consider $(W,S)$ of type $\tilde{\sB}_2$, with $m_{s,t}=4$, $m_{t,u}=4$, and $m_{s,u}=2$. Consider the element $x=tus$. Let $T=T(x^{-1})$. We have $\partial T=\{\alpha_s,\alpha_u\}$, while $\Phi^1(x)=\cE(x)=\{\alpha_s,\alpha_u,su\alpha_t\}$ (see Figure~\ref{fig:B2partitions}).
\end{rem}

Later in this paper we will be interested in the sets 
$$
X_T=\{x\in W\mid T(x^{-1})=T\},\quad\text{for $T\in\mathbb{T}$}.
$$
To obtain a formula for $X_T$ as an intersection of half-spaces, we introduce the \textit{internal roots} of a cone type.

\begin{defn}
Let $T$ be a cone type. A root $\beta\in\Phi^+$ is an \textit{internal root} of $T$ if there exists $w\in T$ with $\beta\in\Phi(w)$. Let $\intT$ denote the set of all internal roots of $T$. Thus $\intT=\bigcup_{w\in T}\Phi(w)$. 
\end{defn}

Geometrically, $\intT$ is the set of roots $\beta\in\Phi^+$ such that the wall $H_{\beta}$ separates two elements of $T$. To see this, note that if $\beta\in\intT$ then $\beta\in \Phi(w)$ for some $w\in T$, and so $\beta$ separates $e\in T$ and $w\in T$. Conversely, if $\beta\in\Phi^+$ and $H_{\beta}$ separates elements $w,v\in T$ then we may assume $w\in H_{\beta}^-$ (and then $v\in H_{\beta}^+$) and so $\beta\in \Phi(w)$.

\begin{thm}\label{thm:conetypeprojection}
For $T\in\mathbb{T}$ we have
$$
X_T=\bigg(\bigcap_{\beta\in\partial T}H_{\beta}^-\bigg)\cap\bigg(\bigcap_{\beta\in\intT}H_{\beta}^+\bigg).
$$
\end{thm}

\begin{proof}
Let $Y$ denote the right hand side of the equation in the statement of the theorem. Suppose that $x\in X_T$. Thus $T(x^{-1})=T$, and so $\partial T\subseteq\Phi(x)$ (by Corollary~\ref{cor:boundaryroots}) and so $x\in H_{\beta}^-$ for all $\beta\in\partial T$. If $\beta\in \intT$ then $\beta\in\Phi(w)$ for some $w\in T$, and since $\Phi(x)\cap\Phi(w)=\emptyset$ (by Proposition~\ref{prop:conetypebasics}) we have $\beta\notin\Phi(x)$, and so $x\in H_{\beta}^+$ for all $\beta\in\intT$. Hence $X_T\subseteq Y$.

Conversely, suppose that $y\in Y$. We claim that $T(y^{-1})=T$. On the one hand, if there exists $w\in T$ with $w\notin T(y^{-1})$ then $\Phi(w)\cap\Phi(y)\neq\emptyset$, and for any $\beta\in\Phi(w)\cap\Phi(y)$ we have $\beta\in\intT$ (as $\beta\in\Phi(w)$ and $w\in T$) and $y\in H_{\beta}^-$ (as $\beta\in\Phi(y)$), a contradiction. Thus $T\subseteq T(y^{-1})$. On the other hand, if $w\notin T$ then by Theorem~\ref{thm:geometry2} there is $\beta\in\partial T$ with $w\in H_{\beta}^-$, and so $\beta\in\partial T\cap \Phi(w)$. Since $y\in Y$ we have $y\in H_{\beta}^-$, and so $\beta\in\Phi(y)$. Thus $\Phi(y)\cap\Phi(w)\neq\emptyset$, and so $w\notin T(y^{-1})$. Thus $T(y^{-1})\subseteq T$, completing the proof.
\end{proof}

Note that the intersection in Theorem~\ref{thm:conetypeprojection} may be over an infinite set of roots, as $\intT$ may be infinite. We will show in Corollary~\ref{cor:finiteintersection} that in fact each set $X_T$ can be expressed as an intersection of finitely many half-spaces. 

\subsection{On containment of cone types}\label{sec:2:3}

In this section we consider the connection between containment of cone types $T(y^{-1})\subseteq T(x^{-1})$ and the property $x\peq y$. We saw in Lemma~\ref{lem:containoneway} that if $x\peq y$ then $T(y^{-1})\subseteq T(x^{-1})$. The converse implication is obviously false in general. For example, if $x$ and $y$ are elements in the red shaded region of Figure~\ref{fig:G2examplecone}, then $T(x^{-1})=T(y^{-1})$, however of course $x\peq y$ may not occur.

However, with some constraints on~$x$, the reverse implication does hold. The following theorem shows that $x\in W_J$, with $J\subseteq S$, is a sufficient condition. Later in this paper we conjecture a generalisation of this result (see Conjecture~\ref{conj:orderisomorphism}).

\begin{thm} \label{thm:parabolicconetype}
Let $x \in W_J$ and $y \in W$ with $J\subseteq S$ spherical. Then $T(y^{-1}) \subseteq T(x^{-1})$ if and only if $x\peq y$.
\end{thm}

\begin{proof}
By Lemma~\ref{lem:containoneway}, we only need to show that if $T(y^{-1}) \subseteq T(x^{-1})$ then $x\peq y$. Hence suppose that $T(y^{-1})\subseteq T(x^{-1})$, with $x\in W_J$. Write $y=uv$ as in  (\ref{eq:WJdecomposition}), with $u\in W_J$ and $v\in W^J$.

Let $z=uw_J$, with $w_J$ the longest element of $W_J$. Since $v\in W^J$ we have
\begin{align*}
\ell(y^{-1}z)&=\ell(v^{-1}w_J)=\ell(v)+\ell(w_J)=\ell(y)-\ell(u)+\ell(w_J)=\ell(y)+\ell(z),
\end{align*}
and so $z\in T(y^{-1})$. Thus $z\in T(x^{-1})$. 

Let $w=w_Jz^{-1}x\in W_J$, and note that $xw^{-1}v=y$. We claim that 
\begin{align}\label{eq:prefix}
\ell(xw^{-1}v)=\ell(x)+\ell(w^{-1}v),
\end{align}
from which the desired result $x\peq y$ follows. To prove~(\ref{eq:prefix}), we have
\begin{align*}
\ell(xw^{-1}v)&=\ell(xw^{-1})+\ell(v)&&\text{as $v\in W^J$}\\
&=\ell(w_J)-\ell(z)+\ell(v)&&\text{as $xw^{-1}=zw_J$}\\
&=\ell(w_J)-(\ell(x^{-1}z)-\ell(x))+\ell(v)&&\text{as $z\in T(x^{-1})$}\\
&=\ell(x)+\ell(w)+\ell(v)&&\text{as $w=w_Jz^{-1}x$ and $z^{-1}x\in W_J$}\\
&=\ell(x)+\ell(w^{-1}v)&&\text{as $v\in W^J$ and $w\in W_J$},
\end{align*}
completing the proof.
\end{proof}

%
%

\begin{cor}
Let $W$ be a finite Coxeter group. Then for $x,y \in W$ we have $T(y^{-1}) \subseteq T(x^{-1})$ if and only if $x\peq y$.
\end{cor}
\begin{proof}
The result follows by letting $W = W_J$ in the statement of Theorem~\ref{thm:parabolicconetype}.
\end{proof}

\subsection{Cone types in finite Coxeter groups}

We can describe cone types in a finite Coxeter group very precisely, and this description will be useful in conjunction with Theorem~\ref{thm:parabolicconetype} in later sections. 

\begin{prop} \label{prop:paraboliccone}
Let $W$ be a finite Coxeter group and let $w_0$ be the longest element of $W$. Then for $x \in W$ we have
$$
    T(x^{-1}) = \{ w \in W \mid w \peq xw_0 \}.
    $$
In particular, for all $x,y \in W$ we have $T(x^{-1}) = T(y^{-1})$ if and only if $x = y$.
\end{prop}
\begin{proof}
If $w\in T(x^{-1})$ then $\ell(x^{-1}w)=\ell(x)+\ell(w)$, and so 
$$
\ell(w^{-1}xw_0)=\ell(w_0)-\ell(w^{-1}x)=\ell(w_0)-\ell(x)-\ell(w)=\ell(xw_0)-\ell(w),
$$
and so $w\peq xw_0$. Conversely, if $w\peq xw_0$ then
$$
\ell(w^{-1}xw_0)=\ell(xw_0)-\ell(w)=\ell(w_0)-\ell(x)-\ell(w),
$$
but also $\ell(w^{-1}xw_0)=\ell(w_0)-\ell(w^{-1}x)=\ell(w_0)-\ell(x^{-1}w)$, and so $\ell(x^{-1}w)=\ell(x)+\ell(w)$, giving $w\in T(x^{-1})$. 

In particular, if $T(x^{-1})=T(y^{-1})$ then since $xw_0\in T(x^{-1})$ we have $xw_0\peq yw_0$, and similarly $yw_0\peq xw_0$. Hence $xw_0=yw_0$ and so $x=y$. 
\end{proof}

\newpage
\begin{lem}\label{lem:sphericalprefix1}
Let $x\in W$ and $J\subseteq S$ with $J$ spherical. Write $x=uv$ with $u\in W_J$ and $v\in W^J$. If $w\in W_J$ with $w\in T(u^{-1})$ then $w\in T(x^{-1})$. 
\end{lem}

\begin{proof}
Since $u^{-1}w\in W_J$, $v\in W^J$, and $w\in T(u^{-1})$, we have
$$
\ell(x^{-1}w)=\ell(v^{-1}u^{-1}w)=\ell(v)+\ell(u^{-1}w)=\ell(v)+\ell(u)+\ell(w),
$$
and the result follows since $\ell(v)+\ell(u)=\ell(uv)=\ell(x)$. 
\end{proof}

\begin{cor} \label{cor:distinguishconetypes}
Let $x,y \in W$ and $J \subseteq S$ with $J$ spherical. Write $x = uv$ and $y=u'v'$ with $u, u' \in W_J$ and $v, v' \in W^J$. If $u \neq u'$, then $T(x^{-1}) \neq T(y^{-1})$.
\end{cor}
\begin{proof}
Since $u \neq u'$ by Proposition~\ref{prop:paraboliccone} there is $w\in W_J$ with $w \in T(u^{-1}) \setminus T(u'^{-1})$. By Lemma~\ref{lem:sphericalprefix1} we have $x \in T(w)$, and since $w \notin T(u'^{-1})$ we have $w \notin T(y^{-1})$.
\end{proof}

We note, in passing, the following corollary which shows that the minimal automata recognising the language of reduced words in a finite Coxeter group~$W$ is just the ``trivial'' automaton with states $W$ and transition function $\mu(w,s)=ws$ if $\ell(ws)=\ell(w)+1$ (note that this gives an automaton recognising $\cL(W,S)$ for all Coxeter systems, however of course it is finite state if and only if $W$ is finite). 

\begin{cor}
If $W$ is a finite Coxeter group then $|\mathcal{A}(W,S)| = |W|$.
\end{cor}
\begin{proof}
By Proposition~\ref{prop:paraboliccone} we have $|\cA(W,S)|=|\mathbb{T}| = |W|$. 
\end{proof}

\section{Regular partitions}\label{sec:regularpartitions}

In this section we introduce one of the main concepts of this paper: the notion of a ``regular partition'' of~$W$. This concept has its genesis in the Ph.D. of P. Headley in his study of the classical Shi arrangement (see \cite[Lemma~V.5]{Hea:94}). We now give an outline of the results of this section.

We begin in Section~\ref{sec:partitions} by setting up appropriate language for working with the partially ordered set of all partitions of~$W$. We then introduce certain special partitions of $W$ that will play an important role in the paper, including the \textit{cone type partition} $\scrT$, \textit{Garside shadow partitions}, and the \textit{$n$-Shi partitions} associated to $n$-elementary inversion sets. 

In Section~\ref{subsec:regularpartitions} we define the notion of a regular partition, and exhibit some of the main examples of such partitions. We show in Theorem~\ref{thm:regularautomaton} that each such partition gives rise to an automaton recognising $\cL(W,S)$, and in Theorem~\ref{thm:converseautomaton} we show that every automaton recognising $\cL(W,S)$ satisfying a mild hypothesis arises in such a way. 

In Section~\ref{sec:regularcompletion} we study the partially ordered set $\Preg$ of all regular partitions of $W$. We show in Theorem~\ref{thm:regularlattice2} that $\Preg$ is a complete lattice with bottom element being the cone type partition~$\scrT$ (note the convention~(\ref{eq:convention})). This in turn allows us to define the \textit{regular completion} $\widehat{\scrP}$ of an arbitrary partition $\scrP$ of $W$ (this is the ``minimal'' regular partition refining $\scrP$). 

In Section~\ref{sec:simplerefinements} we develop an algorithm, based on ``simple refinements'', for producing the regular completion of a partition, and provide natural sufficient conditions for this algorithm to terminate in finite time. As a consequence, we prove in Corollary~\ref{cor:regularlattice1} that $\scrT=\widehat{\scrD}$, where $\scrD$ is the partition of $W$ according to left descent sets. This characterisation of $\scrT$ will be crucial in proving the main result of this paper (Theorem~\ref{thm:main1}).

\subsection{Partitions of $W$}\label{sec:partitions}

Various partitions of $W$ play an important role in this work, and we begin by recalling some terminology. A \textit{partition} of $W$ is a set $\scrP$ of subsets of $W$ such that $\bigcup_{P\in \scrP}P=W$ and $P\cap P'=\emptyset$ for $P,P'\in\scrP$ with $P\neq P'$. The sets $P\in\scrP$ are called the \textit{parts} of the partition. Let $\scrP(W)$ denote the set of all partitions of $W$.

If $\scrP$ and $\scrP'$ are partitions of $W$ such that each part of $\scrP'$ is contained in some part of $\scrP$ then we say that $\scrP'$ is a \textit{refinement} of $\scrP$. We also say that $\scrP'$ is \textit{finer} than $\scrP$, and that $\scrP$ is \textit{coarser} than $\scrP'$. 

We write 
\begin{align}\label{eq:convention}
\scrP\leq\scrP'\quad\text{if $\scrP'$ is a refinement of $\scrP$}.
\end{align}
Note that this is dual to the standard convention. Our choice here is motivated by the fact that we are often interested in the number of parts of a partition, and a partition with few parts is best considered to be ``small''. Thus, the partially ordered set $(\scrP(W),\leq)$ has top element $\mathbf{1}=\{\{w\}\mid w\in W\}$ (the partition into singletons) and bottom element $\mathbf{0}=\{W\}$ (the partition with one part).

A \textit{covering} of $W$ is a set $\mathbb{X}$ of subsets of $W$ with $\bigcup_{X\in\mathbb{X}}X=W$. Each covering of $W$ induces a partition of $W$, as follows.

\begin{defn}\label{defn:covering}
Let $\mathbb{X}$ be a covering of $W$. Let $\sim_{\mathbb{X}}$ be the equivalence relation on $W$ given by $x\sim_{\mathbb{X}}y$ if and only if $\{X\in\mathbb{X}\mid x\in X\}=\{X\in\mathbb{X}\mid y\in X\}$ (that is, $x\in X$ if and only if $y\in X$, for $X\in\mathbb{X}$). The \textit{partition induced by $\mathbb{X}$} is the partition $\scrX$ of $W$ into $\sim_{\mathbb{X}}$ equivalence classes. Thus elements $x,y\in W$ lie in the same part of $\scrX$ if and only if they lie in precisely the same elements of $\mathbb{X}$. 
\end{defn}

Important examples of partitions are provided by hyperplane arrangements. In our general setting, a hyperplane arrangement is most appropriately thought of as a partition of $W$ induced by a set of roots, as follows. Let $\Lambda\subseteq\Phi^+$ be nonempty. The \textit{partition of $W$ induced by $\Lambda$} is the partition $\scrH(\Lambda)$ induced by the covering $\{H_{\beta}^+,H_{\beta}^-\mid \beta\in\Lambda\}$ (as in Definition~\ref{defn:covering}). We will refer to such partitions as \textit{hyperplane partitions} to emphasise this connection to traditional hyperplane arrangements. 

We now provide the main examples of partitions that will appear in this work. Recall the definition of  $C(w)$ from Section~\ref{sec:1:conetypes}, and recall that $\Pi=\{\alpha_s\mid s\in S\}$. Recall that $\mathbb{T}$ denotes the set of all cone types.

\begin{defn}\label{defn:partitions} Let $n\in\mathbb{N}$, and let $B$ be a Garside shadow. 
\begin{enumerate}
\item The \textit{cone type partition} is the partition $\scrT$ induced by the covering $\mathbb{T}$.\item The \textit{Garside partition} associated to $B$ is the partition $\scrG_B$ induced by the covering $\{C(b)\mid b\in B\}$ (this is a covering as $e\in B$). 
\item The \textit{$n$-Shi partition} is the hyperplane partition $\scrS_n=\scrH(\cE_n)$.
\item The \textit{$S$-partition} is the hyperplane partition $\scrD=\scrH(\Pi)$.
\item The \textit{spherical partition} is the hyperplane partition $\scrJ=\scrH(\Phisph)$. 
\end{enumerate}
\end{defn}

We now give a more concrete description of the parts of each of the above partitions. Recall that $\mathbb{E}_n$ denotes the set of all $n$-elementary inversion sets, and $\mathbb{S}$ denotes the set of all spherical inversion sets.

\begin{prop}\label{prop:partsdescription}
Let $B$ be a Garside shadow and let $n\in\mathbb{N}$. 
\begin{enumerate}
\item The parts of the cone type partition $\scrT$ are the sets
$$
X_T=\{w\in W\mid T(w^{-1})=T\},\quad\text{with $T\in\mathbb{T}$}.
$$ 
\item If $B$ is a Garside shadow, the parts of $\scrG_B$ are the sets
$$
\pi_B^{-1}(b)=\{w\in W\mid \pi_B(w)=b\},\quad\text{with $b\in B$}.
$$
\item The parts of the $n$-Shi partition $\scrS_n$ are the sets
$$
\{w\in W\mid \cE_n(w)=E\},\quad\text{with $E\in\mathbb{E}_n$}.
$$
\item The parts of the $S$-partition $\scrD$ are the sets
$$
D_L^{-1}(J)=\{w\in W\mid D_L(w)=J\},\quad\text{with $J\subseteq S$ spherical}.
$$
\item The parts of the spherical partition $\scrJ$ are the sets
$$
\{w\in W\mid \Phisph(w)=\Sigma\},\quad\text{for $\Sigma\in\mathbb{S}$}.
$$
\end{enumerate}
In particular, $|\scrT|=|\mathbb{T}|<\infty$, $|\scrG_B|=|B|$, $|\scrS_n|=|\mathbb{E}_n|<\infty$, and $|\scrD|,|\scrJ|<\infty$.
\end{prop}

\begin{proof}
(1) If $u\sim_{\mathbb{T}}v$ then for all $w\in W$ we have $u\in T(w^{-1})$ if and only if $v\in T(w^{-1})$. Thus, by Proposition~\ref{prop:conetypebasics}, for all $w\in W$ we have $w\in T(u^{-1})$ if and only if $w\in T(v^{-1})$, and hence $T(u^{-1})=T(v^{-1})$. So $u,v\in X_T$, where $T=T(u^{-1})$.

Conversely, suppose that $u,v\in X_T$ for some $T\in\mathbb{T}$. Thus $T(u^{-1})=T(v^{-1})=T$. If $u\in T(w^{-1})$ then by Proposition~\ref{prop:conetypebasics} we have $w\in T(u^{-1})=T(v^{-1})$, and so again by Proposition~\ref{prop:conetypebasics} we have $v\in T(w^{-1})$. Thus $u\sim_{\mathbb{T}}v$. 

Thus $|\scrT|=|\mathbb{T}|$, which is finite by Corollary~\ref{cor:finiteconetypes1}. 

(2) Let $P$ be a part of $\scrG_B$ and let $u,v\in P$. Then for all $b\in B$ we have $u\in C(b)$ if and only if $v\in C(b)$, and so $b\peq u$ if and only if $b\peq v$, and so $\pi_B(u)=\pi_B(v)$. 

Conversely, suppose that $u,v\in W$ with $\pi_B(u)=\pi_B(v)$. If $b\in B$ and $u\in C(b)$ then $b\peq \pi_B(u)=\pi_B(v)\peq v$, and hence $v\in C(b)$. 

(3) Let $P$ be a part of $\scrS_n$, and let $u,v\in P$ and $\beta\in\cE_n$. Then $\beta\in\cE_n(u)$ if and only if $\ell(s_{\beta}u)<\ell(u)$, if and only if $u\in H_{\beta}^-$, if and only if $v\in H_{\beta}^-$ (from the definition of $\scrS_n$, using $u,v\in P$), if and only if $\ell(s_{\beta}v)<\ell(v)$, if and only if $\beta\in \cE_n(v)$. Thus $\cE_n(u)=\cE_n(v)$.

Conversely if $u,v\in W$ with $\cE_n(u)=\cE_n(v)$, and if $\beta\in\cE_n$ and $\epsilon\in\{-,+\}$, then $u\in H_{\beta}^{\epsilon}$ if and only if $v\in H_{\beta}^{\epsilon}$, and so $u$ and $v$ lie in the same part of $\scrS_n$.

Thus $|\scrS_n|=|\mathbb{E}_n|$, which is finite by Corollary~\ref{cor:Enfinite}.

(4) Let $P$ be a part of $\scrD$, and let $u,v\in P$. For $s\in S$ we have $u\in H_{\alpha_s}^-$ if and only if $v\in H_{\alpha_s}^-$, and so $s\in D_L(u)$ if and only if $s\in D_L(v)$, and so $D_L(u)=D_L(v)$. 

Conversely, let $u,v\in W$ with $J=D_L(u)=D_L(v)$. If $s\in J$ then $u,v\in H_{\alpha_s}^-$, and if $s\in S\backslash J$ then $u,v\in H_{\alpha_s}^+$, and hence $u$ and $v$ are in the same part of $\scrD$. Finally, recall that descent sets are always spherical subsets (see \cite[Proposition~2.17]{AB:08}).

(5) This is similar to (3).
\end{proof}

\begin{exa}
Figure~\ref{fig:hyperbolicconetype} shows a cone type $T$ (shaded grey), and the corresponding set $X_T$ (shaded red). 
\begin{figure}[H]
\centering
\includegraphics[totalheight=8cm]{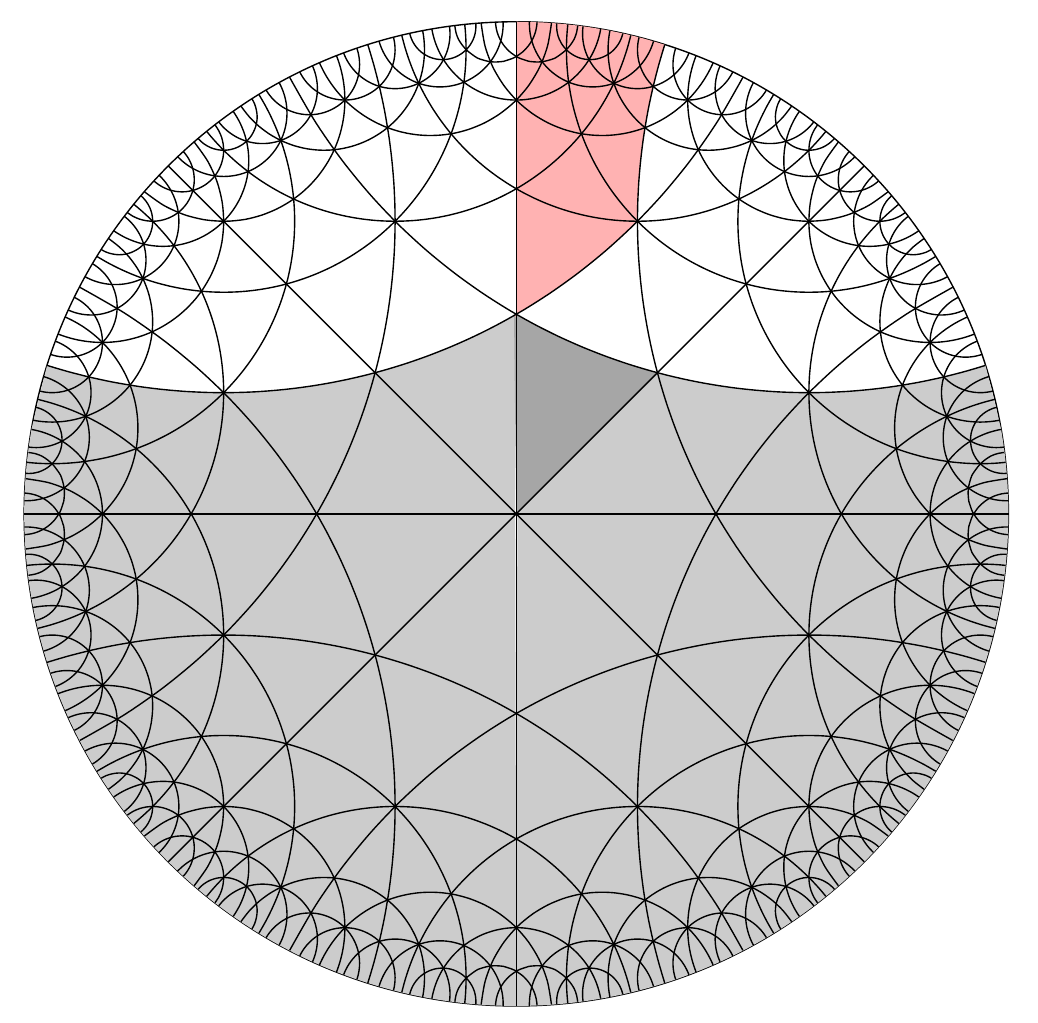}
\caption{A cone type $T$ and the corresponding part $X_T$ of $\scrT$}\label{fig:hyperbolicconetype}
\end{figure}
\end{exa}

Computing the partitions $\scrD$ and $\scrJ$ is of course trivial: geometrically the walls determining the hyperplane partition $\scrD$ are the walls bounding the fundamental chamber, and the walls determining the partition $\scrJ$ are the walls passing through a vertex of the fundamental chamber. Computing the partition $\scrS_n$ is also straightforward once the $n$-elementary roots are known. However computing the cone type partition $\scrT$ is nontrivial (see Algorithm~\ref{alg:regularisation} and Corollary~\ref{cor:regularlattice1}).

\newpage
\begin{exa}
The partitions $\scrD$, $\scrJ$, $\scrS_0$, and $\scrT$ are illustrated for $\tilde{\sG}_2$ in Figure~\ref{fig:G2partitions} (in each case the identity chamber is shaded grey, and the blue and red shaded chambers will be discussed in the following section). The partitions $\scrS_0$ and $\scrT$ for $\tilde{\sB}_2$ and $\tilde{\sA}_2$ are given in Figures~\ref{fig:B2partitions} and~\ref{fig:A2partitions}.

\begin{figure}[H]
\centering
\subfigure[The $S$-partition]{
\centering
\begin{tikzpicture}[scale=0.75]
\path [fill=blue!30] (0,0)--(-0.433,0.75)--(0,1);
\path [fill=blue!30] (0,0)--(0.866,0.5)--(0.433,0.75);
\path [fill=blue!30] (0,1)--(0.433,0.75)--(0.866,1.5);
\path [fill=blue!30] (0.433,0.75)--(0.866,0.5)--(0.866,1.5);
\path [fill=blue!30] (0,1)--(0,1.5)--(-0.866,1.5);
\path [fill=blue!30] (0,0)--(0,-1)--(-0.433,-0.75);
\path [fill=gray!90] (0,0) -- (0.433,0.75) -- (0,1) -- (0,0);
\draw(-4.33,4.5)--(4.33,4.5);
\draw(-4.33,3)--(4.33,3);
\draw(-4.33,1.5)--(4.33,1.5);
\draw(-4.33,0)--(4.33,0);
\draw(-4.33,-1.5)--(4.33,-1.5);
\draw(-4.33,-3)--(4.33,-3);
\draw(-4.33,-3)--(-4.33,4.5);
\draw(-3.464,-3)--(-3.464,4.5);
\draw(-2.598,-3)--(-2.598,4.5);
\draw(-1.732,-3)--(-1.732,4.5);
\draw(-.866,-3)--(-.866,4.5);
\draw(0,-3)--(0,4.5);
\draw(.866,-3)--(.866,4.5);
\draw(1.732,-3)--(1.732,4.5);
\draw(2.598,-3)--(2.598,4.5);
\draw(3.464,-3)--(3.464,4.5);
\draw(4.33,-3)--(4.33,4.5);
\draw(-4.33,3.5)--({-3*0.866},4.5);
\draw(-4.33,2.5)--({-1*0.866},4.5);
\draw(-4.33,1.5)--({1*0.866},4.5);
\draw(-4.33,.5)--({3*0.866},4.5);
\draw(-4.33,-.5)--(4.33,4.5);
\draw(-4.33,-1.5)--(4.33,3.5);
\draw(-4.33,-2.5)--(4.33,2.5);
\draw(-3.464,-3)--(4.33,1.5);
\draw(-1.732,-3)--(4.33,.5);
\draw(0,-3)--(4.33,-.5);
\draw(1.732,-3)--(4.33,-1.5);
\draw(3.464,-3)--(4.33,-2.5);
\draw(4.33,3.5)--({3*0.866},4.5);
\draw(4.33,2.5)--({1*0.866},4.5);
\draw(4.33,1.5)--({-1*0.866},4.5);
\draw(4.33,.5)--({-3*0.866},4.5);
\draw(4.33,-.5)--(-4.33,4.5);
\draw(4.33,-1.5)--(-4.33,3.5);
\draw(4.33,-2.5)--(-4.33,2.5);
\draw(3.464,-3)--(-4.33,1.5);
\draw(1.732,-3)--(-4.33,.5);
\draw(0,-3)--(-4.33,-.5);
\draw(-1.732,-3)--(-4.33,-1.5);
\draw(-3.464,-3)--(-4.33,-2.5);
\draw(-4.33,-1.5)--(-3.464,-3);
\draw(-4.33,1.5)--(-1.732,-3);
\draw(-4.33,4.5)--(0,-3);
\draw({-3*0.866},4.5)--(1.732,-3);
\draw({-1*0.866},4.5)--(3.464,-3);
\draw({1*0.866},4.5)--(4.33,-1.5);
\draw({3*0.866},4.5)--(4.33,1.5);
\draw(4.33,-1.5)--(3.464,-3);
\draw(4.33,1.5)--(1.732,-3);
\draw(4.33,4.5)--(0,-3);
\draw({3*0.866},4.5)--(-1.732,-3);
\draw({1*0.866},4.5)--(-3.464,-3);
\draw({-1*0.866},4.5)--(-4.33,-1.5);
\draw({-3*0.866},4.5)--(-4.33,1.5);
\draw[line width=2pt](0,-3)--(0,4.5);
\draw[line width=2pt](-1.732,-3)--({3*0.866},4.5);
\draw[line width=2pt](-4.33,3.5)--(4.33,-1.5);
\end{tikzpicture}
}\qquad 
\subfigure[The spherical partition]{
\centering
\begin{tikzpicture}[scale=0.75]
\path [fill=blue!30] (0,-1) -- (-0.866,-0.5) -- (-0.866,1.5) -- (0.866,1.5)--(0.866,0.5)--(0.866,-0.5)--(0,-1);
\path [fill=blue!30] (1.732,0)--(2.598,0)--(2.598,-0.5);
\path [fill=blue!30] (-0.866,1.5)--(-1.732,2)--(-1.299,2.25);
\path [fill=blue!30] (-1.732,0)--(-1.299,0.75)--(-0.866,0.5);
\path [fill=blue!30] (-2.598,0)--(-2.598,-0.5)--(-1.732,0);
\path [fill=blue!30] (0.866,0.5)--({0.866+0.433},0.75)--(1.732,0);
\path [fill=blue!30] (0.866,1.5)--({0.866+0.433},2.25)--(1.732,2);
\path [fill=gray!90] (0,0) -- (0.433,0.75) -- (0,1) -- (0,0);
\draw(-4.33,4.5)--(4.33,4.5);
\draw(-4.33,3)--(4.33,3);
\draw(-4.33,1.5)--(4.33,1.5);
\draw(-4.33,0)--(4.33,0);
\draw(-4.33,-1.5)--(4.33,-1.5);
\draw(-4.33,-3)--(4.33,-3);
\draw(-4.33,-3)--(-4.33,4.5);
\draw(-3.464,-3)--(-3.464,4.5);
\draw(-2.598,-3)--(-2.598,4.5);
\draw(-1.732,-3)--(-1.732,4.5);
\draw(-.866,-3)--(-.866,4.5);
\draw(0,-3)--(0,4.5);
\draw(.866,-3)--(.866,4.5);
\draw(1.732,-3)--(1.732,4.5);
\draw(2.598,-3)--(2.598,4.5);
\draw(3.464,-3)--(3.464,4.5);
\draw(4.33,-3)--(4.33,4.5);
\draw(-4.33,3.5)--({-3*0.866},4.5);
\draw(-4.33,2.5)--({-1*0.866},4.5);
\draw(-4.33,1.5)--({1*0.866},4.5);
\draw(-4.33,.5)--({3*0.866},4.5);
\draw(-4.33,-.5)--(4.33,4.5);
\draw(-4.33,-1.5)--(4.33,3.5);
\draw(-4.33,-2.5)--(4.33,2.5);
\draw(-3.464,-3)--(4.33,1.5);
\draw(-1.732,-3)--(4.33,.5);
\draw(0,-3)--(4.33,-.5);
\draw(1.732,-3)--(4.33,-1.5);
\draw(3.464,-3)--(4.33,-2.5);
\draw(4.33,3.5)--({3*0.866},4.5);
\draw(4.33,2.5)--({1*0.866},4.5);
\draw(4.33,1.5)--({-1*0.866},4.5);
\draw(4.33,.5)--({-3*0.866},4.5);
\draw(4.33,-.5)--(-4.33,4.5);
\draw(4.33,-1.5)--(-4.33,3.5);
\draw(4.33,-2.5)--(-4.33,2.5);
\draw(3.464,-3)--(-4.33,1.5);
\draw(1.732,-3)--(-4.33,.5);
\draw(0,-3)--(-4.33,-.5);
\draw(-1.732,-3)--(-4.33,-1.5);
\draw(-3.464,-3)--(-4.33,-2.5);
\draw(-4.33,-1.5)--(-3.464,-3);
\draw(-4.33,1.5)--(-1.732,-3);
\draw(-4.33,4.5)--(0,-3);
\draw({-3*0.866},4.5)--(1.732,-3);
\draw({-1*0.866},4.5)--(3.464,-3);
\draw({1*0.866},4.5)--(4.33,-1.5);
\draw({3*0.866},4.5)--(4.33,1.5);
\draw(4.33,-1.5)--(3.464,-3);
\draw(4.33,1.5)--(1.732,-3);
\draw(4.33,4.5)--(0,-3);
\draw({3*0.866},4.5)--(-1.732,-3);
\draw({1*0.866},4.5)--(-3.464,-3);
\draw({-1*0.866},4.5)--(-4.33,-1.5);
\draw({-3*0.866},4.5)--(-4.33,1.5);
\draw[line width=2pt](0,-3)--(0,4.5);
\draw[line width=2pt](-1.732,-3)--({3*0.866},4.5);
\draw[line width=2pt](1.732,-3)--({-3*0.866},4.5);
\draw[line width=2pt](-4.33,-2.5)--(4.33,2.5);
\draw[line width=2pt](-4.33,-1.5)--(4.33,3.5);
\draw[line width=2pt](-4.33,0)--(4.33,0);
\draw[line width=2pt](-4.33,2.5)--(4.33,-2.5);
\draw[line width=2pt](-4.33,3.5)--(4.33,-1.5);
\end{tikzpicture}
}
\subfigure[The $0$-Shi partition]{
\centering
\begin{tikzpicture}[scale=0.75]
\path [fill=blue!30] (0,-1) -- (-0.866,-0.5) -- (-0.866,1.5) -- (0.866,1.5)--(0.866,0.5)--(1.299,0.75)--(1.732,0)--(0.866,-0.5)--(0,-1);
\path [fill=blue!30] (0.866,-0.5)--(0.866,-1.5)--(1.299,-0.75);
\path [fill=blue!30] (0.866,1.5)--(0.4333,2.25)--(0.866,2.5)--(1.732,2)--(1.732,1.5);
\path [fill=blue!30] (0.866,4.5)--(0.866,5.5)--(1.299,5.25);
\path [fill=blue!30] (0,3)--(-0.433,3.75)--(0,4)--(0.433,3.75);
\path [fill=blue!30] (1.732,0)--(2.598,0)--(2.598,-0.5)--(2.165,-0.75);
\path [fill=blue!30] (-0.866,1.5)--(-1.732,1.5)--(-1.732,2)--(-1.299,2.25);
\path [fill=blue!30] (-2.598,1.5)--(-3.46,1.5)--(-3.46,2);
\path [fill=blue!30] (2.598,1.5)--(3.46,1.5)--(3.46,2);
\path [fill=blue!30] (-1.732,0)--(-1.299,0.75)--(-0.866,0.5);
\path [fill=blue!30] (-2.598,0)--(-2.598,-0.5)--(-2.165,-0.75)--(-1.732,0);
\path [fill=blue!30] (-2.598,-1.5)--(-3.46,-2)--(-3.03,-2.25);
\path [fill=blue!30] (2.598,-1.5)--(3.46,-2)--(3.03,-2.25);
\path [fill=blue!30] (0.866,-1.5)--(0.866,-2.5)--(1.299,-2.25);
\path [fill=red!30] (-0.866,1.5)--(0,2)--(0.866,1.5);
\path [fill=red!30] (-0.866,1.5)--(-0.866,2.5)--(-0.466,2.25);
\path [fill=red!30] (0.866,0.5)--(0.866,1.5)--(1.732,1);
\path [fill=red!30] (-1.732,0)--(-0.866,1.5)--(-1.732,1);
\path [fill=red!30] (1.299,0.75)--(1.732,1)--(1.732,0);
\path [fill=gray!90] (0,0) -- (0.433,0.75) -- (0,1) -- (0,0);
\draw(-4.33,6)--(4.33,6);
\draw(-4.33,4.5)--(4.33,4.5);
\draw(-4.33,3)--(4.33,3);
\draw(-4.33,1.5)--(4.33,1.5);
\draw(-4.33,0)--(4.33,0);
\draw(-4.33,-1.5)--(4.33,-1.5);
\draw(-4.33,-3)--(4.33,-3);
\draw(-4.33,-3)--(-4.33,6);
\draw(-3.464,-3)--(-3.464,6);
\draw(-2.598,-3)--(-2.598,6);
\draw(-1.732,-3)--(-1.732,6);
\draw(-.866,-3)--(-.866,6);
\draw(0,-3)--(0,6);
\draw(.866,-3)--(.866,6);
\draw(1.732,-3)--(1.732,6);
\draw(2.598,-3)--(2.598,6);
\draw(3.464,-3)--(3.464,6);
\draw(4.33,-3)--(4.33,6);
\draw(-4.33,5.5)--(-3.464,6);
\draw(-4.33,4.5)--(-1.732,6);
\draw(-4.33,3.5)--(0,6);
\draw(-4.33,2.5)--(1.732,6);
\draw(-4.33,1.5)--(3.464,6);
\draw(-4.33,.5)--(4.33,5.5);
\draw(-4.33,-.5)--(4.33,4.5);
\draw(-4.33,-1.5)--(4.33,3.5);
\draw(-4.33,-2.5)--(4.33,2.5);
\draw(-3.464,-3)--(4.33,1.5);
\draw(-1.732,-3)--(4.33,.5);
\draw(0,-3)--(4.33,-.5);
\draw(1.732,-3)--(4.33,-1.5);
\draw(3.464,-3)--(4.33,-2.5);
\draw(4.33,5.5)--(3.464,6);
\draw(4.33,4.5)--(1.732,6);
\draw(4.33,3.5)--(0,6);
\draw(4.33,2.5)--(-1.732,6);
\draw(4.33,1.5)--(-3.464,6);
\draw(4.33,.5)--(-4.33,5.5);
\draw(4.33,-.5)--(-4.33,4.5);
\draw(4.33,-1.5)--(-4.33,3.5);
\draw(4.33,-2.5)--(-4.33,2.5);
\draw(3.464,-3)--(-4.33,1.5);
\draw(1.732,-3)--(-4.33,.5);
\draw(0,-3)--(-4.33,-.5);
\draw(-1.732,-3)--(-4.33,-1.5);
\draw(-3.464,-3)--(-4.33,-2.5);
\draw(-4.33,-1.5)--(-3.464,-3);
\draw(-4.33,1.5)--(-1.732,-3);
\draw(-4.33,4.5)--(0,-3);
\draw(-3.464,6)--(1.732,-3);
\draw(-1.732,6)--(3.464,-3);
\draw(0,6)--(4.33,-1.5);
\draw(1.732,6)--(4.33,1.5);
\draw(3.464,6)--(4.33,4.5);
\draw(4.33,-1.5)--(3.464,-3);
\draw(4.33,1.5)--(1.732,-3);
\draw(4.33,4.5)--(0,-3);
\draw(3.464,6)--(-1.732,-3);
\draw(1.732,6)--(-3.464,-3);
\draw(0,6)--(-4.33,-1.5);
\draw(-1.732,6)--(-4.33,1.5);
\draw[line width=0.1pt](-3.464,6)--(-4.33,4.5);
\draw[line width=2pt](0,-3)--(0,6);%
\draw[line width=2pt](0.866,-3)--(0.866,6);%
\draw[line width=2pt](-3.46,-3)--(1.732,6);%
\draw[line width=2pt](-1.732,-3)--(3.46,6);%
\draw[line width=2pt](1.732,-3)--(-3.46,6);%
\draw[line width=2pt](3.46,-3)--(-1.732,6);%
\draw[line width=2pt](-4.33,-2.5)--(4.33,2.5);
\draw[line width=2pt](-4.33,-1.5)--(4.33,3.5);
\draw[line width=2pt](-4.33,0)--(4.33,0);
\draw[line width=2pt](-4.33,2.5)--(4.33,-2.5);
\draw[line width=2pt](-4.33,3.5)--(4.33,-1.5);
\draw[line width=2pt](-4.33,1.5)--(4.33,1.5);
\end{tikzpicture}}\qquad
\subfigure[The cone type partition]{
\centering
\begin{tikzpicture}[scale=0.75]
\path [fill=blue!30] (0,-1) -- (-0.866,-0.5) -- (-0.866,1.5) -- (0.866,1.5)--(0.866,0.5)--(1.299,0.75)--(1.732,0)--(0.866,-0.5)--(0,-1);
\path [fill=blue!30] (0.866,-0.5)--(0.866,-1.5)--(1.299,-0.75);
\path [fill=blue!30] (0.866,1.5)--(0.4333,2.25)--(0.866,2.5)--(1.732,2)--(1.732,1.5);
\path [fill=blue!30] (0.866,4.5)--(0.866,5.5)--(1.299,5.25);
\path [fill=blue!30] (0,3)--(-0.433,3.75)--(0,4)--(0.433,3.75);
\path [fill=blue!30] (1.732,0)--(2.598,0)--(2.598,-0.5)--(2.165,-0.75);
\path [fill=blue!30] (-0.866,1.5)--(-1.732,1.5)--(-1.732,2)--(-1.299,2.25);
\path [fill=blue!30] (-2.598,1.5)--(-3.46,1.5)--(-3.46,2);
\path [fill=blue!30] (2.598,1.5)--(3.46,1.5)--(3.46,2);
\path [fill=blue!30] (-1.732,0)--(-1.299,0.75)--(-0.866,0.5);
\path [fill=blue!30] (-2.598,0)--(-2.598,-0.5)--(-2.165,-0.75)--(-1.732,0);
\path [fill=blue!30] (-2.598,-1.5)--(-3.46,-2)--(-3.03,-2.25);
\path [fill=blue!30] (2.598,-1.5)--(3.46,-2)--(3.03,-2.25);
\path [fill=blue!30] (0.866,-1.5)--(0.866,-2.5)--(1.299,-2.25);
\path [fill=gray!90] (0,0) -- (0.433,0.75) -- (0,1) -- (0,0);
\draw(-4.33,6)--(4.33,6);
\draw(-4.33,4.5)--(4.33,4.5);
\draw(-4.33,3)--(4.33,3);
\draw(-4.33,1.5)--(4.33,1.5);
\draw(-4.33,0)--(4.33,0);
\draw(-4.33,-1.5)--(4.33,-1.5);
\draw(-4.33,-3)--(4.33,-3);
\draw(-4.33,-3)--(-4.33,6);
\draw(-3.464,-3)--(-3.464,6);
\draw(-2.598,-3)--(-2.598,6);
\draw(-1.732,-3)--(-1.732,6);
\draw(-.866,-3)--(-.866,6);
\draw(0,-3)--(0,6);
\draw(.866,-3)--(.866,6);
\draw(1.732,-3)--(1.732,6);
\draw(2.598,-3)--(2.598,6);
\draw(3.464,-3)--(3.464,6);
\draw(4.33,-3)--(4.33,6);
\draw(-4.33,5.5)--(-3.464,6);
\draw(-4.33,4.5)--(-1.732,6);
\draw(-4.33,3.5)--(0,6);
\draw(-4.33,2.5)--(1.732,6);
\draw(-4.33,1.5)--(3.464,6);
\draw(-4.33,.5)--(4.33,5.5);
\draw(-4.33,-.5)--(4.33,4.5);
\draw(-4.33,-1.5)--(4.33,3.5);
\draw(-4.33,-2.5)--(4.33,2.5);
\draw(-3.464,-3)--(4.33,1.5);
\draw(-1.732,-3)--(4.33,.5);
\draw(0,-3)--(4.33,-.5);
\draw(1.732,-3)--(4.33,-1.5);
\draw(3.464,-3)--(4.33,-2.5);
\draw(4.33,5.5)--(3.464,6);
\draw(4.33,4.5)--(1.732,6);
\draw(4.33,3.5)--(0,6);
\draw(4.33,2.5)--(-1.732,6);
\draw(4.33,1.5)--(-3.464,6);
\draw(4.33,.5)--(-4.33,5.5);
\draw(4.33,-.5)--(-4.33,4.5);
\draw(4.33,-1.5)--(-4.33,3.5);
\draw(4.33,-2.5)--(-4.33,2.5);
\draw(3.464,-3)--(-4.33,1.5);
\draw(1.732,-3)--(-4.33,.5);
\draw(0,-3)--(-4.33,-.5);
\draw(-1.732,-3)--(-4.33,-1.5);
\draw(-3.464,-3)--(-4.33,-2.5);
\draw(-4.33,-1.5)--(-3.464,-3);
\draw(-4.33,1.5)--(-1.732,-3);
\draw(-4.33,4.5)--(0,-3);
\draw(-3.464,6)--(1.732,-3);
\draw(-1.732,6)--(3.464,-3);
\draw(0,6)--(4.33,-1.5);
\draw(1.732,6)--(4.33,1.5);
\draw(3.464,6)--(4.33,4.5);
\draw(4.33,-1.5)--(3.464,-3);
\draw(4.33,1.5)--(1.732,-3);
\draw(4.33,4.5)--(0,-3);
\draw(3.464,6)--(-1.732,-3);
\draw(1.732,6)--(-3.464,-3);
\draw(0,6)--(-4.33,-1.5);
\draw(-1.732,6)--(-4.33,1.5);
\draw(-3.464,6)--(-4.33,4.5);
\draw[line width=2pt](0,-3)--(0,6);
\draw[line width=2pt](0.866,-3)--(0.866,0.5);
\draw[line width=2pt](0.866,1.5)--(0.866,6);
\draw[line width=2pt](-3.46,-3)--(-1.732,0);
\draw[line width=2pt](0,3)--(1.732,6);
\draw[line width=2pt](-1.732,-3)--(3.46,6);
\draw[line width=2pt](1.732,-3)--(-3.46,6);
\draw[line width=2pt](3.46,-3)--(1.732,0);
\draw[line width=2pt](0.866,1.5)--(-1.732,6);
\draw[line width=2pt](-4.33,-2.5)--(4.33,2.5);
\draw[line width=2pt](-4.33,-1.5)--(4.33,3.5);
\draw[line width=2pt](-4.33,0)--(4.33,0);
\draw[line width=2pt](-4.33,2.5)--(4.33,-2.5);
\draw[line width=2pt](-4.33,3.5)--(4.33,-1.5);
\draw[line width=2pt](-4.33,1.5)--(-0.866,1.5);
\draw[line width=2pt](0.866,1.5)--(4.33,1.5);
\end{tikzpicture}}
\caption{The partitions $\scrD$, $\scrJ$, $\scrS_0$, and $\scrT$ for $\tilde{\sG}_2$}\label{fig:G2partitions}
\end{figure}
\end{exa}

\begin{rem}\label{rem:shi} We make the following comments.
\begin{enumerate}
\item If $W$ is affine, then the $0$-Shi partition $\scrS_0$ is the partition induced by the classical Shi hyperplane arrangement (see \cite{Shi:87a,Shi:87b}). 
\item The partitions $\scrT$ and $\scrG_B$ (with $B$ a Garside shadow) are generally not hyperplane partitions. For example, see the partition $\scrT$ in Figure~\ref{fig:G2partitions}. 
\item Note that $\scrJ\leq \scrS_0$, as $\Phisph^+\subseteq \cE_0$.
\end{enumerate}
\end{rem}

\subsection{Regular partitions}\label{subsec:regularpartitions}
In this section we define the notion of a \textit{regular partition} of $W$, and show that these partitions are intimately related the automatic structure of~$W$. 

\begin{defn}
A partition $\scrP$ of $W$ is \textit{locally constant} if the function $D_L:W\to 2^S$ is constant on each part of $\scrP$. If $\scrP$ is locally constant, and $P\in\scrP$, we write $D_L(P)=D_L(w)$ for any $w\in W$.
\end{defn}

\newpage

Note that every refinement of a locally constant partition is again locally constant. Moreover, we have the following. 

\begin{lem}\label{lem:Sminimal}
A partition $\scrP$ is locally constant if and only if $\scrD\leq \scrP$. 
\end{lem}

\begin{proof}
This is immediate from Proposition~\ref{prop:partsdescription}.
\end{proof}

\begin{prop}\label{prop:locallyconstantexamples}
All partitions in Definition~\ref{defn:partitions} are locally constant.
\end{prop}

\begin{proof}
By Proposition~\ref{prop:partsdescription} the parts of $\scrT$ are the sets $X_T=\{w\in W\mid T(w^{-1})=T\}$, with $T\in\mathbb{T}$. If $x,y\in X_T$ then $T(x^{-1})=T(y^{-1})$, and thus $D_L(x)=D_L(y)$. So $\scrT$ is locally constant. 

Let $B$ be a Garside shadow. By Proposition~\ref{prop:partsdescription} the parts of $\scrG_B$ are the sets $\pi^{-1}_B(b)$, with $b\in B$. If $x,y\in \pi^{-1}_B(b)$ then $\pi_B(x)=\pi_B(y)=b$. By \cite[Proposition~2.6]{HNW:16} we have $D_L(x)=D_L(\pi_B(x))=D_L(y)$, and hence $\scrG_B$ is locally constant.

If $x,y$ lie in the same part of $\scrS_n$ then $\cE_n(x)=\cE_n(y)$. Since each root $\alpha_s$ with $s\in S$ is elementary, it follows that $D_L(x)=D_L(y)$. 

The remaining cases follow easily from Proposition~\ref{prop:partsdescription}.
\end{proof}

The main definition of this section is as follows.

\begin{defn}
A partition $\scrR$ of $W$ is \textit{regular} if:
\begin{enumerate}
\item $\scrR$ is locally constant, and
\item if $R\in \scrR$ and $s\notin D_L(R)$ then $sR\subseteq R'$ for some part $R'\in\scrR$.\end{enumerate}
Let $\Preg$ denote the set of all regular partitions of $W$.
\end{defn}

Note that the condition $s\notin D_L(R)$ is equivalent to $R\subseteq H_{\alpha_s}^+$, or more gemetrically, that $e$ and $R$ both lie on the same side of the wall $H_{\alpha_s}$ in the Coxeter complex.

\begin{thm}\label{thm:regularpartitions}
The following partitions of $W$ are regular: 
\begin{enumerate}
\item the cone type partition $\scrT$;
\item the Garside partition $\scrG_B$, for any Garside shadow~$B$;
\item the $n$-Shi partition $\scrS_n$, for $n\in\mathbb{N}$.
\end{enumerate}
\end{thm}

\begin{proof}
By Proposition~\ref{prop:locallyconstantexamples} these partitions are all locally constant. 

(1) By Proposition~\ref{prop:partsdescription} each part of $\scrT$ is of the form $X_T=\{w\in W\mid T(w^{-1})=T\}$ for some cone type~$T$. Suppose that $s\in S$ with $s\notin D_L(X_T)$. For any $w\in X_T$ we have $D_L(w)=D_L(X_T)$, and so $s\in T(w^{-1})=T$. By Lemma~\ref{lem:conetypeevolution} we have 
$$
T((sw)^{-1})=T(w^{-1}s)=s\{x\in T\mid \ell(sx)=\ell(x)-1\},
$$
which is independent of the particular $w\in X_T$. Thus, writing $T'=T(w^{-1}s)$ (for any $w\in X_T$), we have $sX_T\subseteq X_{T'}$, and  hence $\scrT\in\Preg$.

(2) Let $B$ be a Garside shadow. By Proposition~\ref{prop:partsdescription} the parts of $\scrG_B$ are the sets $\pi^{-1}_B(b)$, with $b\in B$. If $x,y\in \pi^{-1}_B(b)$ and $s\notin D_L(x)=D_L(y)$ then by \cite[Proposition~2.8]{HNW:16} we have $\pi_B(sx)=\pi_B(s\pi_B(x))=\pi_B(sb)=\pi_B(sy)$, and so $sx$ and $sy$ lie in the part $\pi^{-1}_B(\pi_B(sb))$. Hence $\scrG_B\in\Preg$.

(3) If $x,y\in W$ lie in the same part of $\scrS_n$ then $\cE_n(x)=\cE_n(y)=A$, say. If $s\notin D_L(x)=D_L(y)$ then by Lemma~\ref{lem:elementaryrootsbasics} we have $\cE_n(sx)=(\{\alpha_s\}\sqcup sA)\cap \cE_n=\cE_n(sy)$. Thus $sx$ and $sy$ lie in a common part of $\scrS_n$, and so the partition is regular.
\end{proof}

\begin{rem}
We record the following observations.  
\begin{enumerate}
\item The top element $\mathbf{1}=\{\{w\}\mid w\in W\}$ of $(\scrP(W),\leq)$ is obviously regular. Thus each $\scrP\in\scrP(W)$ can be refined to a regular partition. In Section~\ref{sec:regularcompletion} we will show that there is a unique minimal such ``regular completion''.
\item The partition $\scrD$ is generally not regular (see, for example, Figure~\ref{fig:G2partitions}(a)).  
\item The partition $\scrJ$ is generally not regular. For example, the partition $\scrJ$ is not regular for $\tilde{\sG}_2$ (see Figure~\ref{fig:G2partitions}(b)), however it is regular for $\tilde{\sA}_2$ (see Figure~\ref{fig:A2partitions}). 
\begin{figure}[H]
\centering
\begin{tikzpicture}[scale=0.9]
 \path [fill=blue!30] (-1.5,0.866)--(-1,{2*0.866})--(1,{2*0.866})--(1.5,0.866)--(0.5,-0.866)--(-0.5,-0.866)--(-1.5,0.866);
  \path [fill=blue!30] (0,{2*0.866})--(-0.5,{3*0.866})--(0.5,{3*0.866})--(0,{2*0.866});
   \path [fill=blue!30] (1,0)--(2,0)--(1.5,-0.866);
    \path [fill=blue!30] (-1,0)--(-2,0)--(-1.5,-0.866);
    \path [fill=gray!90] (0,0) -- (-0.5,0.866) -- (0.5,0.866) -- (0,0);
    \draw (2.5, {-3*0.866})--( 3.5, {-1*0.866} );
    \draw (1.5, {-3*0.866})--( 3.5, {1*0.866} );
    \draw  (0.5, {-3*0.866})--( 3.5, {3*0.866} );
    \draw  (-0.5, {-3*0.866})--( 3, {4*0.866} );
    \draw  [line width=2pt](-1.5, {-3*0.866})--( 2, {4*0.866} );
    \draw  [line width=2pt] (-2.5, {-3*0.866})--( 1, {4*0.866} );
    \draw  (-3.5, {-3*0.866})--(0, {4*0.866} );
    \draw  (-3.5, {-1*0.866})--(-1, {4*0.866} );
    \draw  (-3.5, {1*0.866})--(-2, {4*0.866} );
    \draw  (-3.5, {3*0.866})--(-3, {4*0.866} );
   \draw (-2.5, {-3*0.866})--( -3.5, {-1*0.866} );
    \draw (-1.5, {-3*0.866})--( -3.5, {1*0.866} );
    \draw  (-0.5, {-3*0.866})--( -3.5, {3*0.866} );
    \draw  (0.5, {-3*0.866})--( -3, {4*0.866} );
    \draw [line width=2pt] (1.5, {-3*0.866})--( -2, {4*0.866} );
    \draw  [line width=2pt] (2.5, {-3*0.866})--( -1, {4*0.866} );
    \draw  (3.5, {-3*0.866})--(0, {4*0.866} );
    \draw  (3.5, {-1*0.866})--(1, {4*0.866} );
    \draw  (3.5, {1*0.866})--(2, {4*0.866} );
    \draw  (3.5, {3*0.866})--(3, {4*0.866} );
    \draw (-3.5, -2.598)--( 3.5, -2.598);
    \draw (-3.5, -1.732)--( 3.5, -1.732);
    \draw (-3.5, -0.866)--( 3.5, -0.866);
    \draw[line width=2pt] (-3.5, 0)--( 3.5, 0);
    \draw (-3.5, 3.464)--( 3.5, 3.464 );
    \draw (-3.5, 2.598)--( 3.5, 2.598);
    \draw (-3.5, 1.732)--( 3.5, 1.732);
    \draw [line width=2pt] (-3.5, 0.866)--( 3.5, 0.866);
     %
%
\end{tikzpicture}
\caption{In type $\tilde{\sA}_2$ we have $\scrT=\scrS_0=\scrJ$}\label{fig:A2partitions}
\end{figure}
\end{enumerate} 
\end{rem}

The main interest in the concept of regular partitions stems from the following theorem, providing a very general geometric construction of automata recognising $\cL(W,S)$. Note that if $\scrP$ is locally constant then the part of $\scrP$ containing $e$ is the singleton~$\{e\}$ (by considering left descent sets). 

\begin{thm}\label{thm:regularautomaton}
Let $\scrR$ be a regular partition of $W$. Define $\mu:\scrR\times S\to \scrR\cup\{\dagger\}$ by 
$$
\mu(R,s)=\begin{cases}
R'&\text{if $s\notin D_L(R)$ and $sR\subseteq R'\in\scrR$}\\
\dagger&\text{if $s\in D_L(R)$}.
\end{cases}
$$
Then $\mathcal{A}(\scrR)=(\scrR,\mu,\{e\})$ is an automaton recognising $\cL(W,S)$. 

Moreover, if $w=s_1\cdots s_n$ is reduced, then the final state of the path with edge labels $(s_1,\ldots,s_n)$ starting at $\{e\}$ is the part $R\in\scrR$ with $w^{-1}\in R$.
\end{thm}

\begin{proof} Let $R_0=\{e\}$. We show, by induction on $n\geq 1$, that $(s_1,\ldots,s_n)$ is accepted by $\cA(\scrR)$ if and only if $s_1\cdots s_n$ is reduced. Consider $n=1$. Each expression $s$ is reduced. On the other hand, since $s\notin D_L(R_0)$ we have $\mu(R_0,s)=R_1$, where $R_1\in\scrR$ is the part containing $s$. Thus all length $1$ reduced words are accepted by~$\cA(\scrR)$.

Let $k\geq 1$, and suppose that $s_1\cdots s_k$ is reduced if and only if $(s_1,\ldots,s_k)$ is accepted by $\cA(\scrR)$. Let $s_1\cdots s_ks$ be reduced. Then $s_1\cdots s_k$ is reduced, and so $(s_1,\ldots,s_k)$ is accepted by~$\cA(\scrR)$. Let $R_k\in\scrR$ be the part of $\scrR$ containing $s_k\cdots s_1$. Since $s_1\cdots s_ks$ is reduced we have $s\notin D_L(s_k\cdots s_1)=D_L(R_k)$. Hence, by the regularity condition, $sR_k\subseteq R_{k+1}$ where $R_{k+1}$ is the part of $\scrR$ with $ss_k\cdots s_1\in R_{k+1}$. Then $\mu(R_k,s)=R_{k+1}$, and so $(s_1,\ldots,s_k,s)$ is accepted by~$\cA(\scrR)$. 

Conversely, suppose that $(s_1,\ldots,s_k,s)$ is accepted by $\cA(\scrR)$. Then $(s_1,\ldots,s_k)$ is accepted, and so $s_1\cdots s_k$ is reduced. Moreover, with $R_k$ being the part containing $s_k\cdots s_1$, the fact that $(s_1,\ldots,s_k,s)$ is accepted gives $\mu(R_k,s)=R_{k+1}$ where $R_{k+1}$ is the part containing $ss_k\cdots s_1$. Thus, by definition, $s\notin D_L(R_k)$. In particular $s\notin D_L(s_k\cdots s_1)$, and so $s_1\cdots s_ks$ is reduced. 

The final statement is now clear: If $w=s_1\cdots s_k$ is a reduced expression, then the corresponding path in the automaton $\cA(\scrR)=(\scrR,\mu,\{e\})$ is 
$$
\{e\}=R_0\to_{s_1}R_1\to_{s_2}R_2\to_{s_3}\cdots\to_{s_k}R_k,
$$
where $R_j\in\scrR$ is the part containing $s_jR_{j-1}$. Thus $s_n\cdots s_1\in R_n$.
\end{proof}

The above construction leads to a uniform and conceptual proof of the known automata recognising $\cL(W,S)$. In particular, using Theorems~\ref{thm:regularpartitions} and~\ref{thm:regularautomaton} we obtain new proofs of Theorems~\ref{thm:canonicalautomaton} and~\ref{thm:garsideautomaton}. Moreover, the above construction leads to the following remarkable fact that will be used in a crucial way to define the ``regular completion'' of a partition in the following section.

\newpage

\begin{cor}\label{cor:Tisminimal}
If $\scrR\in\Preg$ then $\scrR$ is a refinement of $\scrT$ (that is, $\scrT\leq \scrR$). 
\end{cor}

\begin{proof}
Let $\cA=(\scrR,\mu,\{e\})$ be the automaton constructed in Theorem~\ref{thm:regularautomaton}. Let $R\in \scrR$, and suppose that $x,y\in R$. Let $x=s_1\cdots s_n$ and $y=s_1'\cdots s_m'$ be reduced expressions. By the final statement of Theorem~\ref{thm:regularautomaton}, we have that the paths in $\cA$ starting at $\{e\}$ with edge labels $(s_n,\ldots,s_1)$ and $(s_m',\ldots,s_1')$ both end at the state $R$. Then by Lemma~\ref{lem:stateconetype} we have $T(s_n\cdots s_1)=T(s_m'\cdots s_1')$, and so $T(x^{-1})=T(y^{-1})$. Thus $x$ and $y$ lie in the same part of the cone type partition (by Proposition~\ref{prop:partsdescription}) and so $\scrR$ is a refinement of $\scrT$.  
\end{proof}

There is a partial converse to Theorem~\ref{thm:regularautomaton}. We call an automaton $\cA=(Y,\mu,o)$ recognising $\cL(W,S)$ \textit{reduced} if the following property holds: If $w=s_1\cdots s_n$ and $w=s_1'\cdots s_n'$ are both reduced expressions for $w$, then the paths in $\cA$ starting at $o$ with edge labels $(s_1,\ldots,s_n)$ and $(s_1',\ldots,s_n')$ finish at the same state. 

\begin{prop}
Let $\scrR\in\Preg$. The automaton $\cA(\scrR)$ constructed in Theorem~\ref{thm:regularautomaton} is reduced. 
\end{prop}

\begin{proof}
From the final statement of Theorem~\ref{thm:regularautomaton}, if $w=s_1\cdots s_n$ is a reduced expression then the final state in the path with edge labels $(s_1,\ldots,s_n)$ does not depend on the particular reduced expression chosen, and so $\cA(\scrR)$ is reduced.
\end{proof}
%
%

If $\cA=(Y,\mu,o)$ is reduced, then for $w\in W$ we can define
$$
\mu(o,w)=\mu(\ldots\mu(\mu(o,s_1),s_2),\ldots,s_n)
$$
where $s_1\cdots s_n$ is any reduced expression for~$w$ (this is well defined by the reduced assumption). Thus $\mu(o,w)$ is the common end state of all paths in $\cA$ whose edge labels represent~$w$. 

Let $\mathbf{A}_{\mathrm{red}}(W,S)$ denote the set of isomorphism classes of reduced automata recognising~$\cL(W,S)$.

\begin{thm}\label{thm:converseautomaton}
Let $\cA=(Y,\mu,o)$ be a reduced automaton recognising $\cL(W,S)$. The partition $\scrR(\cA)$ of $W$ into sets 
$$
A_y=\{w\in W\mid \mu(o,w^{-1})=y\},\quad\text{with $y\in Y$},
$$
is a regular partition of $W$.

Moreover, the functions $F:\Preg\to \mathbf{A}_{\mathrm{red}}(W,S)$ and $G:\mathbf{A}_{\mathrm{red}}(W,S)\to \Preg$ with $F(\scrR)=\cA(\scrR)$ (c.f. Theorem~\ref{thm:regularautomaton}) and $G(\cA)=\scrR(\cA)$ are are mutually inverse bijections.
\end{thm}

\begin{proof}
It is clear that $W=\bigcup_{y\in Y}A_y$, and that $A_y\cap A_{y'}=\emptyset$ if $y\neq y'$. Thus $\scrR(\cA)$ is a partition of~$W$. Moreover, if $u,v\in A_y$ then $T(u^{-1})=T(v^{-1})$ by Lemma~\ref{lem:stateconetype}. In particular, $D_L(u)=D_L(v)$, and so $\scrR(\cA)$ is locally constant. 

Suppose that $s\notin D_L(A_y)$, and consider $w\in A_y$. Thus $\ell(sw)=\ell(w)+1$ and so $\ell(w^{-1}s)=\ell(w)+1$. Since $y=\mu(o,w^{-1})$, if $w^{-1}=s_1\cdots s_n$ is reduced then $w^{-1}s=s_1\cdots s_ns$ is also reduced, and hence $\mu(o,w^{-1}s)=\mu(y,s)$. Thus $sw\in A_{\mu(y,s)}$, and hence $sA_y\subseteq A_{\mu(y,s)}$. So $\scrR(\cA)$ is regular.

To prove the final statement, we show that $G(F(\scrR))=\scrR$ and $F(G(\cA))=\cA$ for all $\scrR\in\Preg$ and $\cA\in \mathbf{A}_{\mathrm{red}}(W,S)$. For the first statement, the states of the automaton $F(\scrR)$ are parts of $\scrR$, and so the parts of the partition $G(F(\scrR))$ are the sets
$
A_P=\{w\in W\mid \mu(o,w^{-1})=P\}$ with $P\in\scrR$, 
with $\mu$ as in Theorem~\ref{thm:regularautomaton}. Recall that if $w^{-1}=s_1\cdots s_n$ is reduced then the path in $F(\scrR)$ starting at $\{e\}$ with edge labels $(s_1,\ldots,s_n)$ ends at the part $P$ of $\scrR$ containing $w$ (see Theorem~\ref{thm:regularautomaton}). Thus $A_P=\{w\in W\mid w\in P\}=P$.

On the other hand, if $\cA=(Y,\mu,o)\in\mathbf{A}_{\mathrm{red}}(W,S)$, then the states of the automaton $F(G(\cA))$ are the sets $A_y=\{w\in W\mid \mu(o,w^{-1})=y\}$, with $y\in Y$, and the transition function is given by $\mu'(A_y,s)=A_{y'}$ if $s\notin D_L(A_y)$ and $sA_y\subseteq A_{y'}$, with $y'\in Y$. We showed above that $sA_y\subseteq A_{\mu(y,s)}$, and hence $\mu'(A_y,s)=A_{\mu(y,s)}$. It follows that $f:\cA\to F(G(\cA))$ with $f(y)=A_y$ is an isomorphism of automata. 
\end{proof}

\subsection{The regular completion of a partition}\label{sec:regularcompletion}

In this section we show that the partially ordered set $(\Preg,\leq)$ is a complete lattice (that is, every nonempty subset has both a meet and a join). As a consequence, given an arbitrary partition $\scrP\in\scrP(W)$ there exists a unique minimal regular partition $\widehat{\scrP}\in\Preg$ refining $\scrP$ (we call this partition the \textit{regular completion} of $\scrP$. We provide an algorithm for computing the regular completion, along with sufficient conditions for this algorithm to terminate in finite time. An important consequence of this analysis is that $\scrT=\widehat{\scrD}$ (see Corollary~\ref{cor:regularlattice1}). This fact will play a pivotal role in proving Theorem~\ref{thm:main1}.

\begin{thm}\label{thm:regularlattice2}
The partially ordered set $(\Preg,\leq)$ is a complete lattice, with bottom element $\scrT$ and top element $\mathbf{1}$ (recall convention~(\ref{eq:convention})).
\end{thm}

\begin{proof}
Let $X=\{\scrP_i\mid i\in I\}\subseteq \Preg$ be nonempty. The join of $X$ in the partially ordered set $(\scrP(W),\leq)$ is (see, for example, \cite[p.36]{DP:02})
$$
\bigvee X=\bigg\{\bigcap_{i\in I}P_i\biggm| P_i\in\scrP_i,\,\bigcap_{i\in I}P_i\neq\emptyset\bigg\}.
$$
Write $\scrP=\bigvee X$. We claim that $\scrP\in\Preg$. Clearly $\scrP$ is locally constant (as it is a common refinement of locally constant partitions). Moreover, if $P\in\scrP$ and $s\in S$ with $s\notin D_L(P)$ then writing $P=\bigcap_{i\in I}P_i$ with $P_i\in \scrP_i$ we have $s\notin D_L(P_i)=D_L(P)$ for all $i\in I$ (as $P\neq \emptyset$ and each $\scrP_i$ is locally constant). Thus by regularity of each $\scrP_i$ there is $P_i'\in\scrP_i$ with $sP_i\subseteq P_i'$. Thus $sP=\bigcap_{i\in I}(sP_i)\subseteq \bigcap_{i\in I}P_i'$, which is a part of $\scrP$ by definition. Thus $\scrP\in\Preg$.  Thus every nonempty subset $X\subseteq \Preg$ has a join in the partially ordered set $(\Preg,\leq)$.  

By Corollary~\ref{cor:Tisminimal} the cone type partition $\scrT\in\Preg$ is a lower bound for every nonempty subset $X\subseteq \Preg$. Thus the set 
$$
\{\scrR\in \Preg\mid \scrR\leq \scrP\text{ for all $\scrP\in X$}\}
$$
is nonempty, and using the existence of joins from the previous paragraph, the meet of $X$ is given by
$$
\bigwedge X=\bigvee\{\scrR\in \Preg\mid \scrR\leq \scrP\text{ for all $\scrP\in X$}\}.
$$
Thus $\Preg$ is a complete lattice with bottom element $\scrT$ and top element $\mathbf{1}$.
\end{proof}

Theorem~\ref{thm:regularlattice2} allows us to define the ``regular completion'' of a partition.

\begin{defn}\label{defn:regularcompletion}
The \textit{regular completion} of $\scrP\in\scrP(W)$ is 
$$
\widehat{\scrP}=\bigwedge\{\scrR\in\Preg\mid \scrP\leq \scrR\}
$$
Then $\widehat{\scrP}$ is a regular partition, as $(\Preg,\leq)$ is a complete lattice by Theorem~\ref{thm:regularlattice2}.
\end{defn}

It is not immediate from the definition that $\scrP\leq\widehat{\scrP}$, however we shall see that this is indeed true in Theorem~\ref{thm:newterminateatregularisation} below (and hence $\widehat{\scrP}$ is the minimal regular partition refining~$\scrP$).

\subsection{Simple refinements algorithm}\label{sec:simplerefinements}

We now develop an algorithm to compute $\widehat{\scrP}$. This algorithm will not always terminate in finite time, however we will provide natural sufficient conditions under which it will terminate in finite time.

 To begin with, if $\scrP\in\scrP(W)$ then the minimal locally constant partition refining $\scrP$ is obviously $\scrP\vee \scrD$ (see Lemma~\ref{lem:Sminimal}), whose parts are the sets $P\cap D_L^{-1}(J)$ for $P\in\scrP$, $J\subseteq S$ spherical, and $P\cap D_L^{-1}(J)\neq\emptyset$. Moreover, from the definition of regular completion it is clear that $\widehat{\scrP}=\widehat{\scrP\vee\scrD}$, for if $\scrR$ is regular with $\scrP\leq\scrR$ then $\scrP\vee\scrD\leq \scrR\vee\scrD=\scrR$ (as $\scrR$ is locally constant). Thus after replacing $\scrP$ with $\scrP\vee\scrD$, we may assume that $\scrP$ is locally constant.

We now define a \textit{simple refinement} of a locally constant partition. These operations will be the basic building blocks of our algorithm for computing the regular completion.

\begin{defn}\label{defn:simplerefinement}
Let $\scrP$ be a locally constant partition of $W$. Suppose that $(P,s)\in\scrP\times S$ is such that $s\notin D_L(P)$ and $sP$ is not contained in a part of $\scrP$. Let
$
X=\{P'\in\scrP\mid P'\cap sP\neq\emptyset\},
$
and partition the set $P$ as
$$
P=\bigsqcup_{P'\in X}(P\cap sP'),
$$
and let $\scrP'$ be the refinement of $\scrP$ obtained by replacing the part $P$ of $\scrP$ by the above partition. We call the refinement $\scrP\mapsto \scrP'$ the \textit{simple refinement}, and we say that this refinement is \textit{based at the pair $(P,s)$}. 
\end{defn}

\newpage

We note the following.

\begin{prop}\label{prop:finiteparts} If $\scrP$ is locally constant and $\scrP\mapsto\scrP'$ by a simple refinement, then $\scrP'$ is locally constant, and if $|\scrP|<\infty$ then $|\scrP'|<2|\scrP|<\infty$. 
\end{prop}

\begin{proof}
The first statement is clear, because all refinements of a locally constant partition are locally constant. For the second statement, note that $|\scrP'|=|\scrP|+|X|-1$ (with $X$ as in Definition~\ref{defn:simplerefinement}, as the single part $P$ is replaced by $|X|$ parts $P\cap sP'$ with $P'\in X$), and that $|X|\leq |\scrP|$.
\end{proof}

\begin{exa}
Figure~\ref{fig:simplerefinement} gives an example of a simple refinement in type $\tilde{\sG}_2$. We start with the locally constant partition $\scrP$ determined by the solid heavy lines. Let $P$ be the part of $\scrP$ shaded blue, and let $s$ be reflection in the vertical wall bounding the identity chamber (shaded grey). Note that $s\notin D_L(P)$, as $e$ and $P$ lie on the same side of $H_{\alpha_s}$. The set $sP$ is shaded red. There are $4$ parts $P'$ of $\scrP$ such that $sP\cap P'\neq\emptyset$. Let $\scrP\mapsto \scrP'$ via the simple refinement based at $(P,s)$. The partition $\scrP'$ is given by the union of the solid heavy lines and the dotted heavy lines. The meaning of the black and red circles will be given in Section~\ref{sec:gatedpartitions} (see Example~\ref{ex:gates}).

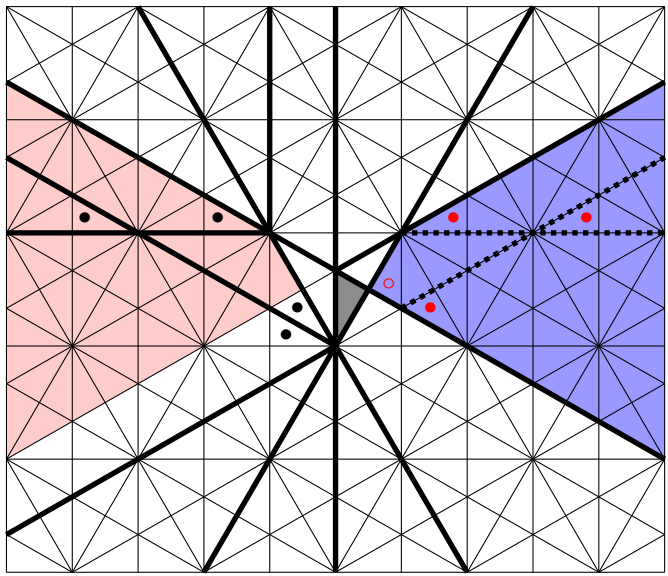
\begin{figure}[H]
\centering
\begin{tikzpicture}[scale=1]
\path [fill=red!20] (-0.433,0.75)--(-0.866,1.5)--(-4.33,3.5)--(-4.33,-1.5);
\path [fill=gray!90] (0,0) -- (0.433,0.75) -- (0,1) -- (0,0);
\path [fill=blue!40] (0.433,0.75)--(0.866,1.5)--(4.33,3.5)--(4.33,-1.5);
\draw(-4.33,4.5)--(4.33,4.5);
\draw(-4.33,3)--(4.33,3);
\draw(-4.33,1.5)--(4.33,1.5);
\draw(-4.33,0)--(4.33,0);
\draw(-4.33,-1.5)--(4.33,-1.5);
\draw(-4.33,-3)--(4.33,-3);
\draw(-4.33,-3)--(-4.33,4.5);
\draw(-3.464,-3)--(-3.464,4.5);
\draw(-2.598,-3)--(-2.598,4.5);
\draw(-1.732,-3)--(-1.732,4.5);
\draw(-.866,-3)--(-.866,4.5);
\draw(0,-3)--(0,4.5);
\draw(.866,-3)--(.866,4.5);
\draw(1.732,-3)--(1.732,4.5);
\draw(2.598,-3)--(2.598,4.5);
\draw(3.464,-3)--(3.464,4.5);
\draw(4.33,-3)--(4.33,4.5);
\draw(-4.33,3.5)--({-3*0.866},4.5);
\draw(-4.33,2.5)--({-1*0.866},4.5);
\draw(-4.33,1.5)--({1*0.866},4.5);
\draw(-4.33,.5)--({3*0.866},4.5);
\draw(-4.33,-.5)--(4.33,4.5);
\draw(-4.33,-1.5)--(4.33,3.5);
\draw(-4.33,-2.5)--(4.33,2.5);
\draw(-3.464,-3)--(4.33,1.5);
\draw(-1.732,-3)--(4.33,.5);
\draw(0,-3)--(4.33,-.5);
\draw(1.732,-3)--(4.33,-1.5);
\draw(3.464,-3)--(4.33,-2.5);
\draw(4.33,3.5)--({3*0.866},4.5);
\draw(4.33,2.5)--({1*0.866},4.5);
\draw(4.33,1.5)--({-1*0.866},4.5);
\draw(4.33,.5)--({-3*0.866},4.5);
\draw(4.33,-.5)--(-4.33,4.5);
\draw(4.33,-1.5)--(-4.33,3.5);
\draw(4.33,-2.5)--(-4.33,2.5);
\draw(3.464,-3)--(-4.33,1.5);
\draw(1.732,-3)--(-4.33,.5);
\draw(0,-3)--(-4.33,-.5);
\draw(-1.732,-3)--(-4.33,-1.5);
\draw(-3.464,-3)--(-4.33,-2.5);
\draw(-4.33,-1.5)--(-3.464,-3);
\draw(-4.33,1.5)--(-1.732,-3);
\draw(-4.33,4.5)--(0,-3);
\draw({-3*0.866},4.5)--(1.732,-3);
\draw({-1*0.866},4.5)--(3.464,-3);
\draw({1*0.866},4.5)--(4.33,-1.5);
\draw({3*0.866},4.5)--(4.33,1.5);
\draw(4.33,-1.5)--(3.464,-3);
\draw(4.33,1.5)--(1.732,-3);
\draw(4.33,4.5)--(0,-3);
\draw({3*0.866},4.5)--(-1.732,-3);
\draw({1*0.866},4.5)--(-3.464,-3);
\draw({-1*0.866},4.5)--(-4.33,-1.5);
\draw({-3*0.866},4.5)--(-4.33,1.5);
\draw[line width=2pt](0,-3)--(0,4.5);
\draw[line width=2pt]({-2*0.866},-3)--({3*0.866},4.5);
\draw[line width=2pt]({-5*0.866},-2.5)--(0,0);
\draw[line width=2pt]({-3*0.866},4.5)--({2*0.866},-3);
\draw[line width=2pt]({-5*0.866},2.5)--(0,0);
\draw[line width=2pt]({-5*0.866},3.5)--({5*0.866},-1.5);
\draw[line width=2pt]({-5*0.866},1.5)--(-0.866,1.5);
\draw[line width=2pt](-0.866,1.5)--(-0.866,4.5);
\draw[line width=2pt](0,1)--({5*0.866},3.5);
\draw[line width=2pt, dotted] (0.866,0.5)--({5*0.866},2.5);
\draw[line width=2pt, dotted] (0.866,1.5)--({5*0.866},1.5);
\node at (-0.5,0.5) {$\bullet$};
\node at (-0.65,0.15) {$\bullet$};
\node at (-1.55,1.7) {$\bullet$};
\node at (-3.3,1.7) {$\bullet$};
\node [color=red] at (0.7,0.82) {$\circ$};
\node [color=red] at (1.25,0.5) {$\bullet$};
\node [color=red] at (1.55,1.7) {$\bullet$};
\node [color=red] at (3.3,1.7) {$\bullet$};
\end{tikzpicture}
\caption{A simple refinement}\label{fig:simplerefinement}
\end{figure}
\end{exa}

The following algorithm is called the \textit{simple refinements algorithm}.

\begin{alg}\label{alg:regularisation}
Let $\scrP\in\scrP(W)$. Let $\scrP_0=\scrP\vee\scrD$. For $j\geq 1$, if there exists a pair $(P,s)$ with $P\in\scrP_{j-1}$ and $s\in S$ with $s\notin D_L(P)$ and $sP\not\subseteq P'$ for any $P'\in\scrP_{j-1}$, let $\scrP_{j}$ be the simple refinement of $\scrP_{j-1}$ based at the pair $(P,s)$. 
\end{alg}

We will show in Theorem~\ref{thm:newterminateatregularisation} below that if Algorithm~\ref{alg:regularisation} terminates in finite time, then the output of the algorithm is the regular completion of the input partition. The key observation is the following lemma.

\begin{lem}\label{lem:minimalrefinement}
Let $\scrP$ be a locally constant partition, and suppose that $\scrR$ is a regular partition with $\scrP\leq \scrR$. Let $\scrP\mapsto \scrP'$ by a simple refinement. Then $\scrP'\leq \scrR$. 
\end{lem}

\begin{proof}
We must show that each part of $\scrR$ is contained in a part of $\scrP'$. Suppose that the simple refinement $\scrP\mapsto\scrP'$ is based at the pair $(P,s)\in\scrP\times S$. Let $R$ be a part of $\scrR$. Then $R\subseteq Q$ for some part $Q$ of $\scrP$ (because $\scrP\leq \scrR$). If $Q\neq P$ then $Q$ is also a part of $\scrP'$ (by the definition of simple refinements), and we are done. So suppose that $Q=P$ and so $R\subseteq P$. Since $\scrR$ is regular, and since $s\notin D_L(R)$ (as $s\notin D_L(P)$) we have that $sR\subseteq R'$ for some part $R'$ of $\scrR$. Moreover, since $\scrP\leq \scrR$ and since $R'\cap R=\emptyset$ (by the locally constant condition) we have $R'\subseteq P'$ for some part $P'$ of $\scrP$ with $P'\in X$ (with $X$ as in Definition~\ref{defn:simplerefinement}). Thus $sR\subseteq R'\subseteq P'$, and so $R\subseteq sP'$. But also $R\subseteq P$, and so $R\subseteq P\cap sP'$, which is a part of $\scrP'$.
\end{proof}

\begin{thm}\label{thm:newterminateatregularisation}
Let $\scrP\in\scrP(W)$ and let $\widehat{\scrP}$ be the regular completion of~$\scrP$.
\begin{enumerate}
\item We have $\scrP\leq\widehat{\scrP}$. 
\item If Algorithm~\ref{alg:regularisation} terminates in finite time with output~$\scrQ$, then $\scrQ=\widehat{\scrP}$.
\item If $|\scrP|<\infty$ then Algorithm~\ref{alg:regularisation} terminates in finite time if and only if $|\widehat{\scrP}|<\infty$. 
\end{enumerate}
\end{thm}

\begin{proof}
Suppose first that the simple refinements algorithm terminates in finite time. Then there is a chain of partitions $\scrP\vee\scrD=\scrP_0\leq\scrP_1\leq\scrP_2\leq\cdots\leq\scrP_n$ with $\scrP_{j-1}\mapsto\scrP_j$ by a simple refinement for $1\leq j\leq n$ and with $\scrP_n$ regular. By Lemma~\ref{lem:minimalrefinement} every regular partition $\scrR$ with $\scrP\leq\scrR$ satisfies $\scrP_n\leq\scrR$, and since $\widehat{\scrP}$ is the meet of such partitions (by definition), and since $\scrP_n$ is regular, we have $\widehat{\scrP}=\scrP_n$. This proves (2), and also proves (1) in the case that the simple refinements algorithm terminates in finite time.

Suppose now that the simple refinements algorithm does not terminate in finite time. Then one can construct an infinite chain $\scrP\vee\scrD=\scrP_0\leq\scrP_1\leq\scrP_2\leq\cdots$ of partitions with $\scrP_{j-1}\mapsto\scrP_j$ by a simple refinement for all $j\geq 1$. Moreover, it is clear that this sequence can be chosen such that for each $N>0$ there exists $M>0$ such that the partition $\scrP_M$ restricted to the ball $B_N=\{w\in W\mid \ell(w)\leq N\}$ is regular (by this we shall mean that if $P$ is a part of $\scrP_M$ and $s\notin D_L(P)$ then $sP\cap B_N$ is contained in some $P'\cap B_N$ with $P'$ a part of $\scrP_M$). Consider the join
$$
\scrP_{\infty}=\bigvee_{n\geq 0}\scrP_n.
$$
Every regular partition $\scrR$ with $\scrP\leq \scrR$ satisfies $\scrP_{\infty}\leq\scrR$ (for if not, then since $\scrP_0\leq\scrP_1\leq\cdots$ there is some $n$ with $\scrP_n\not\leq \scrR$, contradicting Lemma~\ref{lem:minimalrefinement}). Moreover, $\scrP_{\infty}$ is regular because by construction the restriction of $\scrP_{\infty}$ to each finite ball is regular. Thus, as in the previous paragraph, $\widehat{\scrP}=\scrP_{\infty}$, and the proof of (1) is complete.

To prove (3), note that if $|\scrP|<\infty$ and the algorithm terminates in finite time, then the output partition (which is $\widehat{\scrP}$ by (2)) has finitely many parts by Proposition~\ref{prop:finiteparts}. Conversely, if $|\widehat{\scrP}|<\infty$, and if $\scrP\vee\scrD=\scrP_0\mapsto \scrP_1\mapsto\cdots \mapsto\scrP_n$ is a sequence of simple refinements, then since $|\scrP|<|\scrP_1|<\cdots <|\scrP_n|$, and since $\scrP_n\leq \widehat{\scrP}$ (by Lemma~\ref{lem:minimalrefinement}) we have $n\leq |\widehat{\scrP}|-|\scrP|$, and so Algorithm~\ref{alg:regularisation} terminates in finite time.
\end{proof}

\begin{rem}
Note that, as a consequence of Theorem~\ref{thm:newterminateatregularisation}, if Algorithm~\ref{alg:regularisation} terminates in finite time then the output partition is independent of the order of the simple refinements chosen.
\end{rem}

We note, in passing, the following corollary. Recall (c.f. \cite[Section~7.1]{DP:02}) that a \textit{closure operator} on a partially ordered set $(X,\leq)$ is a map $c:X\to X$ satisfying (1) $x\leq c(x)$ for all $x\in X$, (2) if $x,y\in X$ with $x\leq y$ then $c(x)\leq c(y)$, and (3) $c(c(x))=c(x)$ for all $x\in X$. The \textit{closed elements} of the closure operator $c:X\to X$ are the elements $x\in X$ with $c(x)=x$. 

\begin{cor}
The map $c:\scrP(W)\to\scrP(W)$ with $c(\scrP)=\widehat{\scrP}$ is a closure operator on $\scrP(W)$. The set of closed elements is precisely $\Preg$. 
\end{cor}

\begin{proof}
By Theorem~\ref{thm:newterminateatregularisation} we have $\scrP\leq c(\scrP)$ for all $\scrP\in\scrP(W)$. If $\scrP,\scrQ\in\scrP(W)$ with $\scrP\leq\scrQ$ then $\{\scrR\in\Preg\mid \scrP\leq\scrR\}\supseteq \{\scrR\in\Preg\mid \scrQ\leq\scrR\}$, and hence from the definition of regular completion we have $c(\scrP)\leq c(\scrQ)$. Since $\widehat{\scrP}$ is regular we have $c(c(\scrP))=c(\scrP)$, and so $c$ is a closure operator. The closed elements are those partitions $\scrP\in\scrP(W)$ with $\scrP=\widehat{\scrP}$, and these are precisely the regular partitions.
\end{proof}

%
%
%
%
%

\newpage

We note that the simple refinements algorithm (Algorithm~\ref{alg:regularisation}) may not terminate in finite time, even in the case that the input partition $\scrP$ has finitely many parts, as the following example shows.

\begin{exa}\label{exa:nonterminate}
Let $W=\langle s,t\mid s^2=t^2=e\rangle$ be the infinite dihedral group. Let $\scrP=\{P_0,P_1,P_2,P_3\}$ be the partition of $W$ into $4$ parts, with $P_0=\{e\}$, $P_1=\{w\in W\mid D_L(w)=\{s\}\}$, $P_2=\{w\in W\mid D_L(w)=\{t\}\text{ and }\ell(w)\in N\}$, and $P_3=\{w\in W\mid D_L(w)=\{t\}\text{ and }\ell(w)\notin N\}$, where $N=\{n(n+1)\mid n\in\mathbb{Z}_{>0}\}$. This is illustrated in Figure~\ref{fig:nonterminate}, with $P_0$ grey, $P_1$ green, $P_2$ red, and $P_3$ blue.

\begin{figure}[H]
\centering
\begin{tikzpicture}[scale=0.5]
\path [fill=red!30] (-1,0) -- (-2,0) -- (-2,0.6) -- (-1,0.6);
\path [fill=red!30] (3,0) -- (2,0) -- (2,0.6) -- (3,0.6);
\path [fill=red!30] (9,0) -- (8,0) -- (8,0.6) -- (9,0.6);
\path [fill=gray!90] (-3,0) -- (-4,0) -- (-4,0.6) -- (-3,0.6);
\path [fill=blue!40] (-2,0)--(-3,0)--(-3,0.6)--(-2,0.6);
\path [fill=blue!40] (2,0)--(-1,0)--(-1,0.6)--(2,0.6);
\path [fill=blue!40] (8,0)--(3,0)--(3,0.6)--(8,0.6);
\path [fill=blue!40] (13.5,0)--(9,0)--(9,0.6)--(13.5,0.6);
\path [fill=ForestGreen!40] (-4,0)--(-13.5,0)--(-13.5,0.6)--(-4,0.6);
\draw (-13.5,0)--(13.5,0);
\draw (-13,0)--(-13,0.6);
\draw (-12,0)--(-12,0.6);
\draw (-11,0)--(-11,0.6);
\draw (-10,0)--(-10,0.6);
\draw (-9,0)--(-9,0.6);
\draw (-8,0)--(-8,0.6);
\draw (-7,0)--(-7,0.6);
\draw (-6,0)--(-6,0.6);
\draw (-5,0)--(-5,0.6);
\draw (-4,0)--(-4,0.6);
\draw (-3,0)--(-3,0.6);
\draw (-2,0)--(-2,0.6);
\draw (-1,0)--(-1,0.6);
\draw (0,0)--(0,0.6);
\draw (1,0)--(1,0.6);
\draw (2,0)--(2,0.6);
\draw (3,0)--(3,0.6);
\draw (4,0)--(4,0.6);
\draw (5,0)--(5,0.6);
\draw (6,0)--(6,0.6);
\draw (7,0)--(7,0.6);
\draw (8,0)--(8,0.6);
\draw (9,0)--(9,0.6);
\draw (10,0)--(10,0.6);
\draw (11,0)--(11,0.6);
\draw (12,0)--(12,0.6);
\draw (13,0)--(13,0.6);
\node at (-3.5,0.275) {$e$};
\node at (-4.5,0.275) {$s$};
\node at (-2.5,0.32) {$t$};
\end{tikzpicture}
\caption{The simple refinements algorithm does not terminate in finite time}\label{fig:nonterminate}
\end{figure}

\noindent We claim that Algorithm~\ref{alg:regularisation} does not terminate in finite time when applied to $\scrP$. To prove this it is sufficient to show that the regular completion $\widehat{\scrP}$ has infinitely many parts (see Proposition~\ref{prop:finiteparts} and Theorem~\ref{thm:newterminateatregularisation}). For $n\geq 1$ write $w_n=tst\cdots s$ with $\ell(w_n)=n(n+1)$. Thus $P_2=\{w_1,w_2,w_3,\ldots\}$. We claim that if $i\neq j$ then $w_i$ and $w_j$ do not lie in a common part of $\widehat{\scrP}$. Suppose that $i<j$, and that $\{w_i,w_j\}$ is contained in a part $Q_0$ of $\widehat{\scrP}$. Since $s\notin D_L(Q_0)$ we have that $\{sw_i,sw_j\}$ is contained in a part $Q_1$ of $\widehat{\scrP}$ (as the regular completion is regular). Continuing, we have that $\{tsw_i,tsw_j\}$ is contained in a part $Q_2$ of $\widehat{\scrP}$, and so on. Writing $v=t\cdots sts$ with $\ell(v)=2(i+1)$ it follows that $\{vw_i,vw_j\}$ is contained in a part $Q_{i+1}$ of $\widehat{\scrP}$. Note that $vw_i=w_{i+1}$ and $j(j+1)<\ell(vw_j)=j(j+1)+2(i+1)<j(j+1)+2(j+1)=(j+1)(j+2)$. Thus $vw_i\in P_2$ and $vw_j\in P_3$ are in different parts of $\scrP$, and hence in different parts of the refinement $\widehat{\scrP}$, a contradiction. Hence the result.
\end{exa}

We now provide a sufficient condition for Algorithm~\ref{alg:regularisation} to terminate in finite time. Given a partition $\scrP$, we define the \textit{roots of $\scrP$} to be the set $\Phi(\scrP)$ of roots $\beta\in\Phi^+$ such that there exist parts $P_1\neq P_2$ of $\scrP$ and elements $w\in P_1$ and $s\in S$ such that $ws\in P_2$ and $w\alpha_s=\pm\beta$. Geometrically, this means that the wall $H_{\beta}$ of the Coxeter complex separates the chambers of $P_1$ from the chambers of $P_2$, and that $P_1\cap P_2$ (intersection of simplical complexes) contains a panel (codimension~$1$ simplex) of $H_{\beta}$. For example, if $\scrP$ is the hyperplane partition induced by $\Lambda$, then one can easily check that $\Phi(\scrP)=\Lambda$. 

Note that if $|\Phi(\scrP)|<\infty$ then $|\scrP|<\infty$, however the converse is false (see, for example, the infinite dihedral example in Example~\ref{exa:nonterminate}).

\begin{thm}\label{thm:finitetermination}
Let $\scrP$ be a locally constant partition with $|\Phi(\scrP)|<\infty$. Then Algorithm~\ref{alg:regularisation} terminates in finite time, outputting the regular completion $\widehat{\scrP}$, and moreover $|\widehat{\scrP}|<\infty$. 
\end{thm}

\begin{proof}
Since $\cE_0\subseteq \cE_1\subseteq\cdots$ is a filtration of $\Phi^+$, and since $|\Phi(\scrP)|<\infty$, there is $n\geq 0$ such that 
$\Phi(\scrP)\subseteq \cE_{n}$. From the definition of $\Phi(\scrP)$ it is clear that $\scrS_{n}$ is a refinement of $\scrP$, and in particular, $|\scrP|<\infty$ (see Proposition~\ref{prop:partsdescription}). Then, by Lemma~\ref{lem:minimalrefinement}, if $\scrP\to\scrP'$ by a simple refinement we have $\scrP'\leq\scrS_n$. Since $|\scrP|<|\scrP'|<|\scrS_n|$ Algorithm~\ref{alg:regularisation} must terminate after finitely many iterations (at most $|\scrS_n|-|\scrP|$ iterations in fact). The output is $\widehat{\scrP}$ and $|\widehat{\scrP}|<\infty$, by Theorem~\ref{thm:newterminateatregularisation}.
\end{proof}

The following important corollary is a key ingredient in the proof of Theorem~\ref{thm:main1}. 

\begin{cor}\label{cor:regularlattice1} We have $\scrT=\widehat{\scrD}$.
\end{cor}

\begin{proof}
Since $\scrT$ is regular (see Theorem~\ref{thm:regularpartitions}) we have $\scrD\leq \scrT$ (see Lemma~\ref{lem:Sminimal}). Hence by Lemma~\ref{lem:minimalrefinement} and the fact that Algorithm~\ref{alg:regularisation} terminates in finite time (Theorem~\ref{thm:finitetermination}), we have $\widehat{\scrD}\leq \scrT$. Since $\widehat{\scrD}$ is regular Corollary~\ref{cor:Tisminimal} gives $\scrT\leq \widehat{\scrD}$, and hence the result.
\end{proof}

Combining Corollary~\ref{cor:regularlattice1} with Algorithm~\ref{alg:regularisation} and Theorem~\ref{thm:finitetermination} we obtain an algorithmic way to compute the cone type partition~$\scrT$ by starting with $\scrD$ and applying simple refinements. However it is more efficient to instead apply the following theorem, allowing us to start with $\scrJ$ instead of $\scrD$ (see Example~\ref{ex:completingB2} below).

\begin{thm}\label{thm:spherical1}
We have $\scrJ\leq \scrT$, and $\widehat{\scrJ}=\scrT$.
\end{thm}

\begin{proof}
To prove that $\scrJ\leq \scrT$ we must show that if $x,y\in W$ with $T(x^{-1})=T(y^{-1})$, then $\Phisph(x)=\Phisph(y)$. Suppose that $\Phisph(x)\neq\Phisph(y)$. Then, after interchanging the roles of $x$ and $y$ if neccessary, we may assume that there is a root $\beta\in\Phisph^+$ with $\beta\in\Phi(y)\backslash\Phi(x)$. Let $J=\mathrm{supp}(\beta)$. Writing $x=uv$ and $y=u'v'$ with $u,u'\in W_J$ and $v,v'\in W^J$, we have $\Phi_J(x)=\Phi(u)$ and $\Phi_J(y)=\Phi(u')$ (see Lemma~\ref{lem:rootsdecomposition}), and so $u\neq u'$. Thus $T(x^{-1})\neq T(y^{-1})$ by Corollary~\ref{cor:distinguishconetypes}.

Since $\scrJ\leq \scrT$ we have $\widehat{\scrJ}\leq \scrT$, but also $\scrT\leq \widehat{\scrJ}$ by Corollary~\ref{cor:Tisminimal}.
\end{proof}

\begin{exa}\label{ex:completingB2}
Figure~\ref{fig:completingB2} illustrates the calculation of $\scrT$ for $\tilde{\sB}_2$ using Theorem~\ref{thm:spherical1} and Algorithm~\ref{alg:regularisation}. Let $s$ (respectively, $t$) be the reflection in the vertical (respectively, horizontal) wall bounding the fundamental chamber. The spherical partition $\scrJ$ is shown in Figure~\ref{fig:completingB2}(a) (in solid heavy lines). Let $P_0$ be the part of $\scrJ$ shaded in blue, and then $sP_0$ is shaded red. The partition $\scrP_1$ obtained by applying the simple refinement based at $(P_0,s)$ to $\scrJ$ is shown in Figure~\ref{fig:completingB2}(a) as the union of the solid and dotted heavy lines. Similarly, Figure~\ref{fig:completingB2}(b) shows the simple refinement $\scrP_1\to\scrP_2$ based at $(P_1,s)$ (with $P_1\in\scrP_1$ shaded blue), Figure~\ref{fig:completingB2}(c) shows the simple refinement $\scrP_2\to\scrP_3$ based at $(P_2,t)$, and Figure~\ref{fig:completingB2}(d) shows the simple refinement $\scrP_3\to\scrP_4$ based at $(P_3,t)$. Since $\scrP_4$ is regular we have $\scrP_4=\widehat{\scrJ}=\scrT$ (by Theorems~\ref{thm:newterminateatregularisation} and~\ref{thm:spherical1}). 
\begin{figure}[H]
\centering
\subfigure[$\scrJ\mapsto\scrP_1$]{
\begin{tikzpicture}[scale=0.65]
\path [fill=gray!90] (0,0) -- (1,1) -- (0,1) -- (0,0);
\path [fill=blue!40] (1,1)--(1,5)--(0,5)--(0,2)--(1,1);
\path [fill=red!20] (-1,1)--(-1,5)--(0,5)--(0,2)--(-1,1);
\draw (-4,-2) -- (5,-2); 
\draw (-4,-1) -- (5,-1);
\draw [line width=2pt](-4,0) -- (5,0);
\draw [line width=2pt](-4,1) -- (5,1);
\draw (-4,2) -- (5,2);
\draw (-4,3) -- (5,3); 
\draw (-4,4) -- (5,4);
\draw (-4,5) -- (5,5);
\draw (-4,-3) -- (5,-3);
\draw (-4,-4) -- (5,-4);
\draw (-2,-4) -- (-2,5);
\draw (-1,-4) -- (-1,5);
\draw [line width=2pt](0,-4) -- (0,5);
\draw [line width=2pt](1,-4) -- (1,5);
\draw (2,-4) -- (2,5);
\draw (-4,-4) -- (-4,5);
\draw (-3,-4) -- (-3,5);
\draw (3,-4) -- (3,5);
\draw (4,-4) -- (4,5);
\draw (5,-4) -- (5,5);
\draw (-4,4)--(-3,5);
\draw (-4,2)--(-1,5);
\draw (-4,0) -- (1,5);
\draw (-4,-2) -- (3,5);
\draw [line width=2pt](-4,-4) -- (5,5);
\draw (-2,-4) -- (5,3);
\draw (0,-4) -- (5,1);
\draw (2,-4)--(5,-1);
\draw (4,-4)--(5,-3);
\draw (3,5)--(5,3);
\draw (1,5)--(5,1);
\draw (-1,5)--(5,-1);
\draw[line width=2pt] (-3,5)--(5,-3);
\draw[line width=2pt] (-4,4)--(4,-4);
\draw (-4,2)--(2,-4);
\draw (-4,0)--(0,-4);
\draw (-4,-2)--(-2,-4);
\draw[dotted, line width=2pt] (0,2)--(1,3);
\end{tikzpicture}
}\qquad
\subfigure[$\scrP_1\mapsto\scrP_2$]{
\begin{tikzpicture}[scale=0.65]
\path [fill=gray!90] (0,0) -- (1,1) -- (0,1) -- (0,0);
\path [fill=blue!40] (1,1)--(1,5)--(5,5)--(1,1);
\path [fill=red!20] (-1,1)--(-1,5)--(-4,5)--(-4,4)--(-1,1);
\draw (-4,-2) -- (5,-2); 
\draw (-4,-1) -- (5,-1);
\draw [line width=2pt](-4,0) -- (5,0);
\draw [line width=2pt](-4,1) -- (5,1);
\draw (-4,2) -- (5,2);
\draw (-4,3) -- (5,3); 
\draw (-4,4) -- (5,4);
\draw (-4,5) -- (5,5);
\draw (-4,-3) -- (5,-3);
\draw (-4,-4) -- (5,-4);
\draw (-2,-4) -- (-2,5);
\draw (-1,-4) -- (-1,5);
\draw [line width=2pt](0,-4) -- (0,5);
\draw [line width=2pt](1,-4) -- (1,5);
\draw (2,-4) -- (2,5);
\draw (-4,-4) -- (-4,5);
\draw (-3,-4) -- (-3,5);
\draw (3,-4) -- (3,5);
\draw (4,-4) -- (4,5);
\draw (5,-4) -- (5,5);
\draw (-4,4)--(-3,5);
\draw (-4,2)--(-1,5);
\draw (-4,0) -- (1,5);
\draw (-4,-2) -- (3,5);
\draw [line width=2pt](-4,-4) -- (5,5);
\draw (-2,-4) -- (5,3);
\draw (0,-4) -- (5,1);
\draw (2,-4)--(5,-1);
\draw (4,-4)--(5,-3);
\draw (3,5)--(5,3);
\draw (1,5)--(5,1);
\draw (-1,5)--(5,-1);
\draw[line width=2pt] (-3,5)--(5,-3);
\draw[line width=2pt] (-4,4)--(4,-4);
\draw (-4,2)--(2,-4);
\draw (-4,0)--(0,-4);
\draw (-4,-2)--(-2,-4);
\draw[line width=2pt] (0,2)--(1,3);
\draw[dotted, line width=2pt] (1,3)--(3,5);
\end{tikzpicture}
}
\subfigure[$\scrP_2\mapsto\scrP_3$]{
\begin{tikzpicture}[scale=0.65]
\path [fill=gray!90] (0,0) -- (1,1) -- (0,1) -- (0,0);
\path [fill=blue!40] (0,0)--(-4,0)--(-4,1)--(-1,1)--(0,0);
\path [fill=red!20] (0,2)--(-4,2)--(-4,1)--(-1,1)--(0,2);
\draw (-4,-2) -- (5,-2); 
\draw (-4,-1) -- (5,-1);
\draw [line width=2pt](-4,0) -- (5,0);
\draw [line width=2pt](-4,1) -- (5,1);
\draw (-4,2) -- (5,2);
\draw (-4,3) -- (5,3); 
\draw (-4,4) -- (5,4);
\draw (-4,5) -- (5,5);
\draw (-4,-3) -- (5,-3);
\draw (-4,-4) -- (5,-4);
\draw (-2,-4) -- (-2,5);
\draw (-1,-4) -- (-1,5);
\draw [line width=2pt](0,-4) -- (0,5);
\draw [line width=2pt](1,-4) -- (1,5);
\draw (2,-4) -- (2,5);
\draw (-4,-4) -- (-4,5);
\draw (-3,-4) -- (-3,5);
\draw (3,-4) -- (3,5);
\draw (4,-4) -- (4,5);
\draw (5,-4) -- (5,5);
\draw (-4,4)--(-3,5);
\draw (-4,2)--(-1,5);
\draw (-4,0) -- (1,5);
\draw (-4,-2) -- (3,5);
\draw [line width=2pt](-4,-4) -- (5,5);
\draw (-2,-4) -- (5,3);
\draw (0,-4) -- (5,1);
\draw (2,-4)--(5,-1);
\draw (4,-4)--(5,-3);
\draw (3,5)--(5,3);
\draw (1,5)--(5,1);
\draw (-1,5)--(5,-1);
\draw[line width=2pt] (-3,5)--(5,-3);
\draw[line width=2pt] (-4,4)--(4,-4);
\draw (-4,2)--(2,-4);
\draw (-4,0)--(0,-4);
\draw (-4,-2)--(-2,-4);
\draw[dotted, line width=2pt] (-2,0)--(-1,1);
\draw[line width=2pt] (0,2)--(3,5);
\end{tikzpicture}
}\qquad
\subfigure[$\scrP_3\mapsto\scrT$]{
\begin{tikzpicture}[scale=0.65]
\path [fill=gray!90] (0,0) -- (1,1) -- (0,1) -- (0,0);
\path [fill=blue!40] (0,0)--(-4,-4)--(-4,0)--(0,0);
\path [fill=red!20] (0,2)--(-4,2)--(-4,5)--(-3,5)--(0,2);
\draw (-4,-2) -- (5,-2); 
\draw (-4,-1) -- (5,-1);
\draw [line width=2pt](-4,0) -- (5,0);
\draw [line width=2pt](-4,1) -- (5,1);
\draw (-4,2) -- (5,2);
\draw (-4,3) -- (5,3); 
\draw (-4,4) -- (5,4);
\draw (-4,5) -- (5,5);
\draw (-4,-3) -- (5,-3);
\draw (-4,-4) -- (5,-4);
\draw (-2,-4) -- (-2,5);
\draw (-1,-4) -- (-1,5);
\draw [line width=2pt](0,-4) -- (0,5);
\draw [line width=2pt](1,-4) -- (1,5);
\draw (2,-4) -- (2,5);
\draw (-4,-4) -- (-4,5);
\draw (-3,-4) -- (-3,5);
\draw (3,-4) -- (3,5);
\draw (4,-4) -- (4,5);
\draw (5,-4) -- (5,5);
\draw (-4,4)--(-3,5);
\draw (-4,2)--(-1,5);
\draw (-4,0) -- (1,5);
\draw (-4,-2) -- (3,5);
\draw [line width=2pt](-4,-4) -- (5,5);
\draw (-2,-4) -- (5,3);
\draw (0,-4) -- (5,1);
\draw (2,-4)--(5,-1);
\draw (4,-4)--(5,-3);
\draw (3,5)--(5,3);
\draw (1,5)--(5,1);
\draw (-1,5)--(5,-1);
\draw[line width=2pt] (-3,5)--(5,-3);
\draw[line width=2pt] (-4,4)--(4,-4);
\draw (-4,2)--(2,-4);
\draw (-4,0)--(0,-4);
\draw (-4,-2)--(-2,-4);
\draw[line width=2pt] (-2,0)--(-1,1);
\draw[dotted, line width=2pt] (-4,-2)--(-2,0);
\draw[line width=2pt] (0,2)--(3,5);
\end{tikzpicture}
}
\caption{Computing $\scrT$ using $\scrT=\widehat{\scrJ}$}\label{fig:completingB2}
\end{figure}
\end{exa}

\newpage
We note the following corollary.

\begin{cor}\label{cor:someobservations}
The following are equivalent. 
\begin{enumerate}
\item $\cE=\Phisph^+$,
\item $\scrS_0=\scrJ$, 
\item $\scrT=\scrS_0$.
\end{enumerate}
\end{cor}

\begin{proof}
If $\cE=\Phisph^+$ then $\scrS_0=\scrJ$ directly from the definitions. If $\scrS_0=\scrJ$ then $\scrJ$ is regular (as $\scrS_0$ is regular by Theorem~\ref{thm:regularpartitions}). Thus $\scrT\leq\scrJ$ by Corollary~\ref{cor:Tisminimal}. But  $\scrJ\leq\scrT$ by Theorem~\ref{thm:spherical1}, and so equality holds. Hence $\scrT=\scrJ=\scrS_0$. 

On the other hand, suppose that $\scrT=\scrS_0$. Thus the Brink-Howlett (ie, the $0$-canonical) automaton is minimal, and so by \cite[Theorem~1]{PY:19} we have $\cE=\Phisph^+$.
\end{proof}

\section{Gated partitions}\label{sec:gatedpartitions}

In this section we introduce the notion of a \textit{gated partition}. In a gated partition $\scrP$, each part $P\in\scrP$ contains a unique ``gate'' $g$ with the property that $g\peq x$ for all $x\in P$, and we write $\Gamma(\scrP)$ for the set of all gates. We show, in Section~\ref{sec:simplerefinementsgates}, that simple refinements preserve the gated property (provided an additional hypothesis, convexity, is assumed). It follows that $\scrT$ is gated, proving Theorem~\ref{thm:main1}.

\subsection{Convex partitions}\label{sec:convex}

We begin with a discussion of convexity.

\begin{defn}
A subset $X\subseteq W$ is \textit{convex} if for all $x,y\in X$, and all reduced expressions $x^{-1}y=s_1\cdots s_n$, each element $xs_1\cdots s_j$ with $0\leq j\leq n$ is in $X$. A partition $\scrP$ of $W$ is \textit{convex} if each part $P\in \scrP$ is convex. 
\end{defn}

More intuitively, $X\subseteq W$ is convex if for all $x,y\in X$, each chamber that lies on a minimal length gallery from $x$ to $y$ in the Coxeter complex lies in $X$. Here \textit{gallery} means a sequence of adjacent chambers, starting at $x$, and ending at $y$. The following well-known result gives a useful characterisation of convexity. 

\begin{lem}\cite[Proposition~3.94]{AB:08}\label{lem:intersection}
A subset of $W$ is convex if and only if it is an intersection of half-spaces.
\end{lem}

The above characterisation of convex sets leads to the following proposition. 

\begin{prop}\label{prop:convexitybasics}
The following are convex:
\begin{enumerate}
\item the intersection of convex sets;
\item hyperplane partitions;
\item cones and cone types;
\item the cone type partition~$\scrT$.
\end{enumerate}
\end{prop}

\begin{proof}
(1) is clear from Lemma~\ref{lem:intersection}. To prove (2), note that if $\Lambda\subseteq\Phi^+$ then the part of the hyperplane partition $\scrH(\Lambda)$ containing $w\in W$ is 
$$
P=\bigg(\bigcap_{\beta\in\Lambda_+}H_{\beta}^+\bigg)\cap\bigg(\bigcap_{\beta\in\Lambda_-}H_{\beta}^-\bigg)
$$
where $\Lambda_{\pm}=\{\beta\in\Lambda\mid w\in H_{\beta}^{\pm}\}$, and use (1). Part (3) is clear from (1) and the formula in Theorem~\ref{thm:geometry1}, and the fact that $C(w)=wT(w)$. Finally, the partition~$\scrT$ is convex by the description of the parts given in Theorem~\ref{thm:conetypeprojection} combined with~(1).
\end{proof}

In particular, note that Proposition~\ref{prop:convexitybasics}(2) shows that $\scrD$, $\scrJ$, and $\scrS_n$ are all convex.

\begin{rem}\label{rem:garsideconvexity}
Based on examples, we expect that if $B$ is a Garside shadow, then $\scrG_B$ is convex. However we have only proved that $\scrG_B$ satisfies a weaker form of convexity (see Proposition~\ref{prop:garsidegated}). See Remark~\ref{rem:conicalconvex} for further discussion.
\end{rem}

%

\subsection{Gated partitions}\label{sec:gated} We now define gated partitions, provide some of the main examples, and prove some basic properties of the gates of a regular gated partition.

\begin{defn}\label{defn:gatedpartitions}
A subset $X\subseteq W$ is \textit{gated} if there exists $g\in X$ such that $g\peq x$ for all $x\in X$. The element $g$ is called a \textit{gate} of $X$. A partition $\scrP$ of $W$ is called \textit{gated} if each part $P\in\scrP$ is gated. If $\scrP$ is gated we write $\Gamma(\scrP)$ for the set of all gates of $\scrP$. 
\end{defn}

\begin{lem}\label{lem:uniquegate}
Every gated subset $X\subseteq W$ has a unique gate, and this gate is the unique minimal length element of~$X$. 
\end{lem}

\begin{proof}
If $X\subseteq W$ is gated, and if $g_1,g_2\in X$ are gates, then $g_1\peq g_2$ and $g_2\peq g_1$. Hence $g_1=g_2$. Let $g\in X$ be the unique gate. Since $g\peq x$ for all $x\in X$ the element $g$ is the unique minimal length element of~$X$.
\end{proof}

We will show in Corollary~\ref{cor:gateexist} that $\scrT$ is gated (this proves Theorem~\ref{thm:main1}). We first develop some basic theory for gated partitions, and provide simple examples. 

The following weaker notion of convexity is useful for studying Garside partitions.

\begin{defn}
Let $X\subseteq W$ be gated, with gate $g$. We say that $X$ is \textit{weakly convex} if $g\peq y\peq x$ and $x\in X$ implies that $y\in X$. A gated partition $\scrP$ is called \textit{weakly convex} if each part $P\in\scrP$ is weakly convex. 
\end{defn}

It is obvious that if a gated set $X\subseteq W$ (respectively a gated partition $\scrP\in\scrP(W)$) is convex, then $X$ (respectively $\scrP$) is also weakly convex, however the converse is clearly false. For example, consider the partition $W=\{\{e,s,t,st,ts\},\{sts\}\}$ of the $\mathsf{A}_2$ Coxeter group. This gated partition is weakly convex but not convex. 

The gates of a gated weakly convex partition have the following characterisation.

\begin{prop}\label{prop:weakconvexity}
If $X\subseteq W$ is gated and weakly convex then the gate $g$ of $X$ is characterised by the properties $g\in X$ and $gs\notin X$ for all $s\in D_R(g)$. 
\end{prop}

\begin{proof}
We have $gs\notin X$ for all $s\in D_R(g)$ as $g$ has minimal length in $X$ (by Lemma~\ref{lem:uniquegate}). Conversely, suppose that $x\in X$ is not the gate. Then $g\peq x$ gives $x=gs_1\cdots s_n$ with $n\geq 1$ and $\ell(x)=\ell(g)+n$. Then $s_n\in D_R(x)$, and $g\peq xs_n\peq x$, and by weak convexity $xs_n\in X$. 
\end{proof}

The following is a simple, but important, example of a gated partition.

\begin{lem}\label{lem:Spartitiongated}
The partition $\scrD$ is a locally constant, convex, gated partition. Moreover, 
$$\Gamma(\scrD)=\{w_J\mid J\subseteq S\text{ is spherical}\}.$$ 
\end{lem}

\begin{proof}
The $S$-partition is locally constant by the description of the parts in Proposition~\ref{prop:partsdescription}, and it is convex by Proposition~\ref{prop:convexitybasics}. Let $J\subseteq S$ be spherical, and let $w_J$ be the longest element of $W_J$. Each $w\in W$ with $D_L(w)=J$ can be written as $w=w_Jv$ with $\ell(w)=\ell(w_J)+\ell(v)$ (see \cite[Proposition~2.17]{AB:08}), and hence $w_J\peq w$. Thus the part $D_L^{-1}(J)$ of $\scrD$ is gated, with gate~$w_J$.
\end{proof}

The following theorem, applied to the case $X=W$, shows that the join of convex gated partitions is again convex and gated. If $X\subseteq W$ then the notion of a gated partition of~$X$ has the obvious meaning.

\begin{thm}\label{thm:joingated}
Let $X\subseteq W$ be convex. Let $\scrP_i$, $i\in I$, be a family of convex (respectively weakly convex) gated partitions of $X$. Then the join $\scrP=\bigvee_{i\in I}\scrP_i$ is a convex (respectively weakly convex) gated partition of $X$. 
\end{thm}

\begin{proof}
Recall that the parts of $\scrP=\bigvee_{i\in I}\scrP_i$ are of the form $P=\bigcap_{i\in I}P_i$ with $P_i\in\scrP_i$ and $P\neq\emptyset$. Let $g_i$ be the gate of $P_i$. Since $P\neq\emptyset$ there is $w\in P$, and since $w\in P_i$ we have $g_i\peq w$ for all $i\in I$. Thus $\{g_i\mid i\in I\}$ is bounded, and so $g=\bigvee\{g_i\mid i\in I\}$ exists. Moreover, for any $w\in P$ we have $g\peq w$, and thus  $g_i\peq g\peq w$ (for all $i\in I$). We may assume $\scrP_i$ is weakly convex (for if $\scrP_i$ is convex then it is also weakly convex). Thus we have $g\in P_i$ for each $i\in I$, and so $g\in P$. Thus $P$ is gated with gate~$g$. Moreover, if $w\in P$ and $g\peq y\peq w$ then $g_i\peq g\peq y\peq w$ for all $i\in I$ and so $y\in P_i$ for all $i\in I$, giving $y\in P$. Thus $\scrP$ is weakly convex. Finally, it is clear that if each $\scrP_i$ is convex then $\scrP$ is convex (as the intersection of convex sets is convex, by Proposition~\ref{prop:convexitybasics}). 
\end{proof}

\begin{cor}\label{cor:JIsGated}
The spherical partition $\scrJ$ is gated (see Proposition~\ref{prop:Jgated} for a description of the set of gates of~$\scrJ$). 
\end{cor}

\begin{proof}
It is clear that for each spherical subset $J\subseteq S$ the partition $\scrJ_J=\scrH(\Phi_J)$ is convex (being a hyperplane partition, see Propostion~\ref{prop:convexitybasics}) and gated (with $\Gamma(\scrJ_J)=W_J$). Since 
$$
\scrJ=\bigvee_{J}\scrJ_J,
$$
with the union over spherical subsets $J\subseteq S$, the partition $\scrJ$ is gated by Theorem~\ref{thm:joingated}.
\end{proof}

In the following proposition we show that Garside partitions are gated (see Theorem~\ref{thm:conesgated} for a generalisation).

\begin{prop}\label{prop:garsidegated}
Let $B$ be a Garside shadow. The partition $\scrG_B$ is gated and weakly convex, with $\Gamma(\scrG_B)=B$. 
\end{prop}

\begin{proof}
Each part of $\scrG_B$ is of the form $\pi^{-1}_B(b)$ for some $b\in B$, and if $x\in \pi^{-1}_B(b)$ then $b\peq x$ by the definition of $\pi_B(x)$. Hence $\pi^{-1}_B(b)$ is gated with gate~$b$.

We now show that $\scrG_B$ is weakly convex. Suppose that $b\in B$ and $x\in\pi^{-1}_B(b)$, and that $b\peq y\peq x$. Let $\pi_B(y)=b'$. Thus $b'\peq y$, and so $b'\peq y\peq x$, giving $b'\peq \pi_B(x)=b$ (by definition of $\pi_B(x)$). On the other hand, since $b\peq y$ we have $b\peq b'$ (by definition of $\pi_B(y)$), and so $b=b'$. Thus $y\in \pi^{-1}_B(b)$.
\end{proof}

\begin{rem} We note the following.
\begin{enumerate}
\item It is unknown if the $n$-Shi partition $\scrS_n$ is gated (see Conjecture~\ref{conj:Shigated}).
\item The partitions in Figures~\ref{fig:B2partitions}, \ref{fig:G2partitions}, \ref{fig:A2partitions}, and~\ref{fig:nongarside} are all gated. However, we note that the gated property is in fact rather rare. For example, the partition in Figure~\ref{fig:nongated} is convex, regular, but not gated (the part shaded red has no gate).
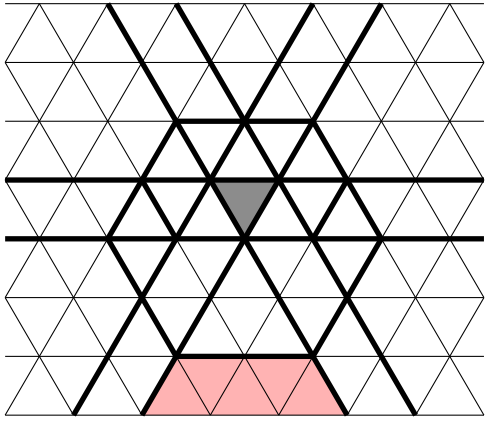
\begin{figure}[H]
\centering
\begin{tikzpicture}[scale=0.9]
    \path [fill=red!30] (-1,{-2*0.866})--(-1.5,{-3*0.866})--(1.5,{-3*0.866})--(1,{-2*0.866});
    \path [fill=gray!90] (0,0) -- (-0.5,0.866) -- (0.5,0.866) -- (0,0);
    \draw (2.5, {-3*0.866})--( 3.5, {-1*0.866} );
    \draw (1.5, {-3*0.866})--( 3.5, {1*0.866} );
    \draw  (0.5, {-3*0.866})--( 3.5, {3*0.866} );
    \draw  (-0.5, {-3*0.866})--( 3, {4*0.866} );
    \draw  [line width=2pt](-1.5, {-3*0.866})--( 2, {4*0.866} );
    \draw  [line width=2pt] (-2.5, {-3*0.866})--( 1, {4*0.866} );
    \draw  (-3.5, {-3*0.866})--(0, {4*0.866} );
    \draw  (-3.5, {-1*0.866})--(-1, {4*0.866} );
    \draw  (-3.5, {1*0.866})--(-2, {4*0.866} );
    \draw  (-3.5, {3*0.866})--(-3, {4*0.866} );
    \draw (-2.5, {-3*0.866})--( -3.5, {-1*0.866} );
    \draw (-1.5, {-3*0.866})--( -3.5, {1*0.866} );
    \draw  (-0.5, {-3*0.866})--( -3.5, {3*0.866} );
    \draw  (0.5, {-3*0.866})--( -3, {4*0.866} );
    \draw  [line width=2pt](1.5, {-3*0.866})--( -2, {4*0.866} );
    \draw  [line width=2pt] (2.5, {-3*0.866})--( -1, {4*0.866} );
    \draw  (3.5, {-3*0.866})--(0, {4*0.866} );
    \draw  (3.5, {-1*0.866})--(1, {4*0.866} );
    \draw  (3.5, {1*0.866})--(2, {4*0.866} );
    \draw  (3.5, {3*0.866})--(3, {4*0.866} );
    \draw (-3.5, -2.598)--( 3.5, -2.598);
    \draw (-3.5, -1.732)--( 3.5, -1.732);
    \draw (-3.5, -0.866)--( 3.5, -0.866);
    \draw[line width=2pt] (-3.5, 0)--( 3.5, 0);
    \draw (-3.5, 3.464)--( 3.5, 3.464 );
    \draw (-3.5, 2.598)--( 3.5, 2.598);
    \draw (-3.5, 1.732)--( 3.5, 1.732);
    \draw [line width=2pt] (-3.5, 0.866)--( 3.5, 0.866);
   \draw [line width=2pt] (-1,{-2*0.866})--(-2,0)--(-1,{2*0.866})--(1,{2*0.866})--(2,0)--(1,{-2*0.866})--(-1,{-2*0.866});
   \draw [line width=2pt]  (-1,0)--(-1.5,0.866);
   \draw [line width=2pt]  (1,0)--(1.5,0.866);
     %
%
\end{tikzpicture}
\caption{A regular convex partition that is not gated}\label{fig:nongated}
\end{figure}
\item \noindent There exist regular gated partitions that are not convex. For example let $W$ be the infinite dihedral group generated by $s$ and $t$, and let $\scrP=\{P_0,P_1,P_2,P_3,P_4\}$ be the partition of~$W$ with $P_0=\{e\}$, $P_1=\{w\in W\mid D_L(w)=\{s\}\text{ and $\ell(w)\notin 2\mathbb{Z}$}\}$, $P_2=\{w\in W\mid D_L(w)=\{s\}\text{ and $\ell(w)\in 2\mathbb{Z}$}\}$, $P_3=\{w\in W\mid D_L(w)=\{t\}\text{ and $\ell(w)\notin 2\mathbb{Z}$}\}$, and $P_4=\{w\in W\mid D_L(w)=\{t\}\text{ and $\ell(w)\in 2\mathbb{Z}$}\}$. This partition is regular and gated (with corresponding gates $g_0=e$, $g_1=s$, $g_2=st$, $g_3=t$, and $g_4=ts$) however it is clearly not convex.
\item The set of gates of a gated regular partition does not determine the partition. For example, let $\scrP$ be the gated regular partition of the infinite dihedral group from (2), and let $\scrP'=\{P_0',P_1',P_2',P_3',P_4'\}$ with $P_0'=\{e\}$, $P_1'=\{s\}$, $P_2'=\{w\in W\mid D_L(w)=\{s\}\text{ and }\ell(w)>1\}$, $P_3'=\{t\}$, and $P_4'=\{w\in W\mid D_L(w)=\{t\}\text{ and }\ell(w)>1\}$. Then $\scrP'$ is gated and regular, and $\Gamma(\scrP')=\Gamma(\scrP)$.
\end{enumerate}
\end{rem}

\newpage

We are particularly interested in partitions that are both gated and regular.

\begin{lem}\label{lem:minimalpath}
Let $\scrR$ be a regular gated partition, and let $\cA(\scrR)=(\scrR,\mu,\{e\})$ be the automaton constructed in Theorem~\ref{thm:regularautomaton}. Let $R\in\scrR$ with gate $g$, and let $(s_1,\ldots,s_n)$ be a reduced word. Then $g^{-1}=s_1\cdots s_n$ if and only if the path in $\cA(\scrR)$ starting at $\{e\}$ with edge labels $(s_1,\ldots,s_n)$ is of minimal length amongst all paths in $\cA(\scrR)$ from $\{e\}$ to~$R$. 
\end{lem}

\begin{proof}
By Theorem~\ref{thm:regularautomaton}, if $w=s_1\cdots s_n$ is reduced then the path in $\cA(\scrR)$ starting at $\{e\}$ with edge labels $(s_1,\ldots,s_n)$ ends at the part $R$ with $w^{-1}\in R$, and the lemma follows.
\end{proof}

\begin{thm}\label{thm:suffixclosuregates}
Let $\scrR$ be a regular gated partition. Then the set $\Gamma(\scrR)$ of gates is closed under taking suffix. Moreover, if $J\subseteq S$ is spherical then $W_J\subseteq \Gamma(\scrR)$, and in particular $S\subseteq \Gamma(\scrR)$. 
\end{thm}

\begin{proof}
Let $\cA(\scrR)=(\scrR,\mu,\{e\})$ be the automaton constructed in Theorem~\ref{thm:regularautomaton} and let $R\in\scrR$ with gate $g\in\Gamma(\scrR)$. Let $s\in D_L(g)$ and choose a reduced expression $g^{-1}=s_1\cdots s_{n-1}s$. Then, by Lemma~\ref{lem:minimalpath} the path 
$$
\{e\}=R_0\to_{s_1}R_1\to_{s_2}\cdots\to_{s_{n-1}}R_{n-1}\to_sR_n=R
$$
in $\cA(\scrR)$ from $\{e\}$ to $R$ with edge labels $(s_1,\ldots,s_{n-1},s)$ is of minimal length amongst all paths in $\cA(\scrR)$ from $\{e\}$ to $R$. Hence the path from $\{e\}$ to $R_{n-1}$ with edge labels $(s_1,\ldots,s_{n-1})$ is of minimal length amongst all paths from $\{e\}$ to $R_{n-1}$, and so by Lemma~\ref{lem:minimalpath} $g'=s_{n-1}\cdots s_1=sg$ is the gate of $R_{n-1}$. Thus $\Gamma(\scrR)$ is closed under taking suffixes by induction.

Since $\scrD\leq \scrR$ (by Lemma~\ref{lem:Sminimal}) and $\scrD$ has gates $w_J$ with $J\subseteq S$ spherical, it follows that $w_J\in\Gamma(\scrR)$. Since $\Gamma(\scrR)$ is closed under suffix we have $W_J\subseteq \Gamma(\scrR)$ for all spherical $J\subseteq S$. 
\end{proof}

\begin{exa}
We note that the set $\Gamma(\scrR)$ of gates of a regular gated partition is not necessarily closed under join (and hence $\Gamma(\scrR)$ is not necessarily a Garside shadow). An example, in type $\tilde{\sA}_2$, is given in Figure~\ref{fig:nongarside}. With $x,y\in\Gamma(\scrR)$ as shown, we have that $z=x\vee y$ exists, yet $z\notin\Gamma(\scrR)$. 
\begin{figure}[H]
\centering
\begin{tikzpicture}[scale=0.9]
    \path [fill=blue!30] (-2,0) -- (-1.5,-0.866) -- (-0.5,-0.866)--(0,-1.732)--(0.5,-0.866)-- (1.5,-0.866) -- (2,0);
    \path [fill=blue!30] (-1,0) -- (1,0) -- (1.5,0.866) -- (1,1.732)--(-1,1.732)--(-1.5,0.866)--(-1,0);
    \path [fill=blue!30] (0,1.732)--(-0.5,2.598)--(0.5,2.598);
    \path [fill=gray!90] (0,0) -- (-0.5,0.866) -- (0.5,0.866) -- (0,0);
    \draw (2.5, {-3*0.866})--( 3.5, {-1*0.866} );
    \draw (1.5, {-3*0.866})--( 3.5, {1*0.866} );
    \draw  (0.5, {-3*0.866})--( 3.5, {3*0.866} );
    \draw  (-0.5, {-3*0.866})--( 3, {4*0.866} );
    \draw  [line width=2pt](-1.5, {-3*0.866})--( 2, {4*0.866} );
    \draw  [line width=2pt] (-2.5, {-3*0.866})--( 1, {4*0.866} );
    \draw  (-3.5, {-3*0.866})--(0, {4*0.866} );
    \draw  (-3.5, {-1*0.866})--(-1, {4*0.866} );
    \draw  (-3.5, {1*0.866})--(-2, {4*0.866} );
    \draw  (-3.5, {3*0.866})--(-3, {4*0.866} );
    \draw (-2.5, {-3*0.866})--( -3.5, {-1*0.866} );
    \draw (-1.5, {-3*0.866})--( -3.5, {1*0.866} );
    \draw  (-0.5, {-3*0.866})--( -3.5, {3*0.866} );
    \draw  (0.5, {-3*0.866})--( -3, {4*0.866} );
    \draw  [line width=2pt](1.5, {-3*0.866})--( -2, {4*0.866} );
    \draw  [line width=2pt] (2.5, {-3*0.866})--( -1, {4*0.866} );
    \draw  (3.5, {-3*0.866})--(0, {4*0.866} );
    \draw  (3.5, {-1*0.866})--(1, {4*0.866} );
    \draw  (3.5, {1*0.866})--(2, {4*0.866} );
    \draw  (3.5, {3*0.866})--(3, {4*0.866} );
    \draw (-3.5, -2.598)--( 3.5, -2.598);
    \draw (-3.5, -1.732)--( 3.5, -1.732);
    \draw (-3.5, -0.866)--( 3.5, -0.866);
    \draw[line width=2pt] (-3.5, 0)--( 3.5, 0);
    \draw (-3.5, 3.464)--( 3.5, 3.464 );
    \draw (-3.5, 2.598)--( 3.5, 2.598);
    \draw (-3.5, 1.732)--( 3.5, 1.732);
    \draw [line width=2pt] (-3.5, 0.866)--( 3.5, 0.866);
    \draw [line width=2pt] (-1,0)--(-0.5,-0.866)--(0.5,-0.866)--(1,0);
     %
%
\node at (-1,-0.55) {$x$};
\node at (0,-1.2) {$y$};
\node at (-0.5,-1.4) {$z$};
\end{tikzpicture}
\caption{An convex regular gated partition that is not join-closed}\label{fig:nongarside}
\end{figure}
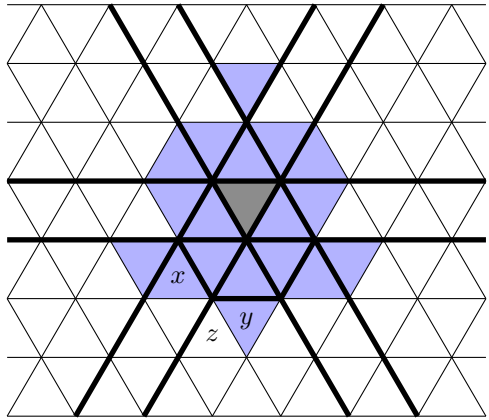
\end{exa}

We conclude this section by noting that for each gated partition~$\scrP$ one can define a ``projection map'' $\pi_{\scrP}:W\to\Gamma(\scrP)$ by
\begin{align*}
\pi_{\scrP}:W\to \Gamma(\scrP),\quad\text{with $\pi_{\scrP}(w)$ the gate of the part containing $w$}.
\end{align*}
The following lemma shows that this map generalises the projection map $\pi_B$ for a Garside shadow~$B$.

\begin{lem}\label{lem:li}
Let $B$ be a Garside shadow. Then $\pi_{\scrG_B}=\pi_B$.
\end{lem}

\begin{proof}
Let $x\in \pi_B^{-1}(b)$, with $b\in B$. Since $b$ is the gate of $\pi_B^{-1}(b)$ we have $\pi_{\scrG_B}(x)=b=\pi_B(x)$.
\end{proof}

In fact, for a regular gated partition the associated automaton can be described purely using the gates (rather than the parts of the partition), and the resulting formulation mirrors the Garside case from Theorem~\ref{thm:garsideautomaton}.

\begin{cor}
Let $\scrR$ be a regular gated partition. Define $\cA'(\scrR)=(\Gamma(\scrR),\mu',e)$ by 
$$
\mu'(g,s)=\begin{cases}
\pi_{\scrR}(sg)&\text{if $s\notin D_L(g)$}\\
\dagger&\text{if $s\in D_L(g)$}.
\end{cases}
$$
Then $\cA'(\scrR)\cong \cA(\scrR)$, where $\cA(\scrR)$ is the automaton constructed in Theorem~\ref{thm:regularautomaton}.
\end{cor}

\begin{proof}
We define a bijection $f:\scrR\cup\{\dagger\}\to \Gamma(\scrR)\cup\{\dagger\}$ by $f(\dagger)=\dagger$ and $f(P)=g$ if $g$ is the gate of $P$. We need to show that $f(\mu(P,s))=\mu'(f(P),s)$ for all $P\in\scrR$. Let $P\in\scrR$ and let $g$ be the gate of $P$. If $s\in D_L(P)$ then $f(\mu(P,s))=f(\dagger)=\dagger$ and also $\mu'(f(P),s)=\mu'(g,s)=\dagger$ as $D_L(g)=D_L(P)$. If $s\notin D_L(P)$
let $P'\in \scrR$ be the part of $\scrR$ with $sP\subseteq P'$. Let $g'$ be the gate of~$P'$. Then $f(\mu(P,s))=f(P')=g'$ and $\mu'(f(P),s)=\mu'(g,s)=\pi_{\scrR}(sg)$. By definition, $\pi_{\scrR}(sg)$ is the gate of the part containing $sg$, and since $g\in P$ and $sP\subseteq P'$ we have $\pi_{\scrR}(sg)=g'$, and hence the result.
\end{proof}

We note, in passing, the following analogue of \cite[Proposition~2.8]{HNW:16} for general projection maps. 

\begin{prop}
Let $\scrR\in\Preg$. If $w\in W$ and $s\notin D_L(w)$ then $\pi_{\scrR}(sw)=\pi_{\scrR}(s\pi_{\scrR}(w))$. 
\end{prop}

\begin{proof}
Let $P$ be the part of $\scrR$ containing $w$, and let $g$ be the gate of $P$. Thus $\pi_{\scrR}(w)=g$. Since $\scrR$ is regular we have $sP\subseteq P'$ for some part $P'$ of $\scrR$. Let $g'$ be the gate of $P'$. Since $sw\in sP\subseteq P'$ we have $\pi_{\scrR}(sw)=g'$. But also $sg\in P'$, and so $\pi_{\scrR}(s\pi_{\scrR}(w))=\pi_{\scrR}(sg)=g'$. Hence the result.
\end{proof}

\subsection{Simple refinements preserve the gate property}\label{sec:simplerefinementsgates}
In this section we show that if $\scrP$ is locally constant, convex, and gated, and if $\scrP\to\scrP'$ via a simple refinement, then $\scrP'$ is also locally constant, convex, and gated. Theorem~\ref{thm1:preservegatedness} follows, and this is a key component of the  proof of Theorem~\ref{thm:main1}. 

\begin{lem}\label{lem:littlegate}
Let $\epsilon\in\{-,+\}$ and $s\in S$. If $X\subseteq H_{\alpha_s}^{\epsilon}$ is gated with gate $g$, then $sX$ is gated with gate~$sg$.
\end{lem}

\begin{proof}
Let $x\in X$. Since $g$ is the gate of $X$ we have $\ell(g^{-1}x)=\ell(x)-\ell(g)$, and since $g,x\in X\subseteq H_{\alpha_s}^{\epsilon}$ we have $\ell(sx)-\ell(sg)=\ell(x)-\ell(g)$. Thus
$$
\ell((sg)^{-1}(sx))=\ell(g^{-1}x)=\ell(x)-\ell(g)=\ell(sx)-\ell(sg),
$$
and so $sg\peq y$ for all $y=sx\in sX$.
\end{proof}

\begin{lem}\label{lem:littlejoin}
Let $x,y\in W$ and $s\in S$ with $\{x,y\}$ bounded and $s\in D_L(x)\cap D_L(y)$. Then $s(x\vee y)=(sx)\vee (sy)$. 
\end{lem}

\begin{proof}
Let $z=x\vee y$. Since $\ell(sx)=\ell(x)-1$ and $x\peq z$ we have $\ell(sz)=\ell(z)-1$. Then
$
\ell((sx)^{-1}(sz))=\ell(x^{-1}z)=\ell(z)-\ell(x)=\ell(sz)-\ell(sx),
$
and so $sx\peq sz$. Similarly $sy\peq sz$, and so $\{sx,sy\}$ is bounded, and $(sx)\vee(sy)\peq sz$. 

Now let $w$ be any bound for $\{sx,sy\}$ with $w\peq sz$. Since $s\notin D_L(sz)$ (as $\ell(sz)=\ell(z)-1$) and $w\peq sz$ we have $s\notin D_L(w)$. Thus
\begin{align*}
\ell(x^{-1}(sw))&=\ell((sx)^{-1}w)\\
&=\ell(w)-\ell(sx)&&\text{as $sx\peq w$ by assumption}\\
&=\ell(w)-\ell(x)+1&&\text{as $s\in D_L(x)$}\\
&=\ell(sw)-\ell(x)&&\text{as $s\notin D_L(w)$}.
\end{align*}
Thus $x\peq sw$, and similarly $y\peq sw$. So $sw$ is a bound for $\{x,y\}$. But also $sw\peq z$, because
\begin{align*}
\ell((sw)^{-1}z)&=\ell(w^{-1}(sz))\\
&=\ell(sz)-\ell(w)&&\text{as $w\peq sz$ by assumption}\\
&=\ell(z)-1-\ell(w)&&\text{as $s\in D_L(z)$}\\
&=\ell(z)-\ell(sw)&&\text{as $s\notin D_L(w)$}.
\end{align*}
Thus $sw=z$ (as $z=x\vee y$ is the least upper bound of $\{x,y\}$) and so $w=sz$. In particular we have $(sx)\vee(sy)=sz$.
\end{proof}

\begin{thm}\label{thm:simplerefinementsgates}
Let $\scrP$ be a locally constant, convex, gated partition of $W$, and for each $P\in\scrP$ let $g_P$ denote the gate of~$P$. Suppose that $\scrP\mapsto\scrP'$ by a simple refinement based at $(P,s)$. Let $X=\{P'\in\scrP\mid sP\cap P'\neq\emptyset\}$, and let $P_0'$ be the element of $X$ with $sg_P\in P_0'$. Then
\begin{enumerate}
\item the partition $\scrP'$ is locally constant, convex, and gated;
\item the element $g_P$ is the gate of $P\cap sP_0'$, and for each $P'\in X\backslash \{P_0'\}$ the set $\{g_P,sg_{P'}\}$ is bounded and $h_{P'}=g_P\vee sg_{P'}$ is the gate of $P\cap sP'$;
\item if $\scrP$ has the property that each part $P\in\scrP$ is an intersection of finitely many half-spaces, then the partition $\scrP'$ also has this property.
\end{enumerate}
\end{thm}

\begin{proof} Note that $s\notin D_L(P)$ (by the definition of simple refinements) and that $\scrP'$ is locally constant and convex (as all refinements of a locally constant partition are locally constant, and the intersection of convex sets is convex). It is also clear that (3) holds, for if $P\in\scrP$ and $P'\in X$ are expressed as an intersection of finitely many half-spaces, then $P\cap sP'$ can also be expressed as an intersection of finitely many half-spaces. Thus it remains to prove that each part $P\cap sP'$ of $\scrP'$ is gated. Since $g_P\in P\cap sP_0'$ and $g_P\peq w$ for all $w\in P$ we have that $P\cap sP_0'$ is gated with gate $g_P$. 

Let $P'\in X\backslash\{P_0'\}$. Let $v\in sP\cap P'$ (note that $sP\cap P'$ is nonempty by hypothesis). Since $v\in sP$ we have $v=s(g_Px)$ for some $x\in W$ with $\ell(g_Px)=\ell(g_P)+\ell(x)$ (by the gate property of $P$), and since $s\notin D_L(P)$ we have $v=(sg_P)x$ with $\ell(sg_Px)=\ell(sg_P)+\ell(x)$. Since $v\in P'$ we have $v=g_{P'}y$ for some $y\in W$ with $\ell(g_{P'}y)=\ell(g_{P'})+\ell(y)$ (by the gate property of $P'$). Thus $v=(sg_P)x=g_{P'}y$ is an upper bound for $\{sg_P,g_{P'}\}$. Let $\overline{v}=sg_P\vee g_{P'}$ be the least upper bound of $\{sg_P,g_{P'}\}$. In particular, $\overline{v}$ is a prefix of each $v\in sP\cap P'$. 

We claim that $\overline{v}\in sP\cap P'$, and hence $\overline{v}$ is a gate of $sP\cap P'$. To see that $\overline{v}\in P'$, note that for all $v\in sP\cap P'$ we have $g_{P'}\peq \overline{v}\peq v$, because $\overline{v}=sg_P\vee g_{P'}$ and we showed above that $\overline{v}\peq v$ for all $v\in sP\cap P'$. Since $g_{P'},v\in P'$ it follows that $\overline{v}\in P'$ as $P'$ is convex. Similarly, to see that $\overline{v}\in sP$ observe that $sg_P\peq \overline{v}\peq v$, and note that $sP$ is convex as $P$ is convex. 

Thus we have shown that $\overline{v}$ is a gate of $sP\cap P'$, and so $P\cap sP'$ is gated with gate $h_{P'}=s\overline{v}=s(sg_P\vee g_{P'})$ (by Lemma~\ref{lem:littlegate}). Since $s\in D_L(sg_P)\cap D_L(g_{P'})$ Lemma~\ref{lem:littlejoin} gives $h_{P'}=g_P\vee sg_{P'}$, completing the proof.

\end{proof}

\begin{exa}\label{ex:gates}
Theorem~\ref{thm:simplerefinementsgates} is illustrated in Figure~\ref{fig:simplerefinement}. We have $X=\{P_0',P_1',P_2',P_3'\}$. The gates $g_0,g_1,g_2,g_3$ of the parts $P_0',P_1',P_2',P_3'$ are shown as black dots. The gate $g_P$ is shown as a red circle, and the ``new'' gates $h_j=g_P\vee sg_j$ are shown as red dots.
\end{exa}

We have the following important corollary, proving Theorem~\ref{thm1:preservegatedness}. 

\begin{cor}\label{cor:regularcompletiongated}
Let $\scrP$ be a locally constant, convex and gated partition of $W$. If Algorithm~\ref{alg:regularisation} terminates in finite time then the regular completion $\widehat{\scrP}$ is gated and convex. 
\end{cor}

\begin{proof}
By assumption $\widehat{\scrP}$ can be obtained from $\scrP$ by a finite sequence of simple refinements, and hence the result by Theorem~\ref{thm:simplerefinementsgates}.
\end{proof}

\subsection{The gates of $W$ and minimal length cone type representatives} \label{sec:gatesW}

We are finally able to prove Theorem~\ref{thm:main1}. 

\begin{cor}\label{cor:gateexist}
The cone type partition $\scrT$ is regular, convex, and gated. In particular, each part $X_T$ of the cone type partition has a unique minimal length element $g_T$, and if $x\in X_T$ then $g_T\peq x$. 
\end{cor}

\begin{proof}
By Theorem~\ref{thm:finitetermination} and Corollary~\ref{cor:regularlattice1} we have that Algorithm~\ref{alg:regularisation} applied to the $S$-partition $\scrD$ terminates in finite time with $\scrT=\widehat{\scrD}$. By Lemma~\ref{lem:Spartitiongated} the partition $\scrD$ is convex and gated, and so by Corollary~\ref{cor:regularcompletiongated} the partition $\scrT$ is also convex and gated. Hence the result.
\end{proof}

\begin{cor}\label{cor:minexist}
Each cone type $T$ has a unique minimal length cone type representative. That is, for each cone type $T$ the set $\{w\in W\mid T(w)=T\}$ has a unique minimal length element~$m_T$. Moreover, if $w\in W$ with $T(w)=T(m_T)$ then $m_T$ is a suffix of~$w$. We have $m_T=g_T^{-1}$, where $g_T$ is the gate of $X_T$. 
\end{cor}

\begin{proof}
It is obvious from Corollary~\ref{cor:gateexist} that if $g_T$ is the gate of the part $X_T$ of $\scrT$ then $m_T=g_T^{-1}$ is the unique minimal length element with $T(m_T)=T$. 
\end{proof}

\begin{cor}\label{cor:finiteintersection}
For each cone type $T$ the set $X_T$ can be expressed as an intersection of finitely many half-spaces. 
\end{cor}

\begin{proof}
This follows from part (3) of Theorem~\ref{thm:simplerefinementsgates}.
\end{proof}

\begin{defn}
The \textit{gates of $W$} are the gates of the cone type partition. Let $\Gamma=\Gamma(\scrT)$ denote the set of gates of $W$. Then $\Gamma^{-1}$ is the set of minimal length cone type representatives. 
\end{defn}

We record the following observations.
\begin{prop} \label{prop:gatesbasicfacts}
We have
\begin{enumerate}
\item the set $\Gamma$ is closed under suffix;
\item $W_J\subseteq \Gamma$ for each spherical subset $J\subseteq S$;
\item if $\scrP$ is regular and gated then $\Gamma\subseteq \Gamma(\scrP)$;
\item if $B$ is a Garside shadow, then $\Gamma\subseteq B$;
\item $\Gamma\subseteq L$;
\item $|\Gamma|\leq|\mathbb{E}|$ with equality if and only if $\cE=\Phisph^+$. 
\end{enumerate}
\end{prop}

\begin{proof}
(1) and (2) are special cases of Theorem~\ref{thm:suffixclosuregates}. (3) follows from the fact that $\scrT\leq \scrP$ for all regular partitions~$\scrP$ (by Corollary~\ref{cor:Tisminimal}), and (4) is a special case of (3), using the Theorem~\ref{thm:regularpartitions} and Proposition~\ref{prop:garsidegated}. Moreover, since $L$ is a Garside shadow we have $\Gamma\subseteq L$ by (4). 

Finally, since $|\Gamma|$ is the number of states of the minimal automata recognising $\cL(W,S)$, and since $\cA_0$ has $|\mathbb{E}|$ states, we have $|\Gamma|\leq|\mathbb{E}|$. Equality holds if and only if the automaton $\cA_0$ is minimal, and by \cite[Theorem~1]{PY:19} this occurs if and only if $\cE=\Phisph^+$. 
\end{proof}

If $W$ is affine the partition $\scrS_0$ is gated by the classical work of Shi~\cite{Shi:87a,Shi:87b}, however in general it is unknown if the $n$-Shi partitions $\scrS_n$ are gated. We make the following conjecture.

\begin{conj}\label{conj:Shigated}
Let $n\in\mathbb{N}$. The $n$-Shi partition $\scrS_n$ is gated. 
\end{conj}

In the case that $n=0$ and $(W,S)$ is affine, Conjecture~\ref{conj:Shigated} is true by Shi's work~\cite{Shi:87a,Shi:87b}. In \cite[Conjecture~2]{DH:16} Dyer and Hohlweg conjecture that the map $\Theta_n:L_n\to \mathbb{E}_n$ with $\Theta_n(x)=\cE_n(x)$ is bijective, and we note that this conjecture, if true, readily implies Conjecture~\ref{conj:Shigated} (with the gates being the $n$-low elements). Recently Chapelier-Laget and Hohlweg~\cite{CH:21} have proved \cite[Conjecture~2]{DH:16} in the case $n=0$ for affine Coxeter groups. Finally, Corollary~\ref{cor:gateexist} implies the following further evidence for Conjecture~\ref{conj:Shigated}.

%

\begin{prop}\label{prop:evidencegated}
If $\cE=\Phisph^+$ then $\scrS_0$ is gated. 
\end{prop}

\begin{proof}
By Corollary~\ref{cor:someobservations} if $\cE=\Phisph^+$ then $\scrS_0=\scrT$, which is gated by Corollary~\ref{cor:gateexist}. (Another approach is to use Theorem~\ref{thm:conjsspherical}).
\end{proof}


In \cite[Conjecture~1]{HNW:16} Hohlweg, Nadeau and Williams conjecture that the automaton $\scrG_{\Gmin}$ is the minimal automaton recognising $\cL(W,S)$. We note that, in terms of the set $\Gamma$, minimality of $\scrG_{\Gmin}$ is equivalent to the following.

\begin{thm}\label{thm:GminGamma}
The automaton $\scrG_{\Gmin}$ is minimal if and only if $\Gamma$ is closed under join.
\end{thm}

\begin{proof}
By Proposition~\ref{prop:gatesbasicfacts} we have $\Gamma\subseteq\Gmin$. If $\scrG_{\Gmin}$ is minimal then $\cA_{\Gmin}\cong \cA(W,S)$ (by Theorem~\ref{thm:MyhillNerode}) and hence $|\Gamma|=|\Gmin|$, giving $\Gamma=\Gmin$. Thus $\Gamma$ is closed under join.

Conversely, if $\Gamma$ is closed under join, then by Proposition~\ref{prop:gatesbasicfacts} parts (1) and (2) we have that $\Gamma$ is a Garside shadow. Since $\Gamma\subseteq \Gmin$ we have $\Gamma=\Gmin$. 
\end{proof}

We have been unable to prove in general that $\Gamma$ is closed under join, however the following theorem establishes this fact in the case $\cE=\Phisph^+$, providing evidence for Conjecture~\ref{conj:garside}.
 
\begin{thm}\label{thm:conjectures}
Suppose that $\cE=\Phisph^+$. Then $\Gamma=\Gmin=L$. In particular, $\Gamma$ is closed under join.
\end{thm}

\begin{proof}
By Proposition~\ref{prop:gatesbasicfacts}(4) we have $\Gamma\subseteq \Gmin$, and since $L$ is a Garside shadow we have $\Gmin\subseteq L$. By \cite[Proposition~3.26]{DH:16} we have $|L|\leq|\mathbb{E}|$, and hence $|\Gamma|\leq|\Gmin|\leq |L|\leq|\mathbb{E}|$. Thus if $\cE=\Phisph^+$ then Proposition~\ref{prop:gatesbasicfacts}(6) forces $|\Gamma|=|\Gmin|=|L|=|\mathbb{E}|$. Since $\Gamma\subseteq\Gmin\subseteq L$ this gives $\Gamma=\Gmin=L$.
%
\end{proof}
%
%


\begin{rem}
Figures~\ref{fig:B2partitions}, \ref{fig:G2partitions}, and~\ref{fig:A2partitions} show the partitions $\scrS_0$ and $\scrT$ for $\tilde{\sB}_2$, $\tilde{\sG}_2$, and $\tilde{\sA}_2$, respectively. In these case it turns out (and conjecturally this is always true) that $\scrS_0$ is gated, with $\Gamma(\scrS_0)=L$. In the $\tilde{\sB}_2$ and $\tilde{\sG}_2$ cases the inclusion $\Gamma\subseteq L$ is strict. The elements of $\Gamma$ are shaded blue, and the elements of $L\backslash\Gamma$ are shaded red. In the $\tilde{\sA}_2$ case we have $\Gamma=L$. 

Moreover, since $\Gamma$ (and also $L$) are closed under suffix, the sets $\Gamma^{-1}$ and $L^{-1}$ are closed under prefix. Thus these sets are connected regions of the Coxeter complex. We draw these sets in Figure~\ref{fig:Shitriangles} for $\tilde{\sA}_2$, $\tilde{\sB}_2$, and $\tilde{\sG}_2$, with $\Gamma^{-1}$ shaded blue, and $L^{-1}\backslash\Gamma^{-1}$ shaded red. The fact that $L^{-1}$ is a dilation of the fundamental alcove in these cases is explained by a celebrated result of Shi (see~\cite[\S8]{Shi:87b}) and the observation that $L$ is the set of gates of $\scrS_0$ in these cases. 

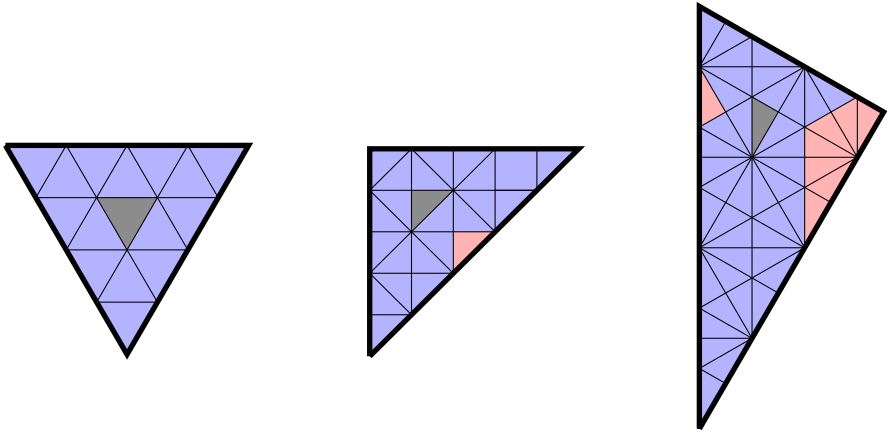
\begin{figure}[H]
\centering
\subfigure{
\begin{tikzpicture}[scale=0.8]
 \path [fill=blue!30] (-2,{2*0.866})--(0,{-2*0.866})--(2,{2*0.866})--(-2,{2*0.866});
    \path [fill=gray!90] (0,0) -- (-0.5,0.866) -- (0.5,0.866) -- (0,0);
    \draw [line width=2pt] (-2,{2*0.866})--(0,{-2*0.866})--(2,{2*0.866})--(-2,{2*0.866});
    \draw (-1,{2*0.866})--(-1.5,0.866);
    \draw (0,{2*0.866})--(-1,0);
    \draw (1,{2*0.866})--(-0.5,-0.866);
     \draw (1,{2*0.866})--(1.5,0.866);
    \draw (0,{2*0.866})--(1,0);
    \draw (-1,{2*0.866})--(0.5,-0.866);
    \draw (-0.5,-0.866)--(0.5,-0.866);
    \draw (-1,0)--(1,0);
    \draw (-1.5,0.866)--(1.5,0.866);
    \phantom{\draw (0,-3)--(0,2);}
\end{tikzpicture}
}\quad\quad\quad\,\,
\subfigure{
\begin{tikzpicture}[scale=0.55]
\path [fill=blue!30] (-1,-3)--(-1,2)--(4,2)--(-1,-3);
\path [fill=red!30] (1,-1)--(2,0)--(1,0)--(1,-1);
\path [fill=gray!90] (0,0) -- (1,1) -- (0,1) -- (0,0);
\draw [line width=2pt] (-1,-3)--(-1,2)--(4,2)--(-1,-3);
\draw (-1,1)--(3,1);
\draw (-1,0)--(2,0);
\draw (-1,-1)--(1,-1);
\draw (-1,-2)--(0,-2);
\draw (0,-2)--(0,2);
\draw (1,-1)--(1,2);
\draw (2,0)--(2,2);
\draw (3,1)--(3,2);
\draw (-1,1)--(0,2);
\draw (-1,-1)--(2,2);
\draw (-1,-1)--(0,-2);
\draw (-1,1)--(1,-1);
\draw (0,2)--(2,0);
\draw (2,1)--(3,1);
\phantom{\draw (0,-4.8)--(0,2);}
\end{tikzpicture}
}\quad
\subfigure{
\begin{tikzpicture}[scale=0.8]
%
\path [fill=red!30] (0,2)--(0,3)--(0.433,2.25);
\path [fill=red!30] (1.732,0)--(1.732,2)--(2.598,2.5)--({2.598+0.433},2.25)--(1.732,0);
\path [fill=blue!30] (0,-3)--(0,2)--(0.433,2.25)--(0,3)--(0,4)--(2.598,2.5)--(1.732,2)--(1.732,0)--(0,-3);
\path [fill=gray!90]  (0.866,1.5) -- (1.299,2.25) -- (0.866,2.5)--(0.866,1.5);
\draw[line width=2pt](0,-3)--(0,4)--(3.031,2.25)--(0,-3);
\draw(0,3)--(1.732,3);
\draw(0,1.5)--(2.598,1.5);
\draw(0,0)--(1.732,0);
\draw(0,-1.5)--(0.866,-1.5);
\draw(0,-3)--(0,4);
\draw(.866,-1.5)--(.866,3.5);
\draw(1.732,0)--(1.732,3);
\draw(2.598,1.5)--(2.598,2.5);
\draw(0,3)--(0.866,3.5);
\draw(0,2)--(1.732,3);
\draw(0,1)--(2.598,2.5);
\draw(0,0)--(2.598,1.5);
\draw(0,-1)--(1.732,0);
\draw(0,-2)--(0.866,-1.5);
\draw({2.598+0.433},2.25)--(0,4);
\draw(2.598,1.5)--(0,3);
\draw({1.732+0.433},0.75)--(0,2);
\draw(1.732,0)--(0,1);
\draw({0.866+0.433},-0.75)--(0,0);
\draw(0.866,-1.5)--(0,-1);
\draw(0.433,-2.25)--(0,-2);
\draw(0,0)--(0.866,-1.5);
\draw(0,3)--(1.732,0);
\draw(1.732,3)--(2.598,1.5);
\draw({2.598+0.433},2.25)--(0,-3);
\draw(1.732,3)--(0,0);
\draw(0.433,3.75)--(0,3);
\phantom{\draw(-1.2,0)--(1,0);}
\end{tikzpicture}
}
\caption{The sets $\Gamma^{-1}$ and $L^{-1}$}\label{fig:Shitriangles}
\end{figure}
\end{rem}

\subsection{A characterisation of the gates of $W$}

The following theorem gives a characterisation of the gates in terms of roots, and is linked to the concept of boundary roots (see Theorem~\ref{thm:boundaryroots}). Recall the definition of $\Phi^0(w)$ from Definition~\ref{defn:finalroots}.

\begin{thm}\label{thm:characterisegates}
Let $x\in W$. Then $x\in\Gamma$ if and only if for each $\beta\in\Phi^0(x)$ there exists $w\in W$ with $\Phi(x)\cap \Phi(w)=\{\beta\}$.
\end{thm}

\begin{proof}
Let $x\in\Gamma$, and let $\beta\in\Phi^0(x)$. Hence $\beta=-x\alpha_s$ for some $s\in D_R(x)$. Since $\ell(xs)<\ell(x)$ we have $T(sx^{-1})\neq T(x^{-1})$ (as $x\in\Gamma$ is of minimal length in its part of $\scrT$). Since $T(x^{-1})\subseteq T(sx^{-1})$ (by Lemma~\ref{lem:containoneway}) we have strict containment, and so there exists $w\in T(sx^{-1})\backslash T(x^{-1})$. Thus by Proposition~\ref{prop:conetypebasics} we have $\Phi(xs)\cap \Phi(w)=\emptyset$ and $\Phi(x)\cap \Phi(w)\neq\emptyset$. Since $\Phi(x)=\Phi(xs)\sqcup \{\beta\}$ we have $\Phi(x)\cap\Phi(w)=\{\beta\}$.

Conversely, suppose that $x\notin\Gamma$. Let $T=T(x^{-1})$, and let $g\in\Gamma$ be the gate of $X_T$. Then $g\peq x$, and by convexity of $X_T$ (see Proposition~\ref{prop:convexitybasics}) each $y\in W$ with $g\peq y\peq x$ has $T(y^{-1})=T$. In particular, since $g\neq x$ there exists $s\in D_R(x)$ with $T(sx^{-1})=T=T(x^{-1})$. Then $\beta=-x\alpha_s\in\Phi^0(x)$, and by Proposition~\ref{prop:conetypebasics} we have that for all $w\in W$ we have $\Phi(x)\cap \Phi(w)=\emptyset$ if and only if $\Phi(xs)\cap \Phi(w)=\emptyset$. Since $\Phi(x)=\Phi(xs)\sqcup\{\beta\}$ it follows that there is no element $w$ with $\Phi(x)\cap \Phi(w)=\{\beta\}$. 
\end{proof}

A priori, given $x\in W$ and $\beta\in\Phi^0(x)$, deciding if there exists $w\in W$ such that $\Phi(x)\cap\Phi(w)=\{\beta\}$ appears difficult to implement in an infinite Coxeter group. However we note that, by the following proposition, one only needs to check $w\in \Gamma$ (a finite set). 

\begin{prop}\label{prop:witnessgate}
Let $x\in W$ and $\beta\in \Phi^+$. Suppose there exists $w\in W$ such that $\Phi(x)\cap\Phi(w)=\{\beta\}$, and let $w$ be of minimal length subject to this property. Then $\Phi^0(w)=\{\beta\}$, and $w$ is a gate.
\end{prop}

\begin{proof}
Let $\alpha\in\Phi^0(w)$. Thus $s_{\alpha}w=ws$ for some $s\in S$ and $\ell(ws)=\ell(w)-1$. If $\alpha\neq\beta$ then $\Phi(x)\cap\Phi(ws)=\{\beta\}$, contradicting the minimal length assumption. Thus $\alpha=\beta$, and so $\Phi^0(w)=\{\beta\}$, and then $w$ is a gate by Theorem~\ref{thm:characterisegates}. 
\end{proof}

Similar ideas give the following corollary.

\begin{cor}
If $T=T(x^{-1})$ then 
$
\partial T=\{\beta\in\Phi^+\mid \Phi(x)\cap\Phi(g)=\{\beta\}\text{ for some $g\in\Gamma$}\}.
$
\end{cor}

\begin{proof}
By Theorem~\ref{thm:boundaryroots} we have
$
\partial T=\{\beta\in\Phi^+\mid \Phi(x)\cap\Phi(w)=\{\beta\}\text{ for some $w\in W$}\},
$
and the result follows from Proposition~\ref{prop:witnessgate}. 
\end{proof}

\begin{rem}\label{rem:calculation}
Theorem~\ref{thm:characterisegates} and Proposition~\ref{prop:witnessgate} allow one to implement the calculation of the set $\Gamma$ into MAGMA~\cite{MAGMA}, utilising the existing Coxeter group package. The main steps are as follows.
\begin{enumerate}
\item The set $\cE$ of elementary roots is computed inductively by setting $E_0=\{\alpha_s\mid s\in S\}$, and defining $E_{j+1}=E_j\cup\{s\alpha\mid \alpha\in E_j,\,s\in S,\,\text{ with } -1<(\alpha,\alpha_s)<0\}$. Once $E_{n+1}=E_n$ we have $\cE=E_n$ (see \cite[\S4.7]{BB:05}). 
\item The set $L$ of low elements is computed by setting $M_0=\{e\}$, and defining $M_{j+1}=M_j\cup\{sw\mid w\in M_j,\,s\in S\backslash D_L(w),\,\Phi^1(sw)\subseteq \cE\}$. Once $M_n=M_{n+1}$ we have $L=M_n$ (by the suffix closure property of $L$). 
\item Since $\Gamma\subseteq L$ we can then determine, in finite time, the set $\Gamma$ by checking, for each $x\in L$ and $\beta\in\Phi^0(x)$, whether there exists $w\in L$ such $\Phi(x)\cap\Phi(w)=\{\beta\}$ (using Proposition~\ref{prop:witnessgate} and the fact that $\Gamma\subseteq L$). 
\end{enumerate}
We have carried through the calculations for a variety of Coxeter groups, and the data is presented in Figure~\ref{fig:data} for some affine and compact hyperbolic groups. See \cite{PY:19} for the definitions of the compact hyperbolic groups $\sX_4(c)$ ($c\in\{4,5\}$), $\sX_5(d)$ ($d\in\{3,4,5\}$), $\sY_4$, $\sZ_4$, and $\sZ_5$.

We note that, as a general rule, groups with large spherical subgroups will have have many gates and low elements (by Proposition~\ref{prop:gatesbasicfacts}), while those with only small spherical subgroups tend to have very few gates. For example, in $\tilde{\sD}_5$ there are $59049$ low elements and $58965$ gates, while in the corresponding Coxeter group with each occurrence of $m_{st}=3$ replaced by $m_{st}=4$ there are only $332$ low elements and $247$ gates. 

\begin{figure}[H]
\centering
\subfigure{
$
\begin{array}{|c||c|c|c|c|c|}
\hline
W& |\cE|  & |L| & |\Gamma| \\
\hline\hline
\tilde{\sA}_2& 6  &16&16\\
 \hline
 \tilde{\sB}_2& 8  &25 &24\\
 \hline
 \tilde{\sG}_2&12    &49 &41\\
 \hline
   \tilde{\sA}_3&12    &125 &125\\
 \hline
  \tilde{\sB}_3&18    &343 &315\\
 \hline
   \tilde{\sC}_3&18    &343 &317\\
    \hline
   \tilde{\sA}_4&20    &1296 &1296\\
    \hline
   \tilde{\sB}_4&32    &6561 &5789\\
    \hline
   \tilde{\sC}_4&32    &6561 &5860\\
    \hline
   \tilde{\sD}_4&24    &2401 &2400\\
    \hline
\tilde{\sF}_4&48   &28561 &22428\\
 \hline
 \tilde{\sA}_5&30&16807&16807\\
 \hline
 \tilde{\sB}_5&50&161051&137147\\
 \hline
 \tilde{\sC}_5&50&161051&139457\\
 \hline
 \tilde{\sD}_5&40&59049&58965\\
 \hline
\end{array}
$
}\qquad\qquad
\subfigure{
$
\begin{array}{|c||c|c|c|c|c|}
\hline
W& |\cE| & |L| & |\Gamma| \\
\hline \hline
   \sX_4(4)&25    &438 &392\\
    \hline
   \sX_4(5)&32    &516 &462\\
 \hline
   \sY_4&32    &687 &578\\
    \hline
   \sZ_4&30    &513 &473\\
     \hline
   \sX_5(3)&114    &101412 &52542\\
    \hline
   \sX_5(4)&83   &25708 &22886\\
    \hline
   \sX_5(5)&135   &42064 &37956\\
    \hline
   \sZ_5&120   &41385 &39138\\
 \hline
 \end{array}
  $}
  \caption{Data for low rank affine and compact hyperbolic Coxeter groups}\label{fig:data}
\end{figure}
\end{rem}

We conclude this section with a conjecture. Define a partial order on the set $\mathbb{T}$ of cone types by $T_1\leq T_2$ if and only if $T_2\subseteq T_1$ (thus $\leq$ is given by reverse containment).

\begin{conj}\label{conj:orderisomorphism}
The map $\Theta:(\Gamma,\peq)\to(\mathbb{T},\leq)$ given by $\Theta(g)=T(g^{-1})$ is an order isomorphism.
\end{conj}

Note that the conjecture, if true, generalises Theorem~\ref{thm:parabolicconetype} because $W_J\subseteq \Gamma$ for all spherical $J\subseteq S$, by Proposition~\ref{prop:gatesbasicfacts}. We note the following consequence of Conjecture~\ref{conj:orderisomorphism}.

\newpage

\begin{prop}
If Conjecture~\ref{conj:orderisomorphism} holds then $\Gamma$ is closed under join (hence Conjecture~\ref{conj:garside} holds)
\end{prop}

\begin{proof}
Let $x,y\in\Gamma$, and suppose that $\{x,y\}$ is bounded. Let $z=x\vee y$, and let $g$ be the gate of the part of $\scrT$ containing~$z$. Since $x\peq z$ and $y\peq z$ we have $T(z^{-1})\subseteq T(x^{-1})\cap T(y^{-1})$ (by Lemma~\ref{lem:containoneway}; in fact equality holds by Proposition~\ref{prop:joinconetype}). Since $T(g^{-1})=T(z^{-1})\subseteq T(x^{-1})\cap T(x^{-1})$ we have $x\peq g$ and $y\peq g$ (assuming Conjecture~\ref{conj:orderisomorphism}), and so $g$ is an upper bound for $\{x,y\}$. Thus $z\peq g$. But also $g\peq z$ by gate properties, and so $z=g$. 
\end{proof}
%
%
\section{Conical partitions}\label{sec:conical}

In this section we define \textit{conical partitions}, generalising the construction of Garside partitions. In fact, in Corollary~\ref{cor:garsideequivalent} we show that regular conical partitions are the partition theoretic equivalent to Garside shadows. 

\begin{defn}\label{defn:conical}
Let $X\subseteq W$ with $e\in X$. The \textit{conical partition} induced by $X$ is the partition $\scrC(X)$ of $W$ induced by the covering $\{C(x)\mid x\in X\}$. (Note that $e\in X$ is required for this to be a covering). 
\end{defn}

In particular, every Garside partition is conical by definition (see Definition~\ref{defn:partitions}). In fact, as we see in Theorem~\ref{thm:conesgated}, every conical partition is gated, and the gates of such a partition are necessarily closed under join, generalising Proposition~\ref{prop:garsidegated}.

\begin{defn} The \textit{join-closure} of a subset $X\subseteq W$ is
$$
X^{\vee}=\{\bigvee Y\mid Y\subseteq X\text{ is bounded}\}.
$$
\end{defn}

The following lemma shows that the join-closure of $X$ is indeed closed under joins. 

\begin{lem}\label{lem:joinclosure}
Let $X\subseteq W$. If $Y\subseteq X^{\vee}$ is bounded, then $\bigvee Y\in X^{\vee}$. 
\end{lem}

\begin{proof}
Let $Y\subseteq X^{\vee}$ be bounded. If $z\in Y\backslash X$ then $z=\bigvee Z$ for some bounded subset $Z\subseteq X$ (by the definition of join-closure). It is clear that
$$
\bigvee Y=\bigvee ((Y\backslash \{z\})\cup Z),
$$
and it follows that $\bigvee Y$ can we written as $\bigvee Y'$ for a bounded subset $Y'\subseteq X$. Hence $\bigvee Y\in X^{\vee}$.
\end{proof}

\begin{thm}\label{thm:conesgated}
Let $X\subseteq W$ with $e\in X$. Let $\scrP=\scrC(X)$ be the conical partition induced by~$X$. Then $\scrP$ is gated, with $\Gamma(\scrP)=X^{\vee}$. In particular, the set $\Gamma(\scrP)$ is closed under join. 
\end{thm}

\begin{proof}
Let $P$ be a part of $\scrP$, and let $w\in P$. Since $\{C(x)\mid x\in X\}$ is a covering of $W$, the set
$$
\{x\in X\mid w\in C(x)\}=\{x\in X\mid x\peq w\}
$$
is nonempty, and bounded above by $w$. Thus 
$
g=\bigvee\{x\in X\mid w\in C(x)\}
$
exists. By Lemma~\ref{lem:conejoin} we have 
$$
\bigcap_{\{x\in X\mid w\in C(x)\}}C(x)=C(g),
$$
and so $w\in C(g)$. Thus $g\peq w$ for all $w\in P$. We claim that $g$ and $w$ lie in the same part of $\scrP$ (from which it follows that $P$ is gated with gate $g$). Thus we must show that for all $x\in X$ we have $w\in C(x)$ if and only if $g\in C(x)$. If $x\in X$ with $w\in C(x)$, then by the definition of $g$ we have $x\peq g$ and so $g\in C(x)$. Conversely, if $g\in C(x)$ then $x\peq g$, and so since $g\peq w$ we have $x\peq w$, and so $w\in C(x)$. Hence the claim.

Now, if $Y\subseteq X$ is any bounded set, and $g=\bigvee Y$, then clearly 
$$
g=\bigvee\{x\in X\mid g\in C(x)\},
$$
and by the above discussion $g$ is the gate of the part of $\scrP$ containing $g$. Hence $\Gamma(\scrP)=X^{\vee}$. 
\end{proof}

\newpage

The following lemma shows that one may replace $X$ by its join closure when constructing~$\scrC(X)$.

\begin{lem}\label{lem:conicalclosed}
If $X\subseteq W$ then $\scrC(X)=\scrC(X^{\vee})$. 
\end{lem}

\begin{proof}
Let $u,v$ be in the same part of $\scrC(X)$. Let $x\in C(X^{\vee})$. Then $x=\bigvee Y$ for some $Y\subseteq X$. Then $C(x)=\bigcap_{y\in Y}C(y)$ (by Lemma~\ref{lem:conejoin}). So if $u\in C(x)$ we have $u\in C(y)$ for all $y\in Y$ (as $u,v$ are in the same part of $\scrC(X)$) and so $v\in C(x)$. Thus $u,v$ are in the same part of $\scrC(X^{\vee})$. Conversely, it is clear that if $u,v$ are in the same part of $\scrC(X^{\vee})$ then they are also in the same part of $\scrC(X)$.
\end{proof}

\begin{rem}\label{rem:conicalconvex} We note the following.
\begin{enumerate}
\item Not all conical partitions are Garside partitions. For example, consider $W$ of spherical type $\mathsf{A}_2$ and let $X=\{e,sts\}$. The conical partition induced by $X$ is $\scrP=\{\{e,s,t,st,ts\},\{sts\}\}$, which is obviously not a Garside partition (in fact, the only Garside partition of a finite Coxeter group is the partition $\mathbf{1}$ into singletons). We shall show below (in Corollary~\ref{cor:garsideequivalent}) that Garside partitions are equivalent to regular conical partitions. 
\item Conical partitions are not necessarily convex, as the above $\mathsf{A}_2$ example shows (compare with Remark~\ref{rem:garsideconvexity}). However conical partitions $\scrP=\scrC(X)$ are necessarily weakly convex. To see this, note that by Lemma~\ref{lem:conicalclosed} we may assume that $X=X^{\vee}$ and hence $\Gamma(\scrP)=X$ by Theorem~\ref{thm:conesgated}. Suppose that $w\in W$ and that $g\in X$ is the gate of the part of $\scrP$ containing $w$. We need to show that if $g\peq v\peq w$ then $v$ and $w$ lie in the same part of $\scrP$. Let $g'\in\Gamma(\scrP)$. If $v\in C(g')$ then $g'\peq v\peq w$ and so $w\in C(g')$. On the other hand, if $w\in C(g')$ then since $g$ and $w$ lie in the same part we have $g\in C(g')$ (by the definition of conical partitions) and hence $g'\peq g\peq v$, giving $v\in C(g')$. Thus $v$ and $w$ lie in the same part of $\scrP$.
\end{enumerate}
\end{rem}

\begin{cor}\label{cor:garsideequivalent}
Let $X\subseteq W$ with $e\in X$ and let $\scrP=\scrC(X)$ be the (necessarily gated, c.f. Theorem~\ref{thm:conesgated}) conical partition induced by $X$. Then $\Gamma(\scrP)$ is a Garside shadow if and only if $\scrP$ is regular. 
\end{cor}

\begin{proof}
By Theorem~\ref{thm:conesgated} $\scrP$ is gated, with $\Gamma(\scrP)=X^{\vee}$. Thus $\Gamma(\scrP)$ is closed under joins. If $\scrP$ is regular then by Theorem~\ref{thm:suffixclosuregates} we have $S\subseteq \Gamma(\scrP)$ and that $\Gamma(\scrP)$ is closed under taking suffixes, and hence $\Gamma(\scrP)$ is a Garside shadow. 

Conversely, suppose that $\Gamma(\scrP)$ is a Garside shadow. By Lemma~\ref{lem:conicalclosed} we have $\scrP=\scrC(X^{\vee})$, and thus $\scrP$ is the Garside partition of the Garside shadow $X^{\vee}=\Gamma(\scrP)$. Hence $\scrP$ is regular by Theorem~\ref{thm:regularpartitions}.
\end{proof}

We showed in Corollary~\ref{cor:JIsGated} that the spherical partition~$\scrJ$ is gated. We now give another proof, using Theorem~\ref{thm:conesgated}, that has the advantage of determining the set of gates. Let $W_{\mathrm{sph}}$ denote the union of all spherical parabolic subgroups of~$W$.

\begin{prop}\label{prop:Jgated}
We have $\scrJ=\scrC(W_{\mathrm{sph}})$. Thus $\scrJ$ is gated, with $\Gamma(\scrJ)=W_{\mathrm{sph}}^{\vee}$
\end{prop}

\begin{proof}
Once we show that $\scrJ=\scrC(W_{\mathrm{sph}})$ the result follows from Theorem~\ref{thm:conesgated}. So, consider a part $P=\{w\in W\mid \Phi_{\mathrm{sph}}(w)=\Sigma\}$ of $\scrJ$, for some $\Sigma\in \mathbb{S}$. Let $x,y\in P$, and let $z\in W_{\mathrm{sph}}$. It is sufficient to show, by symmetry of $x$ and $y$, that if $x\in C(z)$ then $y\in C(z)$ too. First we note that since $\Phi_{\mathrm{sph}}(x)=\Phi_{\mathrm{sph}}(y)$ we have $\Phi_J(x)=\Phi_J(y)$ for all spherical subsets $J\subseteq S$. 

Suppose that $x\in C(z)$ with $z\in W_{\mathrm{sph}}$. Choose a reduced expression of $z$, and let $J$ be the set of generators appearing in this reduced expression (thus $J$ is spherical, and note that $J$ does not depend on the particular reduced expression chosen, by \cite[Proposition~2.16]{AB:08}). Since $z\peq x$ we have $\Phi(z)\subseteq \Phi_J(x)=\Phi_J(y)\subseteq \Phi(y)$. Thus $z\peq y$, and so $y\in C(z)$.
\end{proof}

\section{Ultra low elements}\label{sec:ultralow}

We now define a new class of elements of $W$ called \textit{ultra low} elements, which we denote by $U$. Conjecturally, $U$ is the set of gates $\Gamma$ of $W$, which in turn is conjecturally the smallest Garside shadow $\Gmin$ (see the comments below).

\begin{defn} \label{def:ultra_low_elements}
An element $x \in W$ is \textit{ultra low} if for each $\beta \in \Phi^1(x)$ there exists $w\in W$ such that $\Phi(x) \cap \Phi(w) = \{ \beta \}$.
\end{defn}

For example, each $s\in S$ is ultra low, and $e$ is trivially ultra low.

\begin{prop}
We have $U\subseteq \Gamma\subseteq L$.
\end{prop}

\begin{proof}
The containment $U\subseteq \Gamma$ follows from the definition of ultra low elements, and the characterisation of gates in Theorem~\ref{thm:characterisegates}, noting that $\Phi^0(x)\subseteq\Phi^1(x)$. The containment $\Gamma\subseteq L$ is Proposition~\ref{prop:gatesbasicfacts}.
\end{proof}

\begin{rem}
One can show more directly that $U\subseteq L$ without passing through~$\Gamma$. For if $x\in U$ then from the definition and Lemma~\ref{lem:EandPhi1} we have $\Phi^1(x)\subseteq \cE$, and hence $x$ is low.
\end{rem}

The following proposition connects the concept of ultra low elements with boundary roots. 

\begin{prop}\label{prop:g}
Let $x\in W$ and let $T=T(x^{-1})$. Then $x$ is ultra low if and only if $\Phi^1(x) = \partial T$. Moreover, if $\Phi^1(x)=\partial T$ then $x$ is the gate of $X_T$. 
\end{prop}

\begin{proof}
Note that $\partial T\subseteq\Phi^1(x)$ by Corollary~\ref{cor:boundaryroots}. We have $x\in U$ if and only if for each $\beta\in\Phi^1(x)$ there exists $w\in W$ with $\Phi(x)\cap \Phi(w)=\{\beta\}$, if and only if $\beta\in\partial T$ by Theorem~\ref{thm:boundaryroots}.

Let $y \in X_T$, and so $T(y^{-1}) = T$. By Corollary~\ref{cor:boundaryroots} we have $\partial T \subseteq \Phi^1(y)$. Since $\Phi^1(x) = \partial T$, we have $\Phi(x) = \mathrm{cone}_{\Phi}(\partial T) \subseteq \mathrm{cone}_{\Phi}(\Phi^1(y))=\Phi(y)$ (see Theorem~\ref{thm:eqcond}). Thus by Proposition~\ref{prop:rootsystembasics}(5) we have $x\peq y$ for all $y\in X_T$. Thus $x=g_T$ is the gate of~$X_T$.
\end{proof}

We conjecture that the reverse implication also holds in the second statement of Proposition~\ref{prop:g} (this is the content of Conjecture~\ref{conj:boundary} in the introduction). If this conjecture holds then it follows that $U=\Gamma$. 

\begin{rem}
We have verified that $U=\Gamma$ for right angled Coxeter groups, rank~$3$ Coxeter groups, and all Coxeter groups with complete Coxeter graph (that is, $m_{st}>2$ for all $s\neq t$). See \cite[Chapter~5]{YY:21} for details. 
\end{rem}

\section{Super elementary roots}\label{sec:superelementary}

In Section~\ref{sec:boundaryroots} we observed that the boundary roots $\partial T\subseteq\cE$ of a cone type $T$ gives the minimal amount of root data required to determine $T$. It is natural to ask whether every elementary root is a boundary root of some cone type. Equivalently, is it true that $\Phi(\scrT)=\cE$ (in the notation of Section~\ref{sec:simplerefinements})? In this section we show that $\Phi(\scrT)=\cE$ for spherical and affine Coxeter groups (amongst others), however in general the containment $\Phi(\scrT)\subseteq \cE$ can be strict. In particular, we exhibit a class of rank $4$ Coxeter groups for which there is an elementary root that is not the boundary root of any cone type. We thank Bob Howlett for inspiring the work in this section and for suggesting a motivating example.

\begin{defn} \label{defn:superelementary}
A root $\beta\in\Phi^+$ is \textit{super-elementary} if there exists $x,y \in W$ with $$\Phi(x) \cap \Phi(y) = \{ \beta \}.$$
Let $\cS$ denote the set of all super-elementary roots. 
\end{defn}

By Theorem~\ref{thm:boundaryroots} $\cS$ is precisely the set of roots that occur as the boundary root of some cone type, and hence $\cS=\Phi(\scrT)$. 

\begin{prop}\label{prop:superareelementary}
Every super-elementary root is elementary. Thus $\cS\subseteq\cE$. 
\end{prop}

\begin{proof}
See Lemma~\ref{lem:EandPhi1}.
\end{proof}

We now provide various classes of Coxeter systems for which $\cS=\cE$. 

\begin{thm} \label{thm:sphericalsuperelementary}
If $W$ is spherical then $\cS=\cE=\Phi^+$.
\end{thm}

\begin{proof}
Let $w_0=s_1\cdots s_n$ be a reduced expression for the longest element of $W$. Since $\Phi(w_0)=\Phi^+$ we have $\Phi^+=\{\beta_1,\ldots,\beta_n\}$ where $\beta_j=s_1\cdots s_{j-1}\alpha_{s_j}$ for $1\leq j\leq n$. We claim that
$$
\Phi(s_1\cdots s_j)\cap \Phi(s_1\cdots s_{j-1}w_0)=\{\beta_j\}\quad\text{for $j=1,\ldots,n$}.
$$
For if $w\in W$ then $\Phi(ww_0)=\Phi^+\backslash\Phi(w)$ (because if $\alpha\in\Phi^+\backslash \Phi(w)$ then $(ww_0)^{-1}\alpha=w_0w^{-1}\alpha<0$ as $w^{-1}\alpha>0$, and $\ell(ww_0)=\ell(w_0)-\ell(w)=|\Phi^+\backslash\Phi(w)|$). Thus $\Phi(s_1\cdots s_{j-1}w_0)=\{\beta_j,\ldots,\beta_n\}$, and the claim follows as $\Phi(s_1\cdots s_j)=\{\beta_1,\ldots,\beta_j\}$. Thus every positive root is super elementary, and hence the theorem.
%
%
%
\end{proof}

\begin{cor} \label{cor:sphericalsuperelementary}
We have $\Phisph^+\subseteq\cS$. In particular, if $\cE=\Phisph^+$ then $\cS=\cE$. 
\end{cor}

\begin{proof}
If $\beta\in\Phi_J^+$ with $J\subseteq S$ spherical, then by Theorem~\ref{thm:sphericalsuperelementary} there exists $x,w\in W_J$ with $\Phi(x)\cap\Phi(w)=\{\beta\}$, and so $\beta\in\cS$.
\end{proof}

In particular, Corollary~\ref{cor:sphericalsuperelementary} shows that if $W$ is right angled, or if $W$ has complete Coxeter graph, then $\cS=\cE$ (see \cite[Theorem~1]{PY:19} for the classification of Coxeter systems with $\cE=\Phisph^+$).

We will now show that $\cS=\cE$ for all affine Coxeter groups. When $W$ is affine there is a related notion of a ``root system'', for $W$ where one starts with a crystallographic root system $\Phi_0$ of a spherical Coxeter group~$W_0$. We will not repeat this construction here, we refer to \cite[Section~2.2]{PY:19} for details, and we will use the notation of \cite{PY:19} in the following paragraphs. In particular, the finite crystallographic root system $\Phi_0$ has simple system $\{\alpha_1,\ldots,\alpha_n\}$, and $\langle\cdot,\cdot\rangle$ denotes the bilinear form on the underlying vector space. For $\alpha\in\Phi_0$ we write $\alpha^{\vee}=2\alpha/\langle\alpha,\alpha\rangle$. Let $\omega_1,\ldots,\omega_n$ be the basis dual to the simple roots basis, and so $\langle\alpha_i,\omega_j\rangle=1$ if $i=j$ and equals $0$ otherwise. Let $P=\mathbb{Z}\omega_1+\cdots+\mathbb{Z}\omega_n$ (the \textit{coweight lattice} of $\Phi_0$). The affine root system is $\Phi=\Phi_0+\mathbb{Z}\delta$ with positive roots $\Phi^+=(\Phi_0^++\mathbb{Z}_{\geq 0}\delta)\cup(-\Phi_0^++\mathbb{Z}_{>0}\delta)$. The set of elementary roots is $\cE=\Phi_0^+\cup(-\Phi_0^++\delta)$. 

\begin{lem}\label{lem:affine1}
Let $K>0$ be an integer and $\alpha\in\Phi_0^+$. There exists $\lambda\in P$ such that $\langle\lambda,\alpha\rangle=1$ and $|\langle\lambda,\beta\rangle|>K$ for all $\beta\in\Phi_0^+\backslash\{\alpha\}$.
\end{lem}

\begin{proof}
Let $w\in W_0$ be such that $w\alpha=\alpha_i$ (a simple root). Let 
$$
\lambda'=\omega_i+(K+1)\sum_{j\neq i}\omega_j.
$$
Then $\langle\lambda',\alpha_i\rangle=1$ and $\langle\lambda',\beta\rangle>K$ for all $\beta\in\Phi_0^+\backslash\{\alpha_i\}$. Now let $\lambda=w^{-1}\lambda'$. So $\langle\lambda,\alpha\rangle=\langle\lambda,w^{-1}\alpha_i\rangle=\langle w\lambda,\alpha_i\rangle=1$, and if $\beta\in\Phi_0^+\backslash\{\alpha\}$ we have $\langle\lambda,\beta\rangle=\langle\lambda',w\beta\rangle$. Since $w\beta\in\Phi_0\backslash\{-\alpha,\alpha\}$ there is an index $j\neq i$ and an element $\gamma$ in the $\mathbb{Z}_{\geq 0}$-span of the simple roots, such that either $w\beta=-\alpha_j-\gamma$ (in the case that $w\beta\in-\Phi_0^+$) or $w\beta=\alpha_j+\gamma$ (in the case that $w\beta\in\Phi_0^+$). Then $|\langle\lambda',w\beta\rangle|>K$, and so $|\langle\lambda,\beta\rangle|>K$.
\end{proof}

\begin{thm}\label{thm:affineSE}
If $W$ is affine then $\cS=\cE$.
\end{thm}

\begin{proof}
The roots in $\Phi_0^+$ are super elementary Corollary~\ref{cor:sphericalsuperelementary}. So consider the roots $-\alpha+\delta$ with $\alpha\in\Phi_0^+$. Let 
$$
K=\max_{\alpha,\beta\in\Phi_0^+}|\langle\alpha^{\vee},\beta\rangle|
$$
(this is an integer, as $\Phi_0$ is crystallographic). Using Lemma~\ref{lem:affine1}, choose $\lambda\in P$ with $\langle\lambda,\alpha\rangle=1$ and $|\langle\lambda,\beta\rangle|>K$ for all $\beta\in\Phi_0^+\backslash\{\alpha\}$. Let $\mu=-\lambda+\alpha^{\vee}$. Then $\langle\mu,\alpha\rangle=-\langle\lambda,\alpha\rangle+\langle\alpha^{\vee},\alpha\rangle=1$, and $\langle\mu,\beta\rangle=-\langle\lambda,\beta\rangle+\langle\alpha^{\vee},\beta\rangle$ for $\beta\in\Phi_0^+\backslash\{\alpha\}$. It follows, from the choice of $K$, that $\langle\mu,\beta\rangle$ and $\langle\lambda,\beta\rangle$ have opposite signs for all $\beta\in\Phi_0^+\backslash\{\alpha\}$. Let $w_{\lambda}$ and $w_{\mu}$ be the (unique) elements of $W$ such that $w_{\lambda}(C_0)=\lambda+C_0$ and $w_{\mu}(C_0)=\mu+C_0$ (as sets, where $C_0$ is the fundamental chamber in the standard geometric realisation). Therefore every root in $\Phi(w_{\lambda})\cap\Phi(w_{\mu})$ lies in the parallelism class of $\alpha$, and since $\langle\lambda,\alpha\rangle=\langle\mu,\alpha\rangle=1$ we have $\Phi(w_{\lambda})\cap\Phi(w_{\mu})=\{-\alpha+\delta\}$. Thus the root $-\alpha+\delta$ is super elementary. 
\end{proof}

We have now concluded our discussion of affine Coxeter groups, and thus we return to the ``standard'' notion of root systems henceforth. We turn our attention to exhibiting a class of Coxeter systems for which $\cS$ is a strict subset of $\cE$. Let $(W,S)$ be a rank $4$ Coxeter system with $S=\{s_1,s_2,s_3,s_4\}$ and $m_{s_1,s_2}=m_{s_1,s_3}=m_{s_2,s_3}=2$. Let $m_i = m_{s_i,s_4}$ and $t_i=2\cos(\pi/m_i)$ for $i=1,2,3$. Thus $W$ has Coxeter graph

    \begin{center}
    \begin{tikzpicture}[scale=1]
    \draw
    (2,-3) node[fill=black,circle,scale=0.6] (0) {}
    (3,-3) node[fill=black,circle,scale=0.6] (1)
    {}
    (4,-2.5) node[fill=black,circle,scale=0.6] (2)
    {}
    (4,-3.5) node[fill=black,circle,scale=0.6] (3)
    {};
    \draw (0)--node [midway,above]{$m_{1}$}(1);
    \draw (1)--node [midway,above]{$m_{2}$}(2);
    \draw (1)--node [midway,below]{$m_{3}$}(3);
\end{tikzpicture}
\end{center}

Write $\alpha_i=\alpha_{s_i}$ for $i=1,2,3,4$. If $\lambda=a\alpha_1+b\alpha_2+c\alpha_3+d\alpha_4\in V$ with $a,b,c,d\geq 0$ and $\lambda\neq 0$ we write $\lambda>0$ (note that this notation does not imply that $\lambda$ is a root).

\newpage

\begin{lem}\label{lem:SElemma}
Let $\lambda=a\alpha_1+b\alpha_2+c\alpha_3+d\alpha_4\neq 0$ with $a,b,c,d\geq 0$, $t_1d-2a\geq 0$, $t_2d-2b\geq 0$, $t_3d-2c\geq 0$, and $t_1a+t_2b+t_3c-2d\geq 0$. Then $w\lambda> 0$ for all $w\in W$.
\end{lem}

\begin{proof}
Induction on $\ell(w)$, with the case $w=e$ true by hypothesis (as $a,b,c,d\geq 0$). Suppose the result is true for $w$, and suppose that $\ell(ws_i)=\ell(w)+1$. Thus $w\alpha_i>0$.

If $i=1$ we have $s_1\lambda=\lambda+(t_1d-2a)\alpha_1$, and hence
$$
ws_i\lambda=w(\lambda+(t_1d-2a)\alpha_1)=w\lambda+(t_1d-2a)w\alpha_1> 0
$$
as $t_1d-2a\geq 0$ and $w\alpha_1>0$. Similarly for $i=2,3$. If $i=4$ then
$$
ws_i\lambda=w(\lambda+(t_1a+t_2b+t_3c-2d)\alpha_4)> 0,
$$
hence the result.
\end{proof}

\begin{thm}\label{thm:SEtheorem}
Let $(W,S)$ be the rank $4$ Coxeter system as above, and suppose that $\frac{1}{m_i}+\frac{1}{m_j}\leq \frac{1}{2}$ for each pair $i,j\in\{1,2,3\}$ with $i\neq j$, and that $m_1,m_2,m_3<\infty$. Then the root $$\beta=t_1\alpha_1+t_2\alpha_2+t_3\alpha_3+\alpha_4$$ is elementary but not super elementary.
\end{thm}

\begin{proof}
Since $m_{s_1,s_2}=m_{s_1,s_3}=m_{s_2,s_3}=2$ we have $\beta =s_3s_2s_1(\alpha_4)$. Thus $\alpha_4\mapsto_{s_1}t_1\alpha_1+\alpha_4\mapsto_{s_2}t_1\alpha_1+t_2\alpha_2+\alpha_4\mapsto_{s_3}\beta$ is a path moving up the root poset of $\Phi^{+}$ such that the difference in the $\alpha_i$ coordinate at the $i$th step in the path is less than $2$ (as $m_1,m_2,m_3<\infty$). Thus $\beta$ is elementary by~\cite[\S4.7]{BB:05}.

Let $\lambda=\beta+\alpha_4$ and $\lambda_i=s_i\beta+\alpha_4$ for $i=1,2,3$. We claim that $\lambda$ and $\lambda_i$ satisfy the conditions of Lemma~\ref{lem:SElemma}. To see this, note that the condition $\frac{1}{m_i}+\frac{1}{m_j}\leq \frac{1}{2}$ is equivalent to $t_i^2+t_j^2\geq 4$ for $i,j\in\{1,2,3\}$ with $i\neq j$. We have $\lambda=t_1\alpha_1+t_2\alpha_2+t_3\alpha_3+2\alpha_4$, $\lambda_1=0\alpha_1+t_2\alpha_2+t_3\alpha_3+2\alpha_4$, $\lambda_2=t_1\alpha_1+0\alpha_2+t_3\alpha_3+2\alpha_4$, and $\lambda_3=t_1\alpha_1+t_2\alpha_2+0\alpha_3+2\alpha_4$. Then for $\lambda$ we have $t_1d-2a=t_2d-2b=t_3d-2c=0$ and $t_1a+t_2b+t_3c-2d=t_1^2+t_2^2+t_3^2-4\geq 0$, and for $\lambda_1$ we have $t_1d-2a=2t_1$, $t_2d-2b=t_3d-2c=0$, and $t_1a+t_2b+t_3c-2d=t_2^2+t_3^2-4\geq 0$. Similarly for $\lambda_2$ and $\lambda_3$, using $t_1^2+t_3^2\geq 4$ and $t_1^2+t_2^2\geq 4$.

Thus by Lemma~\ref{lem:SElemma} we have $w^{-1}\lambda> 0$ and $w^{-1}\lambda_i> 0$ ($i=1,2,3$) for all $w\in W$. It follows that there exists no element $w \in W$ with $\{ \alpha_4, \beta \} \subseteq \Phi(w)$ or $\{ \alpha_4, s_i\beta \} \subseteq \Phi(w)$ (for $i = 1,2,3)$, because if so then either $w^{-1}\lambda < 0$ or $w^{-1}\lambda_i < 0$.

We will now show that if $\beta \in \Phi(w)$ then $| \Phi(w) \cap \{ \alpha_1, \alpha_2, \alpha_3 \} | \geq 2 $ (and thus there cannot be $w, v \in W$ with $\Phi(w) \cap \Phi(v) = \{ \beta \}$, and so $\beta\notin\mathcal{S}$). Equivalently, we will show that at least two of the simple reflections $s_1,s_2,s_3$ are in $D_L(w)$. If $\beta\in\Phi(w)$ then the above observation gives $\alpha_4\notin\Phi(w)$, and so $s_4\notin D_L(w)$. Therefore at least one of $s_1,s_2,s_3$ lie in $D_L(w)$ (as $w\neq e$ since $\beta\in \Phi(w)$). Thus $w=s_iv$ with $\ell(s_iv)=\ell(v)+1$ for some $i\in\{1,2,3\}$ and $v\in W$, and note that $v\neq e$ (as $\beta\in\Phi(w)$ and $\beta\neq\alpha_i$). If $s_j\in D_L(v)$ for some $j\in\{1,2,3\}$ then we are done (as $i\neq j$, and $s_is_j=s_js_i$ giving $s_i,s_j\in D_L(w)$). So assume that $D_L(v)=\{s_4\}$. Then $\alpha_4\in\Phi(v)$, and so from the above observations $s_i\beta\notin\Phi(v)$. But then $w^{-1}\beta=v^{-1}s_i\beta>0$, a contradiction, completing the proof.
\end{proof}

\begin{cor}
Let $W$ and $\beta$ be as in Theorem~\ref{thm:SEtheorem}. Then $\beta$ is not a boundary root of any cone type. In particular, $\Phi(\scrT)$ is a strict subset of $\cE$ in this case. 
\end{cor}

%
%
%
%
%
%

\bibliographystyle{plain}


\end{document}